\documentclass[12pt]{amsart}
\usepackage{fullpage,amssymb}

\title[Relations on elliptic curves]{Relations among modular points on elliptic curves}

\author{Alexandru Buium}
\address{University of New Mexico \\ Albuquerque, NM 87131}
\email{buium@math.unm.edu} \urladdr{http://math.unm.edu/\~{}buium}

\author{Bjorn Poonen}
\address{Department of Mathematics, University of California,
    Berkeley, CA 94720-3840, USA}
\email{poonen@math.berkeley.edu}
\urladdr{http://math.berkeley.edu/\~{}poonen}

\date{June 4, 2007}

\def \rec{\text{reciprocity}}
\def \itrec{\textit{reciprocity}}
\def \Rec{\text{Reciprocity}}
\def \thecorr{S  \stackrel{\Pi}{\longleftarrow}  X
\stackrel{\Phi}{\longrightarrow} A}
\def \pdiv{_{\text{$p$-div}}}
\def \dug{\dagger}
\def \ddug{\ddagger}
\def \cA{{\mathcal A^{\dug}}}
\def \cB{{\mathcal A^{\ddug}}}
\def \cP{{\mathcal P}}
\def \ba{{\bf a}}
\def \Sigmat{\Phi^{\dug}}
\def \Sigmatt{\Phi^{\ddug}}

\def \varphic{f^{\flat}}

\def \ZN{\bZ[1/N,\zeta_N]}

\def \<{\langle}
\def \>{\rangle}

\def \cU{\mathcal U}
\def \cF{\mathcal F}

\def \cC{\mathcal C}

\def \cK{\mathcal K}

\def \cP{\mathcal P}

\def \cF{\mathcal F}
\def \h{\hat{\ }}

\def \d{\delta}

\def \bZ{{\bf Z}}

\def \bF{{\bf F}}

\def \bC{{\bf C}}

\def \bA{{\bf A}}

\def \cO{\mathcal O}
\def \ra{\rightarrow}

\def \bX{{\bf X}}

\def \bF{{\bf F}}

\def \bQ{{\bf Q}}

\def \bQ{{\bf Q}}

\def \bC{{\bf C}}

\usepackage{color}

\newcommand{\C}{{\mathbf C}}
\newcommand{\F}{{\mathbf F}}
\newcommand{\Fbar}{{\overline{\F}}}
\newcommand{\PP}{{\mathbf P}}
\newcommand{\Q}{{\mathbf Q}}
\newcommand{\qq}{{\mathfrak{q}}}
\newcommand{\R}{{\mathbf R}}
\newcommand{\Z}{{\mathbf Z}}
\newcommand{\kbar}{{\overline{k}}}
\newcommand{\Qbar}{{\overline{\Q}}}

\newcommand{\injects}{\hookrightarrow}
\newcommand{\isom}{\simeq}

\newcommand{\Intersection}{\bigcap} % intersection of a collection
\newcommand{\intersect}{\cap} % binary intersection
\newcommand{\union}{\cup} % binary union
\newcommand{\tensor}{\otimes}
\newcommand{\calA}{{\mathcal A}}

\newcommand{\calC}{{\mathcal C}}

\newcommand{\calE}{{\mathcal E}}

\newcommand{\calK}{{\mathcal K}}

\newcommand{\calM}{{\mathcal M}}

\newcommand{\calP}{{\mathcal P}}

\newcommand{\calS}{{\mathcal S}}

\newcommand{\calU}{{\mathcal U}}

\newcommand{\OO}{{\mathcal O}}

\DeclareMathOperator{\Frob}{Frob}
\DeclareMathOperator{\crys}{crys}
\DeclareMathOperator{\Fr}{Fr}
\DeclareMathOperator{\Char}{char}
\DeclareMathOperator{\Tr}{Tr}
\DeclareMathOperator{\tr}{tr}
\DeclareMathOperator{\Id}{Id}
\DeclareMathOperator{\CL}{CL}
\DeclareMathOperator{\CM}{CM}
\DeclareMathOperator{\Div}{Div}
\DeclareMathOperator{\End}{End}
\DeclareMathOperator{\id}{id}
\DeclareMathOperator{\im}{Im}
\DeclareMathOperator{\Gal}{Gal}
\DeclareMathOperator{\Hom}{Hom}
\DeclareMathOperator{\rank}{rank}
\DeclareMathOperator{\Spec}{Spec}
\DeclareMathOperator{\Spf}{Spf}

\DeclareMathOperator{\disc}{disc}
\DeclareMathOperator{\GL}{GL}

\newcommand{\et}{{\operatorname{et}}}
\newcommand{\ram}{{\operatorname{ram}}}
\newcommand{\ur}{{\operatorname{ur}}}
\newcommand{\tors}{{\operatorname{tors}}}
\newcommand{\ord}{{\operatorname{ord}}}

\numberwithin{equation}{section}

\newtheorem{theorem}[equation]{Theorem}
\newtheorem{corollary}[equation]{Corollary}
\newtheorem{lemma}[equation]{Lemma}

\theoremstyle{definition}
\newtheorem{definition}[equation]{Definition}
\theoremstyle{remark}
\newtheorem{remark}[equation]{Remark}

\begin{document}


\begin{abstract}
Given a correspondence between a modular curve and an elliptic curve
$A$  we study the group of relations among the $\CM$ points of $A$.
In particular we prove that the intersection of any finite rank
subgroup of $A$ with the set of CM  points of $A$ is finite. We also
prove a local version of this global result with an effective bound
valid also for certain infinite rank subgroups. We deduce the local
result from a ``reciprocity'' theorem for  $\CL$ (canonical lift)
points on $A$. Furthermore we prove similar global and local results
for intersections between subgroups of $A$ and isogeny classes in $A$.  
Finally  we prove Shimura curve analogues and, in some cases, 
higher-dimensional versions of these results.
\end{abstract}

\maketitle

\section{Introduction}

Modular curves possess various remarkable sets of points 
having a modular interpretation. Typical examples are
the set of CM points and the set of points in a given 
isogeny class. Given a correspondence between a modular curve and
an elliptic curve $A$, we study the integer relations
among the points in the image in $A$ of such a
set. We consider also the analogous problem 
in which modular curves are replaced by Shimura curves.

\subsection{Modular curves}

Let $X_1(N)$ over $\Qbar$ be the complete modular curve attached to 
the group $\Gamma_1(N)$ for some $N>3$.
If $Y_1(N) \subset X_1(N)$ is the non-cuspidal locus
then $Y_1(N)(\Qbar)$ is in bijection with the set of isomorphism
classes of pairs $(E,\alpha)$
where $E$ is an elliptic curve over $\Qbar$
and $\alpha\colon \Z/N\Z \injects E(\Qbar)$ is an injection.

A CM-{\em point} on $S$ is a point in $Y_1(N)(\Qbar)$ represented
by an elliptic curve $E$ (with point) such that 
$E$ has complex multiplication, i.e., $\End(E) \neq \bZ$.
Let $\CM \subset S(\Qbar)$ be the set of CM-points on $S$.

\begin{definition}
\label{defmodell}
A {\it modular-elliptic correspondence} is a pair of
non-constant morphisms of smooth connected projective curves over
$\Qbar$, $\thecorr$, where $S=X_1(N)$, $A$ is an elliptic curve, 
and $X$ is equipped with a point $x_{\infty}$ such that
$\Pi(x_{\infty})=\infty$ and $\Phi(x_{\infty})=0$.
Call $\Phi(\Pi^{-1}(\CM)) \subset A(\Qbar)$ the
set of CM-{\em points} on $A$.
\end{definition}

A special case of the above situation arises from the
Eichler-Shimura construction.  
For terminology on modular forms we refer to~\cite{DI}.

\begin{definition}
\label{willes} Let $f=\sum a_n q^n$ be a newform (which we will
usually assume, without notice, to be of weight $2$, on
$\Gamma_0(N)$, normalized, i.e.\ $a_1=1$, and with rational Fourier
coefficients, so $a_n \in \bZ$). 
The Eichler-Shimura construction~\cite{DI} yields a $\Q$-morphism
from $X_1(N)$ (through $X_0(N)$) to an elliptic curve $A_f$.
By a {\em modular parametrization attached to $f$}
we mean a composition $X_1(N) \to A_f \to A$
where $A_f \to A$ is an isogeny of elliptic curves over $\Q$.
A modular-elliptic correspondence is said to
{\em arise from a modular parametrization}
if it is of the form $\thecorr$ where $S=X=X_1(N)$, $\Pi=\Id$, and $\Phi$
is a modular parametrization.
\end{definition}

By work of Wiles and others~\cite{wiles,TW,four},
together with the Isogeny Theorem of Faltings~\cite{faltings},
any elliptic curve $A$ over $\bQ$ has a modular parametrization.

Going back to the case of an arbitrary modular-elliptic
correspondence $\thecorr$, 
we study the linear dependence relations among the CM-points on $A$. 
More precisely, given a Zariski open set
$X^{\dug} \subset X$ define a {\em $\CM$ divisor} on $X^{\dug}$ to
be a divisor on $X$ supported in $\Pi^{-1}(\CM)\cap X^{\dug}(\Qbar)$.
Then we would like to understand the structure of the group
$\Phi^{\perp}$ of all $\CM$ divisors $\sum m_iP_i$ on $X^{\dug}$
such that $\sum m_i \Phi(P_i)$ belongs to the torsion subgroup
$A(\Qbar)_{\tors}$ of $A(\Qbar)$.
The group $\Phi^{\perp}$  can be quite ``large''; for instance,
if the correspondence arises from a
modular parametrization, the group $\Phi^{\perp}$
contains an infinite collection of divisors which we shall call
{\it Hecke divisors}. (This will imply, by the way, that
$\Phi^{\perp}$ has infinite rank.  By the {\em rank} of an abelian
group $\Gamma$ we mean $\dim_{\Q} (\Gamma \tensor \Q)$.)
On the other hand, $\Phi^{\perp}$ is not  ``too large'': 
the intersection of $\Phi(\Pi^{-1}(\CM))$ with any
finite rank subgroup is finite (Corollary~\ref{C:global E}).
In fact, this result can be generalized by allowing $A$
to be an abelian variety of arbitrary dimension,
and by replacing $\Gamma$ by an $\epsilon$-fattening in the style
of the ``Mordell-Lang plus Bogomolov'' statement 
introduced in~\cite{Poonen1999}.

There is also a local analogue in which $\Qbar$ is
replaced by the completion $R$ of the maximal unramified
extension of the ring $\Z_p$ of $p$-adic integers. 
In this introduction the local case will be discussed at an informal level:
see Section~\ref{S:local statements} for precise definitions.
Our proof of this local analogue requires us to
replace the set $\CM$ by the set $\CL$ of canonical lift points,
but on the other hand we obtain effective bounds and the results
are valid for a larger class of groups
(not necessarily of  finite rank).  
We deduce this local finiteness result 
from a {\it reciprocity theorem} for $\CL$
points. 
Our reciprocity theorem asserts the existence of
a $p$-adic formal function $\Sigmat$ on an open set $X^{\dug}$
of $X$ over $R$ with the  property that for any $\CL$ divisor
$\sum m_iP_i$ on $X^{\dug}$ we have
  $\sum m_i \Phi(P_i)\in A(R)_{\tors}$
  if and only if $\sum m_i \Sigmat(P_i)=0$. There is also a ``mod $p$''
analogue saying that $\sum m_i \Phi(P_i)\in A(R)_{\tors}+pA(R)$
if and only if $\sum m_i \overline{\Sigmat}(\overline{P}_i)=0$;
here the bars mean reduction modulo $p$.
 We informally refer to $\Sigmat$ as a $\itrec$
 {\it function for $\CL$ points},
and to $\overline{\Sigmat}$ as 
a $\itrec$ {\it function mod $p$ for $\CL$ points}.

Going back to the global case, one may ask if
$\rec$ functions (respectively, $\rec$ functions mod $p$) exist for
$\CM$ points.
The answer is {\it no} (respectively, {\it yes}), at least if the correspondence  arises from a modular
parametrization. 
Under this hypothesis we prove that the
$\rec$ function mod $p$ is ``essentially'' unique and we will
explicitly compute this function in terms of a certain
remarkable modular form mod $p$ naturally attached to $f$. 

We obtain also reciprocity functions and theorems
(generally weaker)
for the analogous situation in which the set of CM points is replaced
by an isogeny class, defined as follows:

\begin{definition}
 For any non-cusp $Q \in S(\Qbar)$, represented by an
elliptic curve $E$ with a point of order $N$, the {\it isogeny
class} $C$ of $Q$ in $S(\Qbar)$ consists of all points in $S(\Qbar)$
represented by elliptic curves $E'$ with a point of order $N$ such
that there is an isogeny $E \to E'$ (not required to respect
the points of order $N$).
If $\Sigma$ is a set of primes,
define the $\Sigma$-{\it isogeny class} of $Q$ in $S(\Qbar)$
as the analogous set obtained when we allow 
only isogenies having degrees all of
whose prime divisors are in $\Sigma$.
\end{definition}

\subsection{Shimura curves}

Let $D$ be a non-split indefinite quaternion algebra over $\bQ$.
We fix a maximal order $\cO_D$ once and for all.
Let $X^D(\cU)$ be the Shimura curve attached to the pair $(D,\cU)$,
where $\cU$ is a sufficiently small
compact subgroup of $(\cO_D \otimes (\varprojlim \bZ/m\bZ))^{\times}$ 
such that $X^D(\cU)$ is connected: see~\cite{buzzard,ZZ}.

A {\it false elliptic curve} is a pair $(E,i)$
consisting of an abelian surface $E$ over $\Qbar$ and an
embedding $i\colon \cO_D \to \End(E)$.
The set $X^D(\cU)(\Qbar)$
is in bijection with the set of isomorphism classes of
false elliptic curves equipped with a level $\cU$
structure in the sense of~\cite{buzzard,ZZ}.

The classification of endomorphism algebras~\cite[p.~202]{MumfordAV}
shows that for any false elliptic curve $(E,i)$, 
the algebra $(\End E) \tensor \Q$ is isomorphic to either $D$
or $D \tensor \calK \isom M_2(\calK)$
for some imaginary quadratic field $\calK$ embeddable in $D$.
In the latter case, $(E,i)$ is called {\em CM};
then $E$ is isogenous to the square of an elliptic curve
with CM by an order in $\calK$.
A {\em CM-point} of $S(\Qbar)$ is a point whose associated $(E,i)$ is CM.
Let $\CM \subset S(\Qbar)$ be the set of CM-points on $S$.

\begin{definition}
\label{defshell} 
A {\it Shimura-elliptic correspondence} is a pair
of non-constant morphisms of smooth connected projective curves
over $\Qbar$, $\thecorr$, where 
$S$ is a Shimura curve $X^D(\cU)$ as above
and $A$ is an elliptic curve.
Call $\Phi(\Pi^{-1}(\CM)) \subset A(\Qbar)$ the
set of CM-\em{points} on $A$.
\end{definition}

We may ask again for $\rec$ functions (respectively, 
$\rec$ functions mod $p$).
In particular, one can ask if
$\Phi(\Pi^{-1}(\CM)) \cap \Gamma$ is finite (or even
effectively bounded) for any finite rank subgroup $\Gamma \leq
A(\Qbar)$.

Again one can replace CM by various isogeny classes,
defined as follows:

\begin{definition}
For any $Q \in S(\Qbar)$, represented by a false
elliptic curve $(E,i)$ with level $\cU$-structure, 
the {\it isogeny class} $C$ of $Q$ in $S(\Qbar)$ consists of all points in
$S(\Qbar)$ represented by false elliptic curves $(E',i')$ with level
$\cU$-structure such that there is an isogeny $E \to E'$
compatible with the $\cO_D$-action (but not necessarily
compatible with the level $\cU$-structures). If $\Sigma$ is a set
of primes and we insist that the above isogenies have degrees all
of whose prime divisors are in $\Sigma$, then the smaller set $C$ 
obtained will be referred to 
as the $\Sigma$-{\it isogeny class} of $Q$ in $S(\Qbar)$.
\end{definition}

We will prove results for both the global and
the local cases of these questions on Shimura-elliptic
correspondences; the results are similar to (but sometimes weaker
than) the ones for modular-elliptic correspondences.

\subsection{Previous work}
\label{S:previous}

Let us make some comments on previous work on
problems related to those addressed in this paper. Most of this
previous work concerned Heegner points, which are certain special
points in $\Phi(\CM)$ where $\Phi:X_1(N) \to A$ is a modular
parametrization.
 The study  of the  linear dependence relations among Heegner points
 (and their traces) plays an important role in
the work on the Birch and Swinnerton-Dyer conjecture, especially
in the breakthroughs by Gross-Zagier~\cite{GZ} 
and Kolyvagin~\cite{kolivagin}. 
See~\cite{darmon} for an  exposition of this circle of ideas.
See also \cite{turk,vatsal,cornut} for more recent advances,
especially in relation to Mazur's conjectures in~\cite{mazuricm}.
In particular it was proved in~\cite{turk} that
there are only finitely many torsion
Heegner points on any elliptic curve over $\bQ$. Along slightly
different lines it was recently proved in~\cite{RS} that if
$Q_1,\ldots,Q_s$ are Heegner points associated to distinct quadratic
imaginary fields and if the odd parts of the class numbers of
these fields are sufficiently large then $Q_1,\ldots,Q_s$ are linearly
independent.  Finally recall that, by the classical theory
of complex multiplication,
the set of all points in CM defined over a given number field
 is finite; this, plus the Hermite-Minkowski theorem,
 implies that $\Phi(\CM) \cap \Gamma$ is finite
 for any finitely generated $\Gamma \leq A(\overline{\bQ})$.
  In contrast to this note that in our
finiteness results  $\Gamma$ is allowed to have  points of
unbounded degree (which, by Northcott's theorem, is always the
case if $\Gamma$ contains an infinite set of bounded height, for
instance an infinite set of torsion points).

\subsection{Structure of the paper}

The statements of our
main results are given in 
Sections \ref{S:global statements} and~\ref{S:local statements}.

We use methods quite different from those used in the papers
mentioned in Section~\ref{S:previous}.
Our global results on the finiteness
of $\CM$ points lying in finite rank subgroups will be proved using
equidistribution results for Galois orbits in abelian varieties,
modular curves, and Shimura curves, in Section~\ref{S:equidistribution}.

Our local results will be proved using the theory of
``arithmetic differential equations'' in the sense of~\cite{book}.
These proofs will be given in Section~\ref{S:local}, where we
will also review the necessary background from~\cite{book}. 
Finally the rest of our global results will be deduced from the
corresponding local results, in Section~\ref{S:global 2}. 

\bigskip

{\bf Acknowledgments.} While writing this paper the authors were
partially supported by NSF grants DMS-0552314 and DMS-0301280,
respectively. We are indebted to M. Christ, M. Kim, 
J. H. Silverman, P. Vojta, and F. Voloch for their remarks and
suggestions. We thank also W. Duke, P. Michel, and S. Zhang for
discussing equidistribution of CM-points.

\section{Detailed exposition of the global results}
\label{S:global statements}

\subsection{Finiteness for $\CM$ points}

By a {\em coset} in an abelian variety $A$, we will mean a
translate of an abelian subvariety of $A$.
\begin{theorem}
\label{T:global1} Let $S=X_1(N)$ over $\Qbar$ for some $N \ge 1$. Let
$A$ be an abelian variety over $\Qbar$. Let $X$ be a closed
irreducible subvariety of $S \times A$.
Let $\Gamma \le A(\Qbar)$ be a finite-rank subgroup. If
$X(\Qbar) \intersect (\CM \times \Gamma)$ is Zariski
dense in $X$, then $X = S' \times A'$ where $S'$ is a subvariety
of $S$ and $A'$ is a coset in $A$.
\end{theorem}

If we drop the assumption that $X(\overline{\bQ}) \intersect (\CM
\times \Gamma)$ is Zariski dense in $X$, we can apply
Theorem~\ref{T:global1} to the irreducible components of the
Zariski closure of $X(\overline{\bQ}) \intersect (\CM \times
\Gamma)$ to deduce the following equivalent form of
Theorem~\ref{T:global1}.

\begin{theorem}
\label{T:global1'} 
Let $S,A,\Gamma$ be as in Theorem~\ref{T:global1}. 
Let $X$ be a closed subvariety of $S \times A$.  
Then the intersection
$X(\overline{\bQ}) \intersect (\CM \times \Gamma)$ is contained in
a subvariety $Z \subseteq X$ that is a finite union of products
$S' \times A'$ where each $S'$ is a subvariety of $S$ and each
$A'$ is a coset in $A$.
\end{theorem}

\begin{corollary}
\label{C:global E} 
Let $\thecorr$ be a modular-elliptic
correspondence and let $\Gamma \leq A(\overline{\bQ})$ be a finite rank
subgroup. Then $\Phi(\Pi^{-1}(\CM)) \cap \Gamma$ is finite.
\end{corollary}

We can actually
strengthen Theorem~\ref{T:global1}, as~\cite{Poonen1999}
strengthened the Mordell-Lang conjecture, by fattening $\Gamma$ as
follows. Let $h \colon A(\Qbar) \to \R_{\ge 0}$ be a canonical
height function attached to some symmetric ample line bundle on
$A$. For $\Gamma \leq A(\Qbar)$ and $\epsilon \ge 0$, let
\[
    \Gamma_\epsilon:=\{\,\gamma+a \ |\  \gamma \in \Gamma,
    a \in A(\Qbar), h(a) \le \epsilon \,\}.
\]

\begin{theorem}
\label{T:global2} Assume that $S,A,X,\Gamma$ are as in
Theorem~\ref{T:global1}. If $X(\Qbar) \intersect (\CM \times
\Gamma_\epsilon)$ is Zariski dense in $X$ for every $\epsilon>0$,
then $X = S' \times A'$ where $S'$ is a subvariety of $S$ and $A'$
is a coset in $A$.
\end{theorem}

Just as Theorem~\ref{T:global1} implied Theorem~\ref{T:global1'},
Theorem~\ref{T:global2} implies the following more general
(but equivalent) version of itself:

\begin{theorem}
\label{T:global2'} Assume that $S,A,\Gamma$ are as in
Theorem~\ref{T:global2}. Let $X$ be a closed subvariety of $S
\times A$ defined over $\overline{\bQ}$. Then for some
$\epsilon>0$, the intersection $X(\Qbar) \intersect (\CM \times
\Gamma_\epsilon)$ is contained in a subvariety $Z \subseteq X$
that is a finite union of products $S' \times A'$ where each $S'$
is a subvariety of $S$ and each $A'$ is a coset in $A$.
\end{theorem}

In the Shimura curve case, our global results are weaker, but we
can still prove the following:

\begin{theorem}
\label{T:global3} Let $S=X^D(\cU)$ over $\Qbar$, let $A$ be an
abelian variety over $\Qbar$, and let $\Phi\colon S \to A$ be a
morphism.  Let $\Gamma \le A(\overline{\bQ})$ be a
finite-rank subgroup. Then $\Phi(\CM) \intersect \Gamma$ is finite.
\end{theorem}

In the case $\dim A=1$ we can rephrase this as:

\begin{corollary}
\label{C:Shimura CM and finite rank}
 Let $\thecorr$ be a Shimura-elliptic
correspondence with $S=X=X^D(\cU)$ and $\Pi$ the identity.
Let $\Gamma \leq A(\overline{\bQ})$ be a finite rank subgroup. 
Then $\Phi(\CM) \cap \Gamma$ is finite.
\end{corollary}

Finally we can again fatten $\Gamma$:

\begin{theorem}
\label{T:global4} In the notation of Theorem~\ref{T:global3},
there exists $\epsilon>0$ such that 
$\Phi(\CM) \intersect \Gamma_\epsilon$ is finite.
\end{theorem}

Theorems \ref{T:global2} and \ref{T:global4}, which imply all the
other results above, will be proved in Section~\ref{S:equidistribution}.

\subsection{Finiteness for isogeny classes}
We begin by introducing certain subfields of $\overline{\bQ}$ and
certain subgroups of elliptic curves.

\begin{definition}
\label{Qppp}
  A subfield $M \subset \overline{\bQ}$ is {\it maximally unramified} 
  at a prime $\wp$ of $M$ if $\wp$ is unramified
  above $\bQ$ and for any subfield
   $M' \subset \overline{\bQ}$ containing $M$ and
  any prime $\wp'$ of $M'$  unramified over
  $\wp$ we have $M=M'$ (and hence also $\wp=\wp'$.) 
  Let $\cO_{M,\wp}$ be the local ring at $\wp$ of the ring of integers
  $\cO_M$ of $M$.
\end{definition}

\begin{remark}
\hfill
\begin{enumerate}
\item
For any rational prime $p$ there exist subfields $M \subset
\overline{\bQ}$ that are maximally unramified 
at some prime above $p$.
\item 
Any  $M \subset \overline{\bQ}$ maximally unramified at a prime
over $p$ is isomorphic to the algebraic closure of $\bQ$ in the
completion of the maximal unramified extension of $\bQ_p$. So if
$M_1,M_2 \subset \overline{\bQ}$ are two subfields
maximally unramified at  primes
over $p$ then there exists an automorphism $\sigma$ of $\overline{\bQ}$
such that $\sigma M_1=M_2$.
\item
If $F \subset \overline{\bQ}$ is any number field (i.e. finite
extension of $\bQ$) in which $p$ is unramified and if $M \subset
\overline{\bQ}$ is any subfield maximally unramified at a prime
above $p$ then $F \subset M$. Indeed, if $F'$ is  the Galois
closure of $F$ in $\overline{\bQ}$ then $p$ is unramified in $F'$.
Consider any prime in $F'$ above $p$. Then the completion of $F'$
at this prime has an embedding $\sigma$ into the completion of $M$
at $\wp$. By the maximality of $M$ we get $\sigma F' \subset M$.
Since $F'$ is normal over $\bQ$ we have $\sigma F'=F'$ hence $F
\subset F' \subset M$.
\end{enumerate}
\end{remark}

\begin{definition}
\label{D:rankpG}
For any prime $p$ and
any abelian group $G$ (written additively) we set
\[
	G\pdiv:=G_{\tors}+pG.
\]
For any subgroup $\Gamma \leq G$ define
\[
	\rank_p^G(\Gamma):=\dim_{\bF_p} \left(\frac{\Gamma}
					{\Gamma \cap G\pdiv} \right).
\]
\end{definition}

We always have
\begin{align*}
	\rank_p^G(\Gamma) &\le \dim_\Q(\Gamma \tensor \Q)=:\rank \Gamma,\\
	\rank_p^G(\Gamma) &\le \dim_{\F_p}(\Gamma \otimes_{\Z} \F_p).
\end{align*}
If $A$ is an elliptic curve over a field $M \subset \overline{\Q}$
that is maximally unramified at a prime above $p$,
the study of $A(M)\pdiv$ is analogous to the study of {\it Wieferich places}
in \cite{silwief} and~\cite{volwief}: for $a \in \Z \setminus p\Z$, 
the classical Wieferich condition $a^p \equiv a \pmod{p^2}$ is equivalent
to $a \in M^{\times p}$.) 

\begin{remark}
Let $\Delta$ be a positive-density set of rational primes.
Let $K$ be a number field unramified at the primes in $\Delta$.
Let $\calM_\Delta$ be the set of subfields of $\overline{K}$
maximally unramified at some prime above some $p \in \Delta$.
Let $A$ be an elliptic curve over $K$.
Then it is reasonable to expect that 
$\Intersection_{M \in \calM_{\Delta}} A(M)\pdiv 
\subseteq A(\overline{K})_{\tors}$.
\end{remark}

\begin{remark}
\label{infra} 
On the other hand, if $\Delta$ is a finite set of rational primes,
and $K$ and $A$ are as in the previous remark,
then $\Intersection_{M \in \calM_{\Delta}} A(M)\pdiv$
is of infinite rank.
This follows from the following statement:
If $L$ is the compositum in $\overline{K}$ of all quadratic
extensions of $K$ that are unramified at all primes above
$p \in \Delta$, then $A(L)$ is of infinite rank.
To prove this, choose a Weierstrass equation $y^2=f(x)$ for $A$,
where $f(x)$ is a monic cubic polynomial with coefficients
in the ring of integers $\OO_K$ of $K$.
Let $P = \prod_{p \in \Delta} p$.
Consider points with $x$-coordinate $x_n = 1/P^4 + n$ for $n \in \OO_K$.
Then $K(\sqrt{f(x_n)})$ is unramified at $p$ 
since the equation $P^{12} f(x_n) \equiv 1 \pmod{P^4}$
implies by Hensel's lemma that $P^{12} f(x_n)$ 
is a square in the completion of $K$ at any prime above $p \in \Delta$.
Thus we get a collection of points in $A(L)$. 
We may inductively define a sequence of $n_i \in \OO_K$ such that 
each $K(\sqrt{f(x_{n_i})})$ is ramified at a prime of $K$
not ramifying in the field generated by the previous square roots,
by choosing $n_i$ so that $1/P^4+n_i$ has valuation $1$ at some
prime of $K$ splitting completely in the splitting field of $f$.
By choosing the $n_i$ sufficiently
large, we may assume that the corresponding points $P_i \in A(L)$ 
have large height and hence are non-torsion. 
Now we claim that the Galois action forces
$P_1,\ldots,P_m$ to be $\bZ$-independent in $A(L)$. 
Indeed, if there were a relation
$a_1 P_1 + \cdots + a_m P_m = 0$
then we could apply a Galois automorphism fixing all the $P_i$ but $P_1$
to obtain
$-a_1 P_1 + a_2 P_2 + \cdots  + a_m P_m = 0$,
and subtracting would show that
$2 a_1 P_1 = 0$,
but $P_1$ is non-torsion, so $a_1=0$; similarly all $a_i$ would be $0$. 
Since $m$ can be made arbitrarily large, $A(L)$ has infinite rank.
\end{remark}

\begin{theorem}
\label{glver} Let $\thecorr$ be  a modular-elliptic  or a
Shimura-elliptic correspondence and let $Q \in S(\overline{\bQ})$.
Then there exists an infinite set $\Delta$ of primes such that for
any $p \in \Delta$  there is an infinite set $\Sigma$ of primes
with the following property. Let $M \subset \overline{\bQ}$ be
maximally unramified at a prime above $p$.
Suppose that $A$ is definable over $M$.
Let $C$ be the $\Sigma$-isogeny class of $Q$ in $S(\Qbar)$. 
Then there exists a constant $c$ such that 
for any subgroup $\Gamma \leq A(M)$
with $r:=\rank^{A(M)}_p(\Gamma)<\infty$,
the set $\Phi(\Pi^{-1}(C)) \cap \Gamma$
is finite of cardinality at most $c p^r$.
If in addition $Q \in \CM$, then for each $p \in \Delta$ 
one can take $\Sigma=\{l \mid l \neq p\}$.
\end{theorem}

Theorem~\ref{glver} applies, for instance,
to groups of the form $\Gamma=\Gamma_0+A(M)\pdiv$ where $\Gamma_0 \leq
A(M)$ is finitely generated; note that such a $\Gamma$ has
$\rank^{A(M)}_p(\Gamma)<\infty$ and
contains the  prime-to-$p$ division hull
of $\Gamma_0$ in $A(\Qbar)$ if $A$ has good reduction at all primes above $p$.
(Recall that the {\it prime-to-$p$ division hull} 
of $\Gamma_0$ in $A(\Qbar)$ is  the group of all $x \in A(\Qbar)$
such that there exists $n \in \bZ \setminus p\bZ$ with 
$nx \in \Gamma_0$.) 
On the other hand, by Remark~\ref{infra},
if $A$ is defined over $\bQ$, then $\Gamma$ is of infinite rank.

Theorem~\ref{glver} implies the following:

\begin{corollary}
 Let $\thecorr$ be  a modular-elliptic  or a
Shimura-elliptic correspondence, let $Q \in S(\overline{\bQ})$ and
let $\Gamma_0 \leq A(\overline{\bQ})$ be a finitely generated
subgroup. Then there exists an infinite set $\Delta$ of primes
such that for any  $p \in \Delta$  there is an infinite set
$\Sigma$ of primes with the following property. If  $C$ is the
$\Sigma$-isogeny class of $Q$ in $S(\Qbar)$ and $\Gamma$ is the
prime-to-$p$ division hull of $\Gamma_0$ in $A(\Qbar)$ then the set
$\Phi(\Pi^{-1}(C)) \cap \Gamma$ is finite.
\end{corollary}

In the modular-elliptic case we can reverse the
roles of $\Sigma$ and $\Delta$, roughly speaking:

\begin{theorem}
\label{incauna} Let $\thecorr$ be  a modular-elliptic
correspondence and let $Q \in S(\overline{\bQ})$. Then there
is a constant $l_0$ such that for any finite set $\Sigma$ of
primes greater than $l_0$ there is an infinite set $\Delta$ of primes with
the following property. For any $p \in \Delta$ there exists $M
\subset \overline{\bQ}$ maximally unramified at a prime above $p$
and containing a field of definition of $A$
such that if $\Gamma \leq A(M)$ is any  subgroup with
$r:=\rank^{A(M)}_p(\Gamma)<\infty$ and  $C$ is the
$\Sigma$-isogeny class of $Q$ in $S(\Qbar)$ then the set
$\Phi(\Pi^{-1}(C)) \cap \Gamma$ is finite of cardinality at most
$c p^{r}$, where $c$ is a constant not depending on $\Gamma$.
\end{theorem}

Theorem~\ref{incauna} implies, in particular, the following:

\begin{corollary}
 Let $\thecorr$ be  a modular-elliptic
correspondence, let $Q \in S(\overline{\bQ})$, and let $\Gamma_0
\leq A(\overline{\bQ})$ be a finitely generated subgroup. Then
there is a constant $l_0$ such that for any finite set $\Sigma$
of primes greater than $l_0$ there is an infinite set $\Delta$ of primes
with the following property. If $p \in \Delta$, if $C$ is the
$\Sigma$-isogeny class of $Q$ in $S(\Qbar)$, and if $\Gamma$ is the
prime-to-$p$ division hull of $\Gamma_0$ in $A(\Qbar)$, 
then the set $\Phi(\Pi^{-1}(C)) \cap \Gamma$ is finite.
\end{corollary}

Theorems  \ref{glver} and~\ref{incauna} will be derived from their
local counterpart, Corollary~\ref{bizzet}, which will be stated
later; the derivation will not be straightforward and will be
given in Section~\ref{S:global 2}. 
The special case when $Q \in \CM$ in
Theorems \ref{glver} and~\ref{incauna} is not superseded by 
Corollaries \ref{C:global E} and~\ref{C:Shimura CM and finite rank}:
the groups $\Gamma$ in Theorems \ref{glver} and \ref{incauna} are
allowed to have infinite rank and the finiteness statements come
with effective bounds.

\subsection{$\Rec$ functions}
Let $\thecorr$ be a modular-elliptic or a Shimura-elliptic
correspondence and let $X^{\dug} \subset X$ be a Zariski open set.
Recall from the introduction that we are interested in  a
description of the group $\Phi^{\perp}$ of 
all $\CM$ divisors $\sum m_iP_i$ on $X^{\dug}$ 
such that $\sum m_i \Phi(P_i) \in A(\Qbar)_{\tors}$. 
Such a theorem will be obtained in the local case;
see Theorem~\ref{mainth}. Taking the clue from the local
picture one may ask, in our global case here, if there exists a
regular function $\Sigmat$ on $X^{\dug}$ such that
for any $\CM$ divisor $\sum m_i P_i$ on $X^{\dug}$ we have that
$\sum m_i
\Phi(P_i)\in A(\Qbar)_{\tors}$ if and only if $\sum m_i
\Sigmat(P_i)=0$. We could refer to such a $\Phi^{\dug}$ as a
$\itrec$ {\it function for $\CM$ points}.

However, as we shall  presently see, no  $\rec$ function for $\CM$ points
exists in the global case, even in the ``most classical'' situation
when our correspondence arises from a modular parametrization: 
see Corollary~\ref{poiu}.

On the other hand, in the global case, 
for correspondences arising from modular parametrizations,
we will prove the existence of a  $\rec$ function mod $p$ for $\CM$ points;
see Theorem~\ref{simairefined}.

Finally we will prove an elementary,  purely geometric result
comparing linear dependence relations on elliptic curves with
corresponding linear dependence relations in the additive group;
see Theorem~\ref{triples}. Morally this result shows that ``there
are no purely geometric reasons" for the existence of $\rec$
functions (or $\rec$ functions mod $p$); so the existence of such
functions should be viewed, in some sense,  as an effect of  ``arithmetic''
and not  of  ``geometry alone''.

Let us begin by explaining our existence result for $\rec$ functions
mod $p$. Assume that $f=\sum a_n q^n$ is a newform (as usual, of weight
$2$, on $\Gamma_0(N)$, normalized, i.e. $a_1=1$, and  with rational
coefficients, hence $a_n \in \bZ$). As we shall explain in
Remark~\ref{irigutza}, for any prime $p$, 
there exists a modular form $f_{p^2-p}$ of weight
$p^2-p$ on $\Gamma_1(N)$, defined over $\bZ$,
whose Fourier expansion $f_{p^2-p}(q) \in \bZ[[q]]$ satisfies
\[
	f_{p^2-p}(q) \equiv \sum_{(n,p)=1} \frac{a_n}{n}q^n \pmod{p \bZ_{(p)}[[q]]}.
\] 
Fix such a form. 
Also consider the modular form $E_{p-1}$ of weight $p-1$
over $\bZ_{(p)}$ whose $q$-expansion in $\bZ_{(p)}[[q]]$ is the
normalized Eisenstein series of weight $p-1$; here ``normalized''
means ``with constant coefficient $1$''. Then the quotient
\[\Sigmat_p:=\frac{f_{p^2-p}}{E_{p-1}^p}\]
is a rational function on $S$ defined over $\bQ$.

\begin{theorem}[Reciprocity functions mod $p$ for $\CM$ points]
\label{simairefined} Let $S=X_1(N)$ and let $\Phi:S \to A$ be a
modular parametrization attached to a newform $f$, with $A$ non
$\CM$. Let $P_1,\ldots,P_n \in \CM$ correspond to (not necessarily
distinct) imaginary quadratic fields $\cK_1,\ldots,\cK_n$. 
Assume that no $\cK_i$ equals $\bQ(\sqrt{-1})$ or $\bQ(\sqrt{-3})$. 
Let $p$ be a sufficiently large prime splitting
completely in the compositum $\cK_1 \cdots \cK_n$, and let $M \subset
\overline{\bQ}$ be maximally unramified at a prime $\wp$ above
$p$. Then $\Sigmat_p$ does not have poles among $P_1,\ldots,P_n$, we
have $\Sigmat_p(P_1),\ldots,\Sigmat_p(P_n) \in \cO_{M,\wp}$, and for
all $m_1,\ldots,m_n \in \bZ$ we have
\[\sum_{i=1}^n m_i \Phi(P_i) \in A(M)\pdiv \Longleftrightarrow
\sum_{i=1}^n m_i \Sigmat_p(P_i) \in \wp\cO_{M,\wp}.\] 
The implication ''$\Longrightarrow$'' holds even if $A$ is CM.
\end{theorem}

 We expect that the condition $\cK_i
\neq \bQ(\sqrt{-1}),\bQ(\sqrt{-3})$ can be removed.
Theorem~\ref{simairefined}
will be derived in Section~\ref{S:global 2} from its local counterpart, 
Theorem~\ref{refined}. 

Theorem~\ref{simairefined} immediately implies the following
necessary criterion for the trace of a $\CM$ point to be torsion:

\begin{corollary}
Let $S=X_1(N)$ and let $\Phi:S \to A$ be a
modular parametrization attached to a newform $f$. 
Let $\cK$ be an imaginary quadratic field not equal to
$\Q(\sqrt{-1})$ or $\Q(\sqrt{-3})$,
let $L$ be a finite extension of $K$,
and let $P \in S(L)$ be a point corresponding to an elliptic
curve with CM by an order in $\cK$.
If
\[
	\Tr_{L/\cK} \Phi(P) \in A(\cK)_{\tors},
\]
then for any degree-$1$ prime $\wp$ of $\cK$ with
$p:=\Char(\cO_{\cK}/\wp) \gg 0$ we have
\[
	\Tr_{L/\cK} \Sigmat_p(P) \in \wp \cO_{\cK,\wp}.
\]
\end{corollary}

\begin{theorem}[Non-existence of reciprocity functions for isogeny classes]
\label{nonu} 
Let $\Phi \colon S=X_1(N) \to A$ be a modular parametrization.
Let $C \subset S(\Qbar)$ be an isogeny class and let $\Sigmat$ be a
rational function on $S$ none of whose poles is in $C$. Assume
that for any $P_1,\ldots,P_n \in C$ and any $m_1,\ldots,m_n \in \bZ$ we
have
\begin{equation}
\label{vivald} \sum_{i=1}^n m_i\Phi(P_i) \in A(\Qbar)_{\tors} \ \ \
\Rightarrow\ \ \ \sum_{i=1}^n m_i \Sigmat(P_i)=0\in \Qbar.
\end{equation} 
Then $\Sigmat=0$.
\end{theorem}

Theorem~\ref{nonu} will be proved in Section~\ref{S:Hecke divisors}.
It trivially implies the following:

\begin{corollary}[Non-existence of reciprocity functions for $\CM$ points]
\label{poiu} Let $\Phi:X_1(N) \to A$ be a modular parametrization.
Assume there is a  non-empty Zariski open set $X^{\dug} \subset
X_1(N)$ and a regular function $\Sigmat \in \cO(X^{\dug})$ having
the property that for any $P_1,\ldots,P_n \in \CM \cap X^{\dug}(\Qbar)$
and any $m_1,\ldots,m_n \in \bZ$ we have \[ \sum_{i=1}^n m_i\Phi(P_i)
\in A(\Qbar)_{\tors} \ \ \ \Rightarrow\ \ \ \sum_{i=1}^n m_i
\Sigmat(P_i)=0\in \Qbar.
\]
 Then $\Sigmat=0$.
\end{corollary}

We end our discussion here by stating our ``purely geometric''
result comparing linear dependence relations on elliptic curves with
linear dependence relations in the additive group.

\begin{theorem}
\label{triples}
Let $\Phi\colon X \to A$ be a non-constant morphism between smooth
projective curves over an algebraically closed field $k$ of
characteristic $p \geq 0$. 
Let $n \ge 3$, and let $a_1,\ldots,a_n$ be nonzero integers
not all divisible by $p$.
Suppose that $X^{\dug} \subset X$ is an affine open subset 
and $\Sigmat \in \cO(X^{\dug})$ is a regular function 
such that for any $P_1,\ldots,P_n \in X^{\dug}(k)$ we have
\begin{equation} 
\label{ticc} 
\sum_{i=1}^n a_i \Phi(P_i)=0\ \ \
\Longrightarrow\ \ \ \sum_{i=1}^n a_i \Sigmat(P_i)=0.
\end{equation}
Then $\Sigmat$ is constant.
In particular, if $\sum_{i=1}^n a_i$ is not divisible by $p$, 
then $\Phi^{\dug}=0$.
\end{theorem}

Theorem~\ref{triples} will be proved in Section~\ref{S:Hecke divisors}.
Theorem~\ref{triples} fails for both $n=2$ and $n=1$.

\section{Detailed exposition of the local results}
\label{S:local statements}

\subsection{General conventions and notation}
Fix a prime $p$.
Let $\Z_p$ be the ring of $p$-adic integers.
Let $\Z_p^{\ur}$ be the maximal unramified extension of $\Z_p$.
Let $R:=\hat{\Z}_p^{\ur}$ be the completion of $\Z_p^{\ur}$.
We set $k=R/pR$ and $K:=R[1/p]$. 
Thus $k \isom \Fbar_p$, and $R$ is the Witt ring $W(k)$.
Let $\Fr \colon k \to k$ be the automorphism $\Fr(x):=x^p$,
and let $\phi \colon R \to R$ 
be the unique automorphism lifting $\Fr$. 

For each prime $p$, fix a
subfield  $M \subset \overline{\bQ}$  maximally unramified at a
prime $\wp$  above $p$. The completion of $M$ at $\wp$ is
isomorphic to $K$. We fix such an isomorphism. We then get
 an embedding $M \injects K$, which we shall view as an inclusion,
\begin{equation} \label{iinclusion} M \to K,
\end{equation} with $pR$ lying over $\wp$.
Then $M$ is the algebraic closure of $\bQ$ in $K$,
we have
 \begin{equation}
 \label{erasamorr}
 \cO_{M,\wp}=M \cap R,
 \end{equation}
 and $R$ is the completion of $\cO_{M,\wp}$.
 If $F$ is a number field
with ring of integers  $\cO_F$ and $p$ is unramified in $F$, then
$F  \subset M \subset K$, so
  $\cO_{F,\wp_F} \subset \cO_{M,\wp} \subset  R$, 
where $\wp_F:=\wp \cap \cO_F$.

We will use the notion of a {\em canonical lift} (CL) abelian scheme 
over $R$: see Section~\ref{S:review of CL and CM}
for the definition.

\subsection{Hecke correspondences}
\label{S:Hecke correspondences}

For any prime $l$ let
$Y_1(N,l)$ be the affine curve over $\Qbar$ 
parameterizing triples $(E,\alpha,H)$ in
which $(E,\alpha)$, with $\alpha\colon \bZ/N\bZ \hookrightarrow E(\Qbar)$, 
represents a point in $Y_1(N)$ 
and $H \leq E(\Qbar)$ is an order-$l$ subgroup
intersecting $\alpha(\Z/N\Z)$ trivially: see~\cite[p.~207]{conrad}.
Define {\em degeneracy maps}  $\sigma_1,\sigma_2\colon Y_1(N,l) \to Y_1(N)$
by $\sigma_1(E,\alpha,H):=(E,\alpha)$
and $\sigma_2(E,\alpha,H):=(E/H,u \circ \alpha)$,
where $u \colon E \to E/H$ is the quotient map.

Let $X_1(N,l)$ be the smooth projective model of $Y_1(N,l)$.
The $\sigma_i$ extend to $\sigma_i \colon X_1(N,l) \to X_1(N)$.
Define the {\em Hecke operator} $T(l)_*$ on $\Div(X_1(N)(\Qbar))$
by $T(l)_* D := \sigma_{2*} \sigma_1^* D$.
For $P \in X_1(N)(\Qbar)$ write $T(l)_*P =: \sum_i P_i^{(l)}$;
the sum involves $l+1$ or $l$ terms
according as $l \nmid N$ or $l \mid N$.
If in addition $f = \sum a_n q^n \in \Z[[q]]$ is a newform,
then the divisor $\sum_i P_i^{(l)}-a_l P$
will be called a {\it Hecke divisor}.

\subsection{Conventions on modular-elliptic correspondences} 
\label{S:conventions modular-elliptic}
The $\Z[1/N]$-scheme $Y_1(N)$ represents the functor taking
a $\Z[1/N]$-algebra $B$ to the set of isomorphism classes of pairs
$(E,\alpha)$ where $E$ is an elliptic curve over $B$
and $\alpha\colon (\Z/N\Z)_B \to E$ is a closed immersion of group
schemes.
For each $P \in Y_1(N)(B)$, 
let $(E_P,\alpha_P)$ be a pair in the corresponding isomorphism class.
The $\Z[1/N]$-scheme $S=X_1(N)$ is the Deligne-Rapoport compactification:
see~\cite[pp.~78--81]{DI}.
The base extension of $S$ to $\C$ will also be denoted $S$.
The cusp $\infty$ on $X_1(N)$ is defined over $\Q(\zeta_N)$,
where $\zeta_N$ is a primitive $N^{\operatorname{th}}$ root of $1$.

\begin{remark}
\label{R:zeta_N}
Some of the references we cite use a modular curve
parameterizing elliptic curves with an
embedding of $\mu_N$ instead of $\bZ/N\bZ$,
but the two theories are isomorphic provided we work
over $\Z[1/N,\zeta_N]$-algebras.
\end{remark}

Assume that we are given a modular-elliptic correspondence
$\thecorr$ with $S=X_1(N)$. 
We may assume that $A$ comes from a model over $\OO_{F_0}[1/Nm]$,
and that $X,S,\Pi,\Phi$ come from models over $\OO_F[1/Nm]$,
where $F_0 \subseteq F$ are number fields,
and $\OO_{F_0}$ and $\OO_F$ are their rings of integers,
and $m \in \Z_{>0}$.
Then $x_\infty=\Pi(\infty)$ has a model over $\OO_{F_1}[1/Nm]$,
where $F_1$ is a number field containing $F(\zeta_N)$.

If $p$ is large enough to be unramified in $F_1$,
then we may view $F_1$ as a subfield of $M$, which was
embedded in $K$; 
then we obtain an embedding $\OO_{F_1}[1/Nm] \subset R$.
A point $P \in S(R)$ is called {\it ordinary} 
(respectively, a CL-{\it point})
if $P \in Y_1(N)(R)$ and $E_P$ has ordinary reduction $\overline{E}_p$
(respectively, $E_P$ is CL).
If $P \in S(R)$ is ordinary let $\cK_P$ be the imaginary
quadratic field $\End(\overline{E}_P) \otimes \bQ$. 
Finally let $\CL$ be the set of all CL-points of $S(R)$. 
Call $\Phi(\Pi^{-1}(\CL)) \subset A(R)$ the set of CL-{\it points} of $A$.

  \subsection{Conventions on Shimura-elliptic correspondences}
      \label{conventions on m-e}
Now suppose instead that $S$ is a Shimura curve $X^D(\cU)$,
where $\cU$ satisfies 
Let $\cU$ satisfy the conditions in~\cite{buzzard};
then for some $m \in \Z_{>0}$ 
the Shimura curve $S=X^D(\cU)$ is a $\bZ[1/m]$-scheme
with geometrically integral fibers,
such that for any $\bZ[1/m]$-algebra $B$ the set $S(B)$
is in bijection with the set of isomorphism
  classes of triples $(E,i,\alpha)$ where
$(E,i)$ is a false elliptic curve over $B$ (i.e. $E/B$ is an
abelian scheme of relative dimension $2$ and $i\colon \cO_D \ra
\End(E/B)$ is an injective ring homomorphism)
  and $\alpha$ is a level $\cU$ structure.

Assume that we are given a  Shimura-elliptic correspondence $\thecorr$.
With notation as in Section~\ref{S:conventions modular-elliptic},
Replacing $m$ by a multiple if necessary,
we may assume that $A$ comes from a model over $\OO_{F_0}[1/m]$
and that $X,S,\Pi,\Phi$ come from models over $\OO_{F}[1/m]$,
where $F_0 \subseteq F$ are number fields.
Assuming that $p$ is suitably large,
we again have an embedding $\OO_{F}[1/m] \subseteq R$.
A point $P \in S(R)$ is  called {\it ordinary} (respectively a
{\it $\CL$-point}) if $P$ corresponds to a triple
$(E_P,i_P,\alpha_P)$ where $E_P$ has ordinary reduction
$\overline{E}_P$ (respectively $E_P$ is CL). If $P \in S(R)$ is
ordinary, let $\cK_P$ be the imaginary quadratic field
$\End(\overline{E}_P,\overline{i}_P) \otimes \bQ$. 
Finally $\CL$ is the set of all CL-points of $S(R)$. 
Call $\Phi(\Pi^{-1}(\CL)) \subset A(R)$ 
the set of CL-{\it points} of $A$.

\subsection{$\Rec$ functions for $\CL$ points}
\begin{definition}
A degree-$1$ place $v$ of a number field $F_0$ is {\em anomalous}
for an elliptic curve $A$ over $F_0$
if the $p$-power Frobenius on the reduction $A$ mod $v$ 
has trace $a_v \equiv 1 \pmod{p}$.
(See~\cite[p.~186]{mazur}.)
\end{definition}

Let notation be as in Section~\ref{S:conventions modular-elliptic}
or Section~\ref{conventions on m-e}.

\begin{definition}
A rational prime $p$ is {\em good} (for our correspondence)
if $p$ splits completely in $F_0$, 
the elliptic curve $A$ has good reduction at all primes $v|p$,
and in the Shimura-elliptic case each $v|p$ is not anomalous for $A$. 
\end{definition}

\begin{remark}
The Chebotarev density theorem implies that 
there are infinitely many good primes: 
see Lemma~\ref{anom} for details.
\end{remark}

Let $p$ be sufficiently large and set $X_{R}:=X \tensor R$. 
(More generally, throughout this paper the subscript $R$ always
means ``base extension to $R$'' and we use the same
convention for any other ring in place of $R$.
In particular, if $p$ is a good prime, $A_{R}$ comes from
an elliptic curve $A_{\bZ_p}$ over $\bZ_p$ and we let $a_p$ 
be the trace of the $p$-power Frobenius on $A_{\bF_p}$.)
Let $\bar{X}:=X_{k}=X \otimes k$. For any $P \in X(R)$,
let $\bar{P}$ denote the image of $P$ in $\bar{X}(k)$. 
(More generally,
throughout  this paper, when we are dealing with a situation that
is ``localized at $p$'', an upper bar always means ``reduction mod $p$''.)  
Let $\hat{X}_{R}$ the $p$-adic completion of $X_{R}$ viewed as a
formal scheme over $R$. (More generally, throughout this paper,
an upper $\hat{\ }$ will denote ``$p$-adic completion''.) If
$X^{\dug} \subset X_{R}$ is an affine Zariski open set then any
global function
 $\Sigmat \in \cO(\hat{X}^{\dug})=\cO(X^{\dug})\h$ defines a map
  $\Sigmat\colon X^{\dug}(R) \to R$. The reduction $\overline{\Sigmat}
  \in \cO(\bar{X}^{\dug})$
induces a regular map $\overline{\Sigmat}\colon \bar{X}^{\dug}(k) \to k$.
Define
\[A(R)\pdiv:=A(R)_{\tors}+pA(R)\leq A(R).\]
Then $A(R)_{\pdiv} \cap A(M)=A(M)_{\pdiv}$.

\begin{theorem}[Reciprocity functions for $\CL$ points]
\label{mainth} Assume that $\thecorr$ is  a modular-elliptic
 or a Shimura-elliptic correspondence  and that
  $p$ is a sufficiently large good prime.
 Then  there exists an affine Zariski open
subset $X^{\dug} \subset X_{R}$  and a
function $\Sigmat \in \cO(\hat{X}^{\dug})$ with non-constant
reduction $\overline{\Sigmat} \in \cO(\bar{X}^{\dug}) \setminus
k$, such that for any $P_1,\ldots,P_n \in \Pi^{-1}(\CL) \cap
X^{\dug}(R)$ and any $m_1,\ldots,m_n \in \bZ$ we have
\[\begin{array}{lll}
\sum_{i=1}^n m_i\Phi(P_i) \in A(R)_{\tors} & \Longleftrightarrow
&  \sum_{i=1}^n m_i \Sigmat(P_i) =0\in R,\\
\  & \  & \  \\
\sum_{i=1}^n m_i\Phi(P_i) \in A(R)\pdiv  & \Longleftrightarrow &
\sum_{i=1}^n m_i \overline{\Sigmat}(\bar{P}_i) =0\in
k.\end{array}
\]
\end{theorem}

Theorem~\ref{mainth} will be proved in Section~\ref{S:local}.  
It is useful to compare Theorem~\ref{mainth} to 
Corollary~\ref{poiu} and Theorem~\ref{triples}.

\begin{remark}
\label{kokk} As the proof of Theorem~\ref{mainth} will show, the
functions $\Sigmat$ will be functorially associated to tuples
$(X,S,A,\Pi,\Phi,\omega_A)$,
 where $\omega_A$ is a nonzero global $1$-form  on $A$ defined over $F_0$.
More precisely the functions $\Sigmat$ will be constructed such
that the following hold:
\begin{enumerate}
\item
{\it Functoriality in $X$}: If $v:\tilde{X} \to X$ is a
non-constant morphism then
\[(\Phi \circ v)^{\dug}=\Sigmat \circ v.\]
\item
{\it Functoriality in $A$}: If $u \colon A \to \tilde{A}$ is an
isogeny defined over $F_0$ then
\[(u \circ \Phi)^{\dug}=\frac{u^* \omega_{\tilde{A}}}{\omega_A}
\cdot \Sigmat.\]
\item
{\it Invariance with respect to change of level}: If one
replaces $\Pi$ by $w \circ \Pi$ where $w$ is a map between two
modular (respectively Shimura) curves coming from changing levels
then $\Sigmat$ does not change.
\item
{\it Invariance with respect to Hecke correspondences}: 
Suppose we are in the modular-elliptic case.
Recall the degeneracy maps $\sigma_1,\sigma_2\colon X_1(N,l) \to X_1(N)$.
Suppose that $\Pi = \pi \circ \sigma_1$ for some $\pi\colon X \to X_1(N,l)$.
Then $\Sigmat$ does not change if one replaces $\sigma_1 \circ \pi$ by
$\sigma_2 \circ \pi$.
The analogous statement holds in the Shimura-elliptic case,
with $X^D(\cU,\Gamma_0(l))$ for some $l \nmid M_{\cU}\disc(D)$
playing the role of $X_1(N,l)$,
where $M_{\cU}$ is as in~\cite[p.~595]{buzzard}.
\end{enumerate}
\end{remark}

\begin{remark}
\label{aremark} Let $\cC=\Pi^{-1}(\CL) \cap X^{\dug}(R)$ and let
$\Div(\cC)$ be  the free abelian group generated by $\cC$. Then one
can consider the maps $\Phi_*\colon \Div(\cC)\ra
A(R)/A(R)_{\tors}$ and $\Sigmat_*\colon \Div(\cC) \to R$
naturally induced by $\Phi$ and $\Sigmat$ by additivity. Set
$\Phi^{\perp}:=\ker \Phi_*$ and $(\Sigmat)^{\perp}=\ker \Sigmat_*$.
Then the first equivalence in Theorem~\ref{mainth} says that
$\Phi^{\perp}=(\Sigmat)^{\perp}$. A similar description can be given
for the second equivalence. There is a formal similarity between
such a formulation of Theorem~\ref{mainth} and the way classical
reciprocity laws are formulated in number theory and algebraic
geometry. Indeed, in classical reciprocity laws one is usually
presented with maps $\Phi:\cC \to G$ and $\Phi^{\dug}:\cC \ra
G^{\dug}$ from a set $\cC$ of places of a global field to two 
groups $G$ and $G^{\dug}$ (typically a Galois group and a class
group), and one claims the equality of the kernels of
the induced maps $\Phi_*:\Div(\cC) \to G$ and $\Phi^{\dug}_*:\Div(\cC)
\to G^{\dug}$.
\end{remark}

Let us discuss some consequences of Theorem~\ref{mainth}. 

\begin{corollary}
\label{cor1} In the notation of Definition~\ref{D:rankpG}
and Theorem~\ref{mainth},
 we have
\[
\begin{array}{lcl}
\rank \left(\sum_{i=1}^n \bZ \cdot   \Phi(P_i)\right) & = & \rank
\left(\sum_{i=1}^n \bZ \cdot  \Sigmat(P_i)
\right)\\
\  & \  & \  \\
\rank_p^{A(R)}\left(\sum_{i=1}^n \bZ \cdot   \Phi(P_i)\right)
& = & \dim_{\bF_p} \left( \sum_{i=1}^n \bF_p \cdot
\overline{\Sigmat}(\bar{P}_i) \right).\end{array}
\]
\end{corollary}

\begin{corollary}
 \label{coru}
Assume that $\thecorr$ is  a modular-elliptic
 or a Shimura-elliptic correspondence  and assume
that  $p$ is a sufficiently large good prime.
Then there exists a constant $c$
such that for any subgroup $\Gamma \leq A(R)$ with
$r:=\rank_p^{A(R)}(\Gamma)<\infty$,
the set $\Phi(\Pi^{-1}(\CL)) \cap \Gamma$
is finite of cardinality at most $c p^r$.
 \end{corollary}

\begin{proof}
By Corollary \ref{cor1}, the $\bF_p$-span of
\[\overline{\Sigmat}(\overline{\Phi^{-1}(\Gamma) \cap \Pi^{-1}(\CL)\cap
X^{\dug}(R)})\] has dimension $\leq r$ over $\bF_p$. So
\[
	\# \overline{\Phi^{-1}(\Gamma) \cap \Pi^{-1}(\CL)\cap X^{\dug}(R)}
	\le p^{r} \deg(\overline{\Sigmat}). 
\] 
Now CL elliptic curves over $R$ are uniquely determined, up to
isomorphism, by their reduction mod $p$: see Theorem~\ref{T:revv2}.
Similarly, by loc.\ cit., 
if $(E_1,i_1)$ and $(E_2,i_2)$ are two false elliptic curves such
that $E_1,E_2$ are CL and $(\bar{E}_1,\bar{i}_1) \simeq
(\bar{E}_2,\bar{i}_2)$ then $(E_1,i_1) \simeq (E_2,i_2)$.
Thus \[\Phi^{-1}(\Gamma) \cap \Pi^{-1}(\CL)\cap
X^{\dug}(R)\] has at most $p^{r} \deg(\overline{\Sigmat})
\cdot d_1d_2$ elements, where $d_1 := \deg \Pi$ and
$d_2$ is the number of level $\Gamma_1(N)$ structures
(respectively, level $\cU$ structures) on a given elliptic
(respectively, false elliptic) curve. Also, 
$\# \left( \Pi^{-1}(\CL) \setminus  X^{\dug}(R) \right) \le d_1d_2d_3$,
where
$d_3 = \# \left( \Pi^{-1}(\overline{S}^{\ord}(k)) \setminus \bar{X}^{\dug}(k) \right)$,
where the {\it ord} superscript indicaes the ordinary locus.
So
\begin{equation}\label{estt}
\#\Phi(\Pi^{-1}(\CL))\cap \Gamma 
\le
\#\Phi^{-1}(\Gamma) \cap \Pi^{-1}(\CL)
\le
(p^{r} \deg(\overline{\Sigmat})  +d_3) d_1 d_2,
\end{equation}
which is at most $c p^r$,
where $c:=\deg(\overline{\Sigmat})  + d_1d_2 d_3$.
\end{proof}

Corollary \ref{elll} will make the bound in \eqref{estt} 
explicit in the case where $S=X=X_1(N)$, $\Pi=\Id$, 
and $\Phi$ is a modular parametrization.

To explain our next application of Theorem~\ref{mainth} we fix a
modular-elliptic or Shimura-elliptic correspondence $\thecorr$
 and  a vector $\ba=(a_1,\ldots,a_n)\in \bZ^n$ of nonzero integers.
For a prime $p$, consider the set
\[D_{\ba}:=\{(Q_1,\ldots,Q_n) \in A^n({R})\ |\ \sum_{i=1}^n a_iQ_i=0\}\]
of all tuples ``killed by  $\ba$''. (E.g. if $n=3$ and
$a_1=a_2=a_3=1$ then $D_{\ba}$ is the set of  triples of collinear
points on $A(R)$ if we use a Weierstrass model for $A$;
any triple in $X^3(R)$ mapping to a  triple of collinear points
can be referred to  as a {\it triple of collinear points on}
$X(R)$.)

If $p$ is sufficiently large, it does not divide all the $a_i$,
and then the image of $D_{\ba}$ in $A^n(k)$ coincides with the set
$\bar{D}_{\ba}$ of all tuples  in $A^n(k)$ killed by $\ba$.
Clearly  $\bar{D}_{\ba}$ is (the set of points of) an irreducible
divisor (isomorphic to $\bar{A}^{n-1}$). We may consider the map
$\Phi_n := \Phi \times \cdots \times \Phi \colon X^n(R) \to A^n(R)$.
Similarly we have a map $\Pi_n \colon X^n(R) \ra
S^n(R)$ and a map $\bar{\Phi}_n\colon X^n(k) \to A^n(k)$.
Then $\bar{\Phi}_n^{-1}(\bar{D}_{\ba})$ is (the set of points of)
a possibly reducible divisor in $\bar{X}^n$. The next corollary is
a ``degeneracy'' result for CL points:

\begin{corollary}
\label{cor2} Let $\ba \in \bZ^n$ be a tuple of nonzero integers,
$n \geq 3$ and assume that $\thecorr$ is  a modular-elliptic
 or a Shimura-elliptic correspondence.  Assume
  $p$ is a sufficiently large good prime.
 Then
  the set $\overline{\Phi_n^{-1}(D_{\ba}) \cap \Pi_n^{-1}(\CL^n)}$
is not Zariski dense in $\bar{\Phi}_n^{-1}(\bar{D}_{\ba})$.
\end{corollary}

The overline means, as usual, the reduction mod $p$ map which in
this case is a map $X^n(R) \to X^n(k)$. In particular the set of
reductions mod $p$ of  triples of collinear CL points in $X(R)$ is
not Zariski dense in the set of triples of collinear points on
$X(k)$.  Corollary \ref{cor2} ceases to be true in case
$n=2$; indeed if $a_1=1$, $a_2=-1$, the set
$\overline{\Phi_2^{-1}(D_{\ba}) \cap \Pi_2^{-1}(\CL^2)}$ is the
complement in the curve $\bar{\Phi}_2^{-1}(\bar{D}_{\ba})=\bar{X}
\times_{\bar{A}} \bar{X}$ of a finite set.

\begin{proof}[Proof of  Corollary \ref{cor2}]
Let $X^{\dug}$ and $\Sigmat$ be as in the conclusion of Theorem
\ref{mainth};  in particular  $\overline{\Sigmat}$ is a non-constant
function and, for any tuple $(\bar{P}_1,\ldots,\bar{P}_n)$ in the set
\begin{equation}
\label{bigset} \overline{\Phi_n^{-1}(D_{\ba}) \cap \Pi_n^{-1}(\CL^n)
\cap X^{\dug}(R)^n}= \overline{\Phi_n^{-1}(D_{\ba}) \cap
\Pi_n^{-1}(\CL^n)} \cap X^{\dug}(k)^n\end{equation} we have
\begin{equation}
\label{dantz} \sum_{i=1}^n a_i
\overline{\Sigmat}(\bar{P}_i)=0.\end{equation} Now if the conclusion
of Corollary \ref{cor2} is false, the set
\ref{bigset} is Zariski dense in $\bar{\Phi}_n^{-1}(\bar{D}_{\ba})
\cap X^{\dug}(k)^n$ so \eqref{dantz} holds for all
$(\bar{P}_1,\ldots,\bar{P}_n)$ in $\bar{\Phi}_n^{-1}(\bar{D}_{\ba})
\cap X^{\dug}(k)^n$. But then, Theorem~\ref{triples} implies
that $\overline{\Sigmat}$ is a constant function, a contradiction.
\end{proof}

\subsection{Refinement of results on $\CL$ points 
for modular parametrizations} 
\label{S:refinement for modular CL}
Theorem~\ref{refined} below is
a refinement of Theorem~\ref{mainth} in the special case 
of a modular-elliptic correspondence $\thecorr$ arising from a
modular parametrization attached
to a newform $f=\sum a_nq^n$; recall that $S=X=X_1(N)$, $\Pi=\Id$, and
we always
assume $f$ of weight $2$, on $\Gamma_0(N)$, normalized, with
rational Fourier coefficients. In this case we may (and will)
take $F=F_0=\bQ$. Recall  that $a_1=1$, that $a_n \in \bZ$ for $n
\geq 1$, and that for sufficiently large $p$, the coefficient $a_p$ 
equals the trace of Frobenius on $A_{\bF_p}$.
One can ask if
in this case the function $\Sigmat$ also has a description in
terms of eigenforms. This is indeed the case, as we shall explain
below.
Consider the series
\begin{equation}
\label {fruct1}
 f^{(-1)}(q) := \sum_{(n,p)=1} \frac{a_n}{n} q^n \in
 \bZ_p[[q]].\end{equation}
The series $f^{(-1)}(q)$ is called $f|R_{-1}$ in~\cite[p.~115]{serre}.
Assume that $p\gg 0$
and that $A_{R}$ has ordinary reduction.
Then $a_p  \not\equiv 0 \pmod{p}$.
Let $up \in \bZ_p^{\times}$ 
be the unique root in $p \bZ_p$ of the equation
$x^2-a_px+p=0$;
thus $\bar{a}_p \bar{u} = 1$. 
Let $V:\bZ_p[[q]]\to \bZ_p[[q]]$ be the operator
$V(\sum c_n q^n)=\sum c_n q^{np}$. Define
  \begin{equation}
  \label{fruct2}f^{(-1)}_{[u]}(q) :=\left( \sum_{i=0}^{\infty} u^i
  V^i\right) f^{(-1)}(q)=
  \sum_{i \geq 0}  \sum_{(n,p)=1} u^i \frac{a_n}{n} q^{np^i} \in
 \bZ_p[[q]].
\end{equation}
Then
\begin{equation}
\label{fruct4}
-\left(\overline{f^{(-1)}_{[u]}(q)}\right)^p+\bar{a}_p
\overline{f^{(-1)}_{[u]}(q)}=
\bar{a}_p\overline{f^{(-1)}(q)},
\end{equation}
in $\F_p[[q]]$, where the bars denote reduction modulo $p$, as usual.
The series $\overline{f^{(-1)}(q)}$ has a nice interpretation in
terms of modular forms mod $p$.
Indeed, recall from \cite[pp.~451,~458]{gross} that
if $M_m$ is the $k$-linear space of modular forms
over $k$ on $\Gamma_1(N)$ of weight $m$ then there is an
injective $q$-{\it expansion map} $M_m \to k[[q]]$ and a 
{\em Serre operator} $\theta\colon M_m \to M_{m+p+1}$ that on
$q$-expansions acts as $q \, d/dq$.
Let $\bar{E}_{p-1} \in M_{p-1}$ 
be the reduction mod $p$ of the modular form $E_{p-1}$
over $\bZ_{(p)}$ whose $q$-expansion in $\bZ_{(p)}[[q]]$ is the
normalized Eisenstein series of weight $p-1$; hence
$\bar{E}_{p-1}$ is the Hasse invariant and has $q$-expansion $1$
in $\bF_p[[q]]$.  

Define the affine curve
\[
	\overline{X_1(N)}^{\ord}
	:=\overline{X_1(N)} \setminus \{\text{zero locus of $\bar{E}_{p-1}$}\}.
	= \overline{Y_1(N)}^{\ord} \cup \{\text{cusps}\}
\]
where $\overline{Y_1(N)}^{\ord}$ 
is the open set of points in $\overline{Y_1(N)}$ 
represented by ordinary elliptic curves.

If $\alpha \in M_{m+w}$, and $\beta \in M_m$ is nonzero,
call $\alpha/\beta$ a {\em weight-$w$ quotient} of modular forms over $k$.
A weight-$0$ quotient of modular forms is a rational function
on $\overline{X_1(N)}$.
In particular, $\theta^{p-2} \bar{f}, \bar{E}_{p-1}^p \in M_{p^2-p}$, and
\begin{equation} \label{fmu} 
	\bar{f}^{(-1)}:=(\theta^{p-2} \bar{f})/\bar{E}_{p-1}^p
\end{equation}
is a regular function on $\overline{X_1(N)}^{\ord}$.
Let $g \mapsto g_\infty$ be the natural $q$-expansion map
$k(\overline{X_1(N)}) \to k((q))$.
The corresponding point in $\overline{X_1(N)}(k((q)))$
will be called the {\it Fourier} $k((q))$-{\it point}.
Then $\bar{f}^{(-1)}_{\infty} =\overline{f^{(-1)}(q)}$. 
For primes $l \neq p$, define the {\em Hecke operator}
$T(l)\colon k[[q]] \to k[[q]]$ 
by $T(l)(\sum c_nq^n)=\sum c_{ln}q^n+\epsilon(l)l^{-1} \sum c_nq^{ln}$, 
where $\epsilon(l)=0$ or $1$ according as $l$ divides $N$ or not. 
Define the {\em $U$-operator} $U \colon k[[q]]\to k[[q]]$ 
by $U(\sum c_nq^n):=\sum c_{np}q^n$.
By~\cite[p.~458]{gross}, $\overline{f^{(-1)}(q)}$
is an eigenvector of $T_l$ for every $l \ne p$;
moreover, $\overline{f^{(-1)}(q)} \in \ker U$.
Finally, for any open set $X' \subset
X_1(N)_{R}$ containing the $\infty$ section $[\infty]$ we have a
natural injective $q$-expansion map 
$\cO(X' \setminus [\infty])\h \to R((q))\h$, 
which we write as $G \mapsto G_{\infty}$. 
(See Section~\ref{S:delta-Fourier} for more details.) 

\begin{definition}
\label{D:standard}
An open set of the form $X' \setminus [\infty]$ with $X'$ as above
will be called {\em standard}.
\end{definition}

\begin{remark}
\label{irigutza} 
The following discussion is relevant 
in the context of Theorem~\ref{simairefined},
but will not be used.
The modular form $\theta^{p-1}\bar{f}$ is defined over $\F_p$. 
Hence, by Theorems 12.3.2 and 12.3.7 of~\cite{DI}, 
there exists a modular form $f_{p^2-p}$ over $\bZ$ on $\Gamma_1(N)$, 
of weight $p^2-p$, whose reduction mod $p$ is $\theta^{p-2}\bar{f}$.
In particular, the Fourier expansion $f_{p^2-p}(q) \in \bZ[[q]]$ satisfies
\[f_{p^2-p}(q) \equiv f^{(-1)}(q)\ \ mod\ \ p\]
in $\bZ_{(p)}[[q]]$.  
By the Fourier expansion principle over $\bF_p$, 
any modular form over $\bZ$ on
$\Gamma_1(N)$ of weight $p^2-p$ whose Fourier expansion is
congruent to $f^{(-1)}(q)$ mod $p$ 
has reduction mod $p$ equal to $\theta^{p-1} \bar{f}$.
\end{remark}

Let $j(x) \in k$ be the $j$-invariant of $x \in \overline{Y_1(N)}(k)$.

\begin{theorem}[Explicit reciprocity functions for $\CL$ points]
\label{refined} 
Assume, in Theorem~\ref{mainth}, 
that $X=S=X_1(N)$, $\Pi=\Id$, and $\Phi$ is a modular parametrization
attached to a newform $f$.
Then one can choose $X^{\dug}$ and $\Sigmat$ in Theorem~\ref{mainth} 
such that
\begin{enumerate}
\item 
$X^{\dug}$ is standard and 
$\bar{X}^{\dug}=\overline{Y_1(N)}^{\ord}\setminus \{x \mid j(x)=0,1728\}$.
\item
If $A_R$ is not $\CL$ then $\Sigmat_{\infty}=f^{(-1)}(q)$. 
In particular, $\overline{\Sigmat}=\bar{f}^{(-1)}$.
\item 
If $A_R$ is $\CL$ then $\Sigmat_{\infty}=-uf^{(-1)}_{[u]}(q)$. 
In particular,
$(\overline{\Sigmat})^p-\bar{a}_p \overline{\Sigmat}=\bar{f}^{(-1)}$.
\end{enumerate}
\end{theorem}

In both cases, (2) and~(3), the function $\overline{\Sigmat}$ is
integral over the integrally closed ring
$\cO(\overline{X_1(N)}^{\ord})$ and belongs to the fraction field
of $\cO(\overline{X_1(N)}^{\ord})$.  
So $\overline{\Sigmat} \in \cO(\overline{X_1(N)}^{\ord})$. 
Theorem~\ref{refined} will be proved in Section~\ref{S:local}.

\begin{remark}
If $A_{R}$ is $\CL$, then assertion $3$ in
Theorem~\ref{refined} implies that $\overline{f^{(-1)}_{[u]}(q)}$ 
is the Fourier expansion
of a rational function on $\overline{X_1(N)}$, hence
of a quotient $\frac{\alpha}{\beta}$ where $\alpha,\beta \in M_{\nu}$
are modular forms defined over $k$ of some weight $\nu$.
Is there a direct argument for this?
\end{remark}

\begin{corollary}
\label{elll} Let $\Phi \colon X_1(N) \to A$ be a modular
parametrization and let
$\Gamma \leq A(R)$ be a subgroup with $r:=\rank_p^{A(R)}(\Gamma)<\infty$.
Then the set $\Phi(\CL) \cap \Gamma$ is finite of cardinality at most
\[\left[(2g-2+\nu)\cdot \frac{p^2-p}{2} \cdot p^{r}+2
\lambda \right]\lambda,\] where $g$ is the genus of
$X_1(N)$,  $\nu$ is the number of  cusps of $X_1(N)$, and
$\lambda$ is the degree of $X_1(N) \to X_1(1)$.
\end{corollary}

\begin{proof} 
By Theorem~\ref{refined} 
we have $d_1=1$, $d_2=\lambda$, and $d_3\leq 2\lambda$ in \eqref{estt}.
So it will be enough to check that
\begin{equation}
\label{ineqdebaza} \deg(\overline{\Sigmat}) \leq (2g-2+\nu)
\cdot \frac{p^2-p}{2}.
\end{equation}
Taking degrees in parts (2) and~(3) of Theorem~\ref{refined}
yields either
$\deg(\overline{\Sigmat}) = \deg(\bar{f}^{(-1)})$
or
$p \deg(\overline{\Sigmat}) = \deg(\bar{f}^{(-1)})$.
In both cases, $\deg(\overline{\Sigmat}) \le \deg(\bar{f}^{(-1)})$.
Now \eqref{ineqdebaza} follows from the fact that the numerator and
denominator of the fraction in \eqref{fmu} are sections of the line bundle
$(\Omega^1(\text{cusps}))^{\frac{p^2-p}{2}}$, where $\Omega^1$ is the
cotangent bundle on $\overline{X_1(N)}$.
\end{proof}

We next discuss a uniqueness property for the function $\Sigmat$
in Theorem~\ref{refined}. 
Let $S=X_1(N)$, let $X^{\dug} \subset S$
be a standard Zariski open subset over $R$ such that
\begin{equation}
\label{faraj} \bar{X}^{\dug} \subset \overline{Y_1(N)}^{\ord}
\setminus \{x|j(x)=0,1728\}\end{equation}
 and define
\begin{equation}
\label{cineiP}
 {\mathcal P} :=\{P \in \CL \mid 
		\text{$\bar{P}$ is not in the isogeny class of any
		of the $k$-points of $Y_1(N) \setminus X^{\dug}$} \}.
\end{equation}
Clearly $\overline{\mathcal P}$ is infinite. 
We have $\calP \subset X^{\dug}(\cO_{M,\wp})$ by~\eqref{erasamorr}.
Let $f=\sum a_n q^n$ be a newform.
Let $\sum P_i^{(l)}-a_lP$ be the Hecke divisor on $S(\Qbar)$
associated to any $P\in {\mathcal P}$ 
and any prime $l \neq p$ (see Section~\ref{S:Hecke correspondences}).
Then $P_i^{(l)} \in \CL \cap X^{\dug}(\cO_{M,\wp})$. 
For $d \in (\Z/N\Z)^\times$,
let $\langle d \rangle$ be the diamond operator acting on $\overline{X_1(N)}$
and on $\cO(\overline{X_1(N)}^{\ord})$.
Consider the
$k$-linear space 
\begin{equation} 
\label{raade}
	\cF:=\left\{\overline{\Theta}\in \cO(\overline{X_1(N)}^{\ord}) \mid
	\<d\>\overline{\Theta}=\overline{\Theta} 
	\text{ for all $d \in (\bZ/N\bZ)^{\times}$ and } 
	U \overline{\Theta}(q)=0 \right\},
\end{equation}
where $\overline{\Theta}(q) \in k[[q]]$ 
is the Fourier expansion of $\overline{\Theta}$.
Note that $\bar{f}^{(-1)} \in \cF$.

\begin{theorem}[Uniqueness of reciprocity functions for CL points]
\label{uunnii} Let $S=X_1(N)$, and let $\Phi \colon S \to A$ be a
modular parametrization attached to a newform $f=\sum a_nq^n$ and
let $p$ be a sufficiently large good prime. Assume that $X^{\dug}
\subset S$ is a Zariski open subset over $R$ as in 
\eqref{faraj}.  Let $\cP$ be as in \eqref{cineiP}.
Then the following conditions on $\overline{\Theta} \in \cF$
are equivalent.

1) For any $P_1,\ldots,P_n \in \CL \cap X^{\dug}(R)$ and any
integers $m_1,\ldots,m_n$ we have \[ \sum_{i=1}^n m_i\Phi(P_i) \in
A(R)\pdiv \Longrightarrow
  \sum_{i=1}^n m_i \overline{\Theta}(\bar{P}_i) =0\in k.\]

2) For any $P \in \cP$ and any prime $l \neq p$ we have
\[\sum \overline{\Theta}(\bar{P}_i^{(l)})-a_l\overline{\Theta}(\bar{P})=0.\]

3) $\overline{\Theta}=\bar{\lambda} \cdot \overline{f^{(-1)}}$ for
some $\bar{\lambda} \in k$.
\end{theorem}

\begin{proof}
Condition 1 implies condition 2 by \eqref{heckerel}.
That condition 2 implies condition 3 will be proved in
Section~\ref{S:local}: see Lemma \ref{nuzzi}. Finally
condition 3 implies condition 1 by Theorem~\ref{refined}.
\end{proof}

\subsection{$\Rec$ functions and finiteness for isogeny classes} 
Fix a set $\Sigma$ of rational primes.

Suppose that $S=X_1(N)$.
Let $B$ be a $\bZ[1/N]$-algebra.
Let $Q$ be a $B$-point of $Y_1(N)$, represented by $(E_Q,\alpha_Q)$. 
The $\Sigma$-{\it isogeny class} 
(respectively, the {\it prime-to-$\Sigma$ isogeny class}) 
of $Q$ in $S(B)$ is the set $C=C_Q \subset S(B)$ 
of all $B$-points of $Y_1(N)$ represented by
$(E_{Q'}, \alpha_{Q'})$ such that there exists an isogeny $E_Q
\to E_{Q'}$ of degree divisible only by primes in $\Sigma$
(respectively, outside $\Sigma$).
We do not require the isogeny to be compatible with $\alpha_Q$ 
and $\alpha_{Q'}$.

The definition for $S=X^D(\cU)$ is similar.
Let $B$ be a $\bZ[1/m]$-algebra.
Let $Q \in S(B)$ be represented by $(E_Q,i_Q,\alpha_Q)$. 
The $\Sigma$-{\it isogeny class} 
(respectively, the {\it prime-to-$\Sigma$ isogeny class}) of
$Q$ in $S(B)$ is the set $C=C_Q \subset S(B)$ of all $B$-points of
$S$ represented by $(E_{Q'}, i_{Q'},\alpha_{Q'})$ such that
there exists an isogeny $E_Q \to E_{Q'}$, compatible with the
$\cO_D$-action, and of degree divisible only by primes in $\Sigma$
(respectively, outside $\Sigma$).
Again the isogeny need not respect $\alpha_Q$ and $\alpha_{Q'}$.

Let now $S$ be either $X_1(N)$ or $X^D(\cU)$ and let $C$ be a
$\Sigma$-isogeny class where $p \notin \Sigma$ or a prime to
$\Sigma$ isogeny class where $p \in \Sigma$. 
Say that $C$ is
{\em ordinary} (respectively {\em CL}) 
if it contains an ordinary point (respectively a CL point); 
in this case all points in $C$ are ordinary (respectively, CL).

\begin{theorem}[Reciprocity functions mod $p$ for isogeny classes]
\label{isoggg} Assume that $\thecorr$ is a  modular-elliptic or 
Shimura-elliptic correspondence, assume that $p$ is a sufficiently
large good prime, and assume $C$ is an
ordinary prime-to-$p$ isogeny class in $S(R)$. Then there exist
an affine Zariski open subset $X^{\dug} \subset X$, a (not
necessarily connected) finite \'{e}tale cover
 $\pi\colon \bar{X}^{\ddug} \to \bar{X}^{\dug}$
of degree $p$, a regular function $\overline{\Sigmatt} \in
\cO(\bar{X}^{\ddug})$ that is non-constant on each component of
$\bar{X}^{\ddug}$, and a map $\sigma\colon \Pi^{-1}(C) \cap
X^{\dug}(R)
 \ra
\bar{X}^{\ddug}(k)$
such that $\pi(\sigma(P))=\bar{P}$ for all $P$,
and for any $P_1,\ldots,P_n \in \Pi^{-1}(C) \cap X^{\dug}(R)$
and any $m_1,\ldots,m_n \in \bZ$ we have
\begin{equation}
\label{trebuie}
\sum_{i=1}^n m_i\Phi(P_i) \in A(R)\pdiv
\ \ \ \ \Longleftrightarrow \ \ \ \ \sum_{i=1}^n m_i
\overline{\Sigmatt}(\sigma(P_i)) =0\in k.\end{equation}
\end{theorem}

Theorem~\ref{isoggg} will be proved in Section~\ref{S:local}.

 \begin{remark} 

 1) Again, as the proof will show,
 the maps $\overline{\Sigmatt}$ and $\sigma$ will have a functorial nature.
In Theorem~\ref{isoggg} $\sigma$ is simply a map of sets,
but the proof will show that $\sigma$ has actually an
algebro-geometric flavor.

2) Theorem~\ref{isoggg} is an analogue of the second equivalence
in Theorem~\ref{mainth}.
Is there also an isogeny-class analogue of the first equivalence
in Theorem~\ref{mainth}?

3) The sum in the right half of \eqref{trebuie} may be viewed as 
a function $\eta^{\ddug}$ on $\overline{X^{\ddug}}^n$ 
evaluated at $(\sigma(P_1),\ldots,\sigma(P_n))$.
If the value is zero, then so is $\eta^{\dug}(\bar{P}_1,\ldots,\bar{P}_n)$,
where $\eta^{\dug}$ is the norm of $\eta^{\ddug}$
in the degree-$p^n$ extension 
$\OO\left( \overline{X^{\ddug}}^n \right)$ 
of $\OO\left( \overline{X^{\dug}} ^n\right)$.
Here $\eta^{\dug}$ may be expressed as a polynomial in the $m_i$
and the coefficients of the characteristic polynomial of
multiplication-by-$\Phi^{\ddug}$ on the locally free 
$\cO(\overline{X^{\dug}})$-algebra
 $\cO(\overline{X^{\ddug}})$.
Thus the left half of \eqref{trebuie} implies a statement 
expressible in terms of evaluation of functions on $\overline{X^{\dug}}$
instead of $\overline{X^{\ddug}}$.
Theorem~\ref{refined2}(4) will show
that $\eta^{\dug}$ is not always zero (consider the case $n=1$, for example),
so the statement is not always vacuous.
\end{remark}

Theorem~\ref{isoggg} trivially implies

\begin{corollary}
\label{trid}
 In the notation of Theorem~\ref{isoggg}
  we have
\[\rank_p^{A(R)} \left(\sum_{i=1}^n \bZ \cdot  \Phi(P_i)\right)=
 \dim_{\bF_p} \left( \sum_{i=1}^n \bF_p \cdot
 \overline{\Sigmatt}(\sigma(P_i))\right).\]
 \end{corollary}

Just as Corollary~\ref{cor1} implied Corollary~\ref{coru},
Corollary \ref{trid} applied to subsets $\{P_1,\ldots,P_n\}$
of $\Phi^{-1}(\Gamma) \intersect \Pi^{-1}(C) \intersect X^{\dug}(R)$ 
implies

 \begin{corollary}
\label{tridd} Assume $\thecorr$ is a  modular-elliptic  or a
Shimura-elliptic correspondence  and assume $p$ is a sufficiently
big, good prime. Let $C$
 be an
ordinary  prime-to-$p$ isogeny class in $S(R)$ and let $\Gamma \leq
A(R)$ be a subgroup with $r:=\rank_p^{A(R)}(\Gamma)<\infty$.
 Then the set
 $\overline{\Phi(\Pi^{-1}(C))\cap \Gamma} \subseteq A(k)$
 is finite of cardinality at most $c p^{r}$
 where $c$ is a constant not depending on $\Gamma$. In
 particular, the set $\Phi(\Pi^{-1}(C))\cap A(R)_{\tors}$ is finite.
 \end{corollary}

The first conclusion of Corollary~\ref{tridd} implies the last
because the reduction map $A(R)_{\tors} \to A(k)$ is injective 
for large $p$.

One can ask if the set $\Phi(\Pi^{-1}(C))\cap \Gamma$ is finite
for every $\Gamma$ with $\rank_p^{A(R)}(\Gamma)<\infty$. 
We prove a result of this type for certain $\Sigma$-isogeny classes
instead of prime-to-$p$ isogeny classes:

\begin{corollary}
\label{bizzet} Assume that $\thecorr$ is a modular-elliptic or a
Shimura-elliptic correspondence and that $p$ is a sufficiently
large good prime.
Let $Q \in S(R)$ be an ordinary point.
Let $\Sigma$ be the set of all rational primes that are inert in the
imaginary quadratic field $\cK_Q$ (so $p \notin \Sigma$).
Let $C$ be the $\Sigma$-isogeny class of $Q$ in $S(R)$. 
Let $\Gamma \leq A(R)$ be a subgroup 
with $r:=\rank_p^{A(R)}(\Gamma)<\infty$. 
Then the set $\Phi(\Pi^{-1}(C))\cap \Gamma$
is finite of cardinality at most $c p^{r}$,
where $c$ is a constant not depending on $\Gamma$.
\end{corollary}

The proof of Corollary \ref{bizzet} is not straightforward and
will be given in Section~\ref{S:local}.

\subsection{Refinement of results on isogeny classes for 
modular parametrizations}  
\label{S:refinement for isogeny classes}

Suppose that $\thecorr$ arises from a newform $f=\sum a_n q^n$.
Our goal in this subsection is to state Theorem~\ref{refined2}, 
which describes the cover $\overline{X^{\ddug}}$ and 
the function $\overline{\Sigmatt}$ explicitly in this case.

Let $I_1(N)$ be the Igusa curve from pp.~460--461 of~\cite{gross},
except that we view $I_1(N)$ as a smooth projective integral curve.
It is a Galois cover of $\overline{X_1(N)}$ ramified only over
supersingular points,
and the Galois group is naturally isomorphic to $\F_p^\times$.
Let $J:=I_1(N)/\langle -1 \rangle$ be the intermediate cover
of degree $(p-1)/2$ obtained by taking the quotient of $I_1(N)$
by the involution corresponding to $-1 \in \F_p^\times$.
We will describe $\overline{X^{\ddug}}$ in terms of $J$.
There is a point $\infty$ on each of these covers that is unramified
over $\infty \in \overline{X_1(N)}$.
In particular, rational functions on $I_1(N)$ and $J$
have Fourier expansions in $k((q))$.

Let 
\begin{equation}
\label{fructnou1}
 f^{(0)}(q) := \sum_{(n,p)=1} a_n q^n \in
 \bZ_p[[q]].\end{equation}
(The series $f^{(0)}(q)$ is called $f|R_{0}$ in~\cite[p.~115]{serre}.)
Let
\begin{equation}
\label {fructnou2}
 f^{(0)}_{[a_p]}(q) := \left( \sum_{i=0}^{\infty} a_p^i V^i\right)
 f^{(0)}(q) =\sum_{i=0}^{\infty} \sum_{(n,p)=1} a_p^i a_n q^{np^i} \in
 \bZ_p[[q]].\end{equation}
Corollary~\ref{C:series are modular}
and Lemma~\ref{L:series is modular function}
will show that for $p \gg 0$, the series $\overline{f^{(0)}_{[a_p]}(q)}$
is the Fourier expansion of some $\eta \in k(J)$.
For a constant $\bar{\lambda} \in k$ to be specified later, define
\[
\overline{\Phi^{\dug \dug}}:=
\begin{cases}
\bar{\lambda} \eta^{p^2} - \bar{a}_p \eta^p, &\text{ if $A_{R}$ is not CL} \\
\eta^p, &\text{ if $A_{R}$ is CL.}
\end{cases}
\]

\begin{theorem}[Explicit reciprocity functions mod $p$ for isogeny classes]
\label{refined2}
 Assume, in Theorem~\ref{isoggg},
that $\thecorr$ arises from  a modular
parametrization attached to a newform $f$ on $\Gamma_0(N)$.
Then there exists $\bar{\lambda} \in k^{\times}$
such that  $X^{\dug}$, $\overline{X^{\ddug}}$, and $\overline{\Sigmatt}$ 
can be chosen to satisfy:

1) 
The cover $\overline{X^{\ddug}}$ of $\overline{X^{\dug}}$ 
is a disjoint union
$\overline{X^0} \coprod \overline{X^+} \coprod \overline{X^-}$,
where $\overline{X^0} \simeq \overline{X^{\dug}}$ is the trivial cover
and $\overline{X^+}$ and $\overline{X^-}$ are each isomorphic to 
the inverse image of $\overline{X^{\dug}}$ under $J \to \overline{X_1(N)}$.

2) The restrictions of $\overline{\Sigmatt}$ 
to $\overline{X^0}$, $\overline{X^+}$, $\overline{X^-}$ equal
\[
	\overline{\Phi^{\dug}}, \qquad 
	\overline{\Phi^{\dug}}+\lambda_{+} \overline{\Phi^{\dug \dug}}, \qquad
	\overline{\Phi^{\dug}}+\lambda_{-} \overline{\Phi^{\dug \dug}},
\]
respectively, where $\lambda_{\pm} \in k$ are such that 
$\lambda_{+}^{(p-1)/2},\lambda_{-}^{(p-1)/2}$
are the two square roots of $\bar{\lambda}$.
\end{theorem}

Theorem~\ref{refined2} will be proved in Section~\ref{S:local}.

\begin{corollary}
Let notation be as in Theorem~\ref{refined2}.
The characteristic polynomial of 
the endomorphism ``multiplication by $\overline{\Phi^{\ddug}}$'' 
in the locally free $\cO(\overline{X^{\dug}})$-algebra
 $\cO(\overline{X^{\ddug}})$ is
 \[
x^p-\bar{\lambda} h^2 x + (\bar{\lambda} h^2 \overline{\Phi^{\dug}}- (\overline{\Phi^{\dug}})^p),\]
where $h:=\left(\overline{\Phi^{\dug \dug}}\right)^{(p-1)/2} 
\in k(\overline{X_1(N)})$.
\end{corollary}

\begin{proof}
The characteristic polynomial of 
$\overline{\Phi^{\ddug}} - \overline{\Phi^{\dug}}$
equals
\[
	x \left( x^{(p-1)/2} - \lambda_+^{(p-1)/2} \overline{\Phi^{\dug \dug}}^{(p-1)/2} \right) \left( x^{(p-1)/2} - \lambda_-^{(p-1)/2} \overline{\Phi^{\dug \dug}}^{(p-1)/2} \right) = x^p - \bar{\lambda} h^2 x.
\]
In this, replace $x$ by $x-\overline{\Phi^{\dug}}$.
\end{proof}

\subsection{Strategy of proofs}
The proof of our local results
 will be an application of the theory of
 $\d$-characters~\cite{char,frob} and $\d$-modular
 forms ~\cite{difmod,shimura}. These two types
  of objects are special cases of {\it arithmetic differential equations}
  in the sense of ~\cite{book}.
Section~\ref{S:local} reviews the facts from this theory 
that are necessary for the proof.
As a sample of our strategy
let us explain, very roughly, the idea of our proof of Theorem~\ref{mainth}.
Assume for simplicity that we are dealing with 
a modular-elliptic correspondence $\thecorr$ arising from  a modular
parametrization attached to a newform $f$. Following~\cite{char}
 consider the
{\it Fermat quotient operator} $\d \colon R \to R$ defined by $\d
x:=(\phi(x)-x^p)/p$, where $\phi \colon R \to R$ is the lift of
Frobenius.
We view $\d$ as an analogue of a derivation operator
with respect to $p$. Recall from~\cite{char} that
 if $Y$ is any smooth scheme  over
 $R$ then a function $g \colon Y(R) \to R$ is called a
 $\d$-{\it function of order}
$r$ if it is Zariski locally 
of the form $P \mapsto G(x, \d x,\ldots,\d^r x)$, 
where $G$ is a restricted power series with $R$-coefficients 
and $x \in R^N$ is a tuple of affine coordinates
of $P$ in some $N$-dimensional affine space. 
If $A$ is our elliptic curve then by~\cite{char} there
exists a $\d$-function of order $2$,
$\psi\colon A(R) \to R$, which is also a group homomorphism;
$\psi$ is called in~\cite{char} a {\em $\d$-character} 
and may be viewed as an arithmetic
analogue of the ``Manin map''~\cite{man, man1}.
Consider the composition $f^{\sharp}=\psi \circ \Phi \colon X(R) \to R$.
On the other hand, the theory of
$\d$-modular forms~\cite{difmod} yields an open subset $X^{\dug}$ of $S$
and a $\d$-function of order~$1$, $f^{\flat} \colon X^{\dug}(R) \to R$, 
that vanishes at all CL-points. 
Then we prove that there exist $\d$-functions of order~$2$, 
denoted $h_0,h_1:X^{\dug}(R) \to R$, such that the $\d$-function
\[\Sigmat:=f^{\sharp}-h_0 \cdot f^{\flat}-h_1 \cdot \d \circ f^{\flat}\]
has order $0$, or equivalently is a formal function in the usual
sense of algebraic geometry. (Intuitively, in the system of
``arithmetic differential equations'' $f^{\sharp}=f^{\flat}=0$ one
can eliminate all the ``derivatives'' of the unknowns.) It follows
that $f^{\sharp}$ and $\Sigmat$ have the same value at each
CL-point $P_i$. So 
\[\sum m_i \Sigmat(P_i)=\sum m_i f^{\sharp} (P_i)=\psi(\sum m_i \Phi(P_i)).\]
By the arithmetic
analogue in~\cite{char,frob} of Manin's Theorem of the Kernel
\cite{man,man1}, $\psi(\sum m_i \Phi(P_i))$ vanishes if and only
if $\sum m_i \Phi(P_i)$ is torsion. (Actually for our application
to Corollaries \ref{coru} and \ref{cor2} we need only the ``if''
part, which does not require the analogue of the Theorem of
the Kernel.) On the other hand we will check that
$\overline{\Sigmat} \notin k$ by looking at Fourier
$q$-expansions, and this will end the proof of the first
equivalence in Theorem~\ref{mainth} in the special case we considered.

In particular, our proof of the (effective) finiteness of
$\Phi(\CL) \cap \Gamma$ in the case $\Gamma=A(R)_{\tors}$ 
can be intuitively described as follows. The points of $\CL$ are
solutions of the ``arithmetic differential equation''
$f^{\flat}=0$ whereas the points of $\Phi^{-1}(\Gamma)$ are
solutions of the ``arithmetic differential equation''
$f^{\sharp}=0$. Hence the points of $\CL \cap \Phi^{-1}(\Gamma)$
are solutions of the system of ``arithmetic differential
equations'' $f^{\flat}=f^{\sharp}=0$. By what was said above  one
can eliminate, in this system,  the ``derivatives'' of the
unknowns and hence one is left with a (non-differential) algebraic
equation mod $p$, whose ``degree'' can be estimated. There are
only finitely many solutions to this algebraic equation and their
number is effectively bounded by the ``degree".

\section{Proofs of global results, I}
\label{S:global}

In this section we prove some of our global results. The rest of
them will be proved in Section~\ref{S:global 2}, as a consequence
of the local theory to be developed in Section~\ref{S:local}.

\subsection{Proofs using equidistribution}
\label{S:equidistribution}

We begin with Theorem~\ref{T:global2}; for its proof we need some
measure-theoretic prerequisites.

\begin{lemma}
\label{L:tube} Let $S$ be a smooth projective curve over $\C$. Let
$X$ be a (possibly singular) closed $N$-dimensional
subvariety of $\PP^n_\C$. Let $\pi \colon X \to S$ be a morphism.
Let $s \in S(\C)$. Equip $S(\C)$ and $\PP^n(\C)$ with real
analytic Riemannian metrics. Let $B_r$ be the open disk
in $S$ with center $s$ and radius $r$, and let $B_r'=B_r - \{s\}$. 
Then there exists $\delta>0$ such that the $N$-dimensional volume of
$\pi^{-1}(B_r')$ with respect to the metric on $\PP^n(\C)$ is
$O(r^{\delta})$ as $r \to 0$.
\end{lemma}

\begin{proof}
Define $\Delta:=\{z \in \C : |z| < 1\}$
and $\blacktriangle:=\{z \in \C : |z| \le 1/2 \}$.
Let $g_{\PP}$ and $g_S$ be the given metrics on $\PP^n(\C)$ and $S(\C)$.
Let $\mu$ be Lebesgue measure on $\C^N$.
Fix a holomorphic chart $\iota\colon \Delta \to S(\C)$
mapping $0$ to $s$.

We may assume that $\dim \pi^{-1}(s) < N$.
By work of Hironaka, 
there exists a desingularization $p \colon Y \to X$ 
(see \cite[Corollary~3.22]{Kollar})
and we may assume that the fiber of the map
$f:=\pi \circ p \colon Y \to S$ above $s$ 
is a simple normal crossing divisor (see \cite[Theorem~3.21]{Kollar}).
Then for each $y \in f^{-1}(s)$, there exists a holomorphic chart
$\iota_Y \colon \Delta^N \to Y(\C)$ mapping $0$ to $y$ 
such that $f$ is given with respect to $\iota_Y$ and $\iota_S$ by
\begin{align*}
	h\colon \Delta^N &\to \Delta \\
	z = (z_1,\ldots,z_N) &\mapsto u(z) z_1^{e_1} \cdots z_N^{e_N}
\end{align*}
for some nonvanishing holomorphic function $u \colon \Delta^N \to \C$
and $e_i \in \Z_{\ge 0}$.
By compactness, there exist $\varepsilon>0$ and finitely many $\iota_Y$
such that the sets $\iota_Y(\blacktriangle^N)$ cover
$f^{-1}(B_\varepsilon)$.

Since $\blacktriangle$ is compact, 
$\left.\left(\iota_S^* g_S \right)\right|_{\blacktriangle}$ 
is bounded above and below by positive constants 
times the standard metric.
Similarly, the pullback of $g_{\PP}$ to $\blacktriangle^N$
is bounded above in terms of the standard metric.
Thus we reduce to showing that for each $\iota_Y$,
there exists $\delta>0$ such that
\[
	\mu\left(\{z \in \blacktriangle^N : |h(z)| < r \}\right) = O(r^{\delta})
\]
as $r \to 0$.

Let $u_{\min} := \inf\{|u(z)| :  z\in \blacktriangle^N \} > 0$.
We may assume that $r<u_{\min}$.
Fix $E > \sum e_i$.
Let $\rho := (r/u_{\min})^{1/E} < 1$.
If $|z_i| \ge \rho$ for all $i$,
then $|h(z)| \ge u_{\min} \rho^E = r$.
Thus
\[
	\{z \in \blacktriangle^N : |h(z)| < r \} \subseteq  \{z \in \blacktriangle^N : |z_i| < \rho \text{ for some $i$} \}.
\]
The volume of the latter is $O(\rho^2) = O(r^{2/E})$ as $r \to 0$.
\end{proof}

\begin{lemma}
\label{L:Hilbert continuity}
Let $Y,H$ be varieties over $\C$, with $Y$ proper.
Let $X$ be a closed subvariety of $Y \times H$.
Let $(h_i)$ be a sequence in $H(\C)$ converging to $h_\infty$.
For $i \le \infty$, let $X_i$ be the fiber of $X \to H$ above $h_i$.
View $X_i$ as a subvariety of $Y$.
Then any open neighborhood $N$ of $X_\infty(\C)$ in the
complex topology contains $X_i(\C)$ for all sufficiently large $i$.
\end{lemma}

\begin{proof}
The open set $((Y \times H)-X)(\C) \union (N \times H(\C))$ contains
$Y(\C) \times \{h_\infty\}$,
so it contains also $Y(\C) \times U$ for some open neighborhood $U$
of $h_\infty$ in $H$,
by the ``tube lemma for compact spaces''
(Lemma~5.8 on p.~169 of~\cite{MunkresTopology}).
For large $i$, we have $h_i \in U$, and then $X_i(\C) \subseteq N$.
\end{proof}

{}From now on, we assume that $S=X_1(N)$ as in Theorem~\ref{T:global2}.
Choose a real analytic Riemannian metric on $S(\C)$.
Define $B_r'$ to be the punctured open disk in $S(\C)$
with center $\infty$ and radius $r$ with respect to the metric.
Let $\mu_{\mathcal H}$ be the probability measure on $S(\C)$
whose pullback to the upper half plane $\mathcal H$
equals a multiple of the hyperbolic measure $\frac{dx\,dy}{y^2}$.

We next show that $\mu_{\mathcal H}$
blows up relative to the Riemannian metric near the cusp $\infty$.
(The Riemannian volume of $B_r$ is only $O(r^2)$ as $r \to 0$.)

\begin{lemma}
\label{L:hyperbolic ball}
There exists $u>0$ such that for all sufficiently small $r>0$,
we have $\mu_{\mathcal H}(B_r) > u/\log(1/r)$.
\end{lemma}

\begin{proof}
Let $\tau$ be the usual parameter on $\mathcal H$.
Then $q:=e^{2\pi i \tau}$ is a uniformizer at $\infty$ on $X_1(1)$.
So there exists $c>0$ such that for all sufficiently small $r$,
in the fundamental domain in $\mathcal H$,
the part corresponding to $B_r$ contains the part
where $|q| < c r$.
The inequality is equivalent to $\im(\tau)>\frac{1}{2\pi}\log(1/(cr))$,
and so the hyperbolic measure is at least a constant times $1/\log(1/(cr))$
for sufficiently small $r$.
If $u$ is small enough,
this exceeds $u/\log(1/r)$ for all sufficiently small $r$.
\end{proof}

We will need also an equidistribution result for CM points.
The first such equidistribution result was proved in~\cite{Duke1988},
and this has been generalized in several directions by several authors:
see Section~5.4 of the survey paper~\cite{Michel-Venkatesh2006},
for instance.
The version we use is a special case of a result in~\cite{Zhang2005}.

\begin{lemma}
\label{L:CM equidistribution} Let $k$ be a finite extension of
$\Q$. Fix an embedding $\kbar \injects \C$. Let $S$ be a modular
curve $X_1(N)$ or a Shimura curve $X^D(\cU)$ over $\kbar$. Let
$(x_i)$ be an infinite sequence of distinct $\CM$-points in
$S(\kbar)$. The uniform probability measure on the
$\Gal(\kbar/k)$-orbit of $x_i$ converges weakly as $i \to \infty$
to the measure $\mu_{\mathcal H}$ on $S(\C)$.
\end{lemma}

\begin{proof}
This follows from Corollary~3.3 of~\cite{Zhang2005}.
Namely, we choose $\delta<1/2$ as on p.~3663 of~\cite{Zhang2005},
choose $\epsilon>0$ so that $\delta/2+1/4+\epsilon<1/2$,
and define the ``CM-suborbit'' $O(x_i)$ as
the $\Gal(\kbar/k)$-orbit of $x_i$.
The hypothesis of Corollary~3.3 of~\cite{Zhang2005}
is satisfied, by the Brauer-Siegel theorem (see\ the first remark
following Corollary~3.3 of~\cite{Zhang2005}).
\end{proof}

\begin{proof}[Proof of Theorem~\ref{T:global2}]
Let $A'$ be the image of $X \to A$.
By Corollary~9 of~\cite{Poonen1999}
(also proved partially independently as Theorem~1.2 of~\cite{Zhang2000})
applied to $A' \subseteq A$, we have that $A'$ is a coset.
We may translate to assume that $A'$ is an abelian subvariety,
and hence reduce to the case where $X \to A$ is surjective.
We may assume also that $X \to S$ is surjective, since $\dim S = 1$.
We want $X = S \times A$.
Suppose not.
Then $X \to A$ is generically finite, say of degree $d$.

The group $\Gamma$ is contained in
the division hull of a finitely generated group $\Gamma_0$.
Choose a number field $k \subset \Qbar$
such that $A,S,X$ are all defined over $k$
and $\Gamma_0 \le A(k)$.

Since $X(\Qbar) \intersect (\CM \times \Gamma_\epsilon)$
is Zariski dense in $X$
for every $\epsilon>0$,
and since $X$ has only countably many subvarieties,
we may choose a generic infinite sequence of points
$x_i=(s_i,a_i) \in X(\Qbar)$
with $s_i \in CM$ and $a_i \in \Gamma_{\epsilon_i}$
where $\epsilon_i \to 0$.
(``Generic'' means that each proper subvariety of $X$
contains at most finitely many $x_i$.)
In particular, each $s_i$ appears only finitely often.
Since class numbers of imaginary quadratic fields tend to infinity,
we have $[k(s_i):k] \to \infty$.
So $[k(x_i):k] \to \infty$.
For all but finitely many $i$, the $a_i$ lie in the open locus
above which the fibers of $X \to A$ have size $d$,
and then $[k(x_i):k] \le d [k(a_i):k]$.
Thus $[k(a_i):k] \to \infty$.

The $a_i$ form a sequence of almost division points relative to $k$
in the sense of~\cite{Zhang2000}.
By passing to a subsequence we may assume that they have a coherent
limit $(C,b+T)$ in the sense of~\cite{Zhang2000},
where $C$ is an abelian subvariety of $A$, and $b \in A(\C)/C(\C)$,
and $T$ is a finite set of torsion points of $A/C$.
Since $[k(a_i):k] \to \infty$, we have $\dim C > 0$
by definition of coherent limit.
By replacing $X$ by its image under
$S \times A \stackrel{(\id,\phi)}\longrightarrow S \times \tilde{A}$
for a suitable isogeny $\phi\colon A \to \tilde{A}$,
we may reduce to the case where $T=\{0\}$
and $A \isom B \times C$
for some abelian subvariety $B$ of $A$.
Identify $A/C$ with $B$.
Write $a_i=(b_i,c_i)$ with $b_i \in B(\Qbar)$ and $c_i \in C(\Qbar)$.
By definition of $T$, we have $b_i \in B(k)$.
By Theorem~1.1 of~\cite{Zhang2000},
the uniform probability measure on the orbit $\Gal(\kbar/k) a_i$
(supported on $\{b_i\} \times C(\C)$)
converges weakly as $i \to \infty$
to the $C(\C)$-invariant probability measure
on $\{b\} \times C(\C)$.
So the uniform probability measure on $\Gal(\kbar/k) c_i$
converges to Haar measure $\mu_C$ on $C(\C)$.

For each $i$, let $X_{b_i}$ be the fiber of the projection
$X \to B$ above $b_i$, viewed as a subvariety of $S \times C$.
Since the $b_i$ are generic in $B$,
we may discard finitely many
to assume that the $X_{b_i}$ have the same Hilbert polynomial
(with respect to some embedding $S \times C \injects \PP^N$)
and that the corresponding points of the Hilbert scheme $H$
converge in the complex topology; let $X_{b_\infty} \subseteq S \times C$
be the closed subscheme corresponding to the limit.
We have $\dim X_{b_i} < \dim (S \times C)$ for all finite $i$
(and hence also for $i=\infty$),
since otherwise by genericity of the $b_i$,
we would have $X = S \times B \times C = S \times A$.

Let $\pi_C\colon S \times B \times C \to \C$ be the projection.
Also, for $i \le \infty$, let $\pi_{S,i} \colon X_{b_i} \to S$
be the projection.

Choose a real analytic Riemannian metric on $C(\C)$
whose associated volume form equals $\mu_C$.
Let $g=\dim C$.
Let $B_r' = B_r - \{\infty\}$.
By Lemma~\ref{L:hyperbolic ball}, there exists $u>0$ such that
$\mu_{\mathcal H}(B_r') = \mu_{\mathcal H}(B_r) > u/\log(1/r)$
for all sufficiently small $r$.
On the other hand, Lemma~\ref{L:tube} implies that for some $\delta>0$,
the $g$-dimensional volume of $\pi_{S,\infty}^{-1}(B_r')$
is $O(r^\delta)$ as $r \to 0$.
Let $L_r:=\pi_C(\pi_{S,\infty}^{-1}(B_r'))$.
Projection onto $C$ can only decrease $g$-dimensional volume, so
$\mu_C(L_r) = O(r^\delta)$.
Thus we may fix $r>0$ such that $\mu_{\mathcal H}(B_r') > \mu_C(L_r)$.
Let $L=L_r$.
Fix a compact annulus $K \subseteq B_r'$ large enough so that
$\mu_{\mathcal H}(K) > \mu_C(L)$.

For a compact subset $M'$ of a metric space $M$,
let $N_\rho M'$ be the set of points in $M$ whose distance
to $M'$ is less than $\rho$.
Fix $\rho>0$ such that $N\rho K \subseteq B_r'$.
By Lemma~\ref{L:Hilbert continuity} with $Y=S \times C$,
with $H$ the Hilbert scheme, and $X$ the universal family in $Y \times H$,
we have $X_{b_i}(\C) \subseteq N_\rho X_{b_\infty}(\C)$
after discarding finitely many $i$.
In particular, every point of $\pi_{S,i}^{-1}(K)$
is within $\rho$ of a point of $X_{b_\infty}(\C)$.
The $S$-projections of the points of $X_{b_\infty}(\C)$ so used
are then within $\rho$ of $K$, so
\[
    \pi_{S,i}^{-1}(K) \subseteq \pi_{S,\infty}^{-1}(N_\rho K)
    \subseteq \pi_{S,\infty}^{-1}(B_r').
\]
Projecting to $C$, we obtain
\begin{equation}
\label{E:last}
    \pi_C(\pi_{S,i}^{-1}(K)) \subseteq L.
\end{equation}

Now as $i \to \infty$,
the fraction of points of $\Gal(\kbar/k)x_i$ whose $S$-projection
lies in $K$ tends to $\mu_{\mathcal H}(K)$
by Lemma~\ref{L:CM equidistribution},
and the fraction of points of $\Gal(\kbar/k)x_i$ whose $C$-projection
lies in $L$ tends to $\mu_C(L)$.
But \eqref{E:last} implies
that the first set of points is contained
in the second set of points,
so $\mu_{\mathcal H}(K) \le \mu_C(L)$,
contradicting the choice of $K$.
\end{proof}

For the proof of Theorem~\ref{T:global4}, we will need the following:

\begin{lemma}
\label{L:hyperbolic measure and Haar measure}
Let $\Phi\colon S \to A$ be a morphism from a Shimura curve
to an elliptic curve $A$ over $\C$.
Let $\mu_{\mathcal H}$ be the hyperbolic probability measure on $S(\C)$.
Let $\mu_A$ be the Haar probability measure on $A(\C)$.
Then $\Phi_* \mu_{\mathcal H} \ne \mu_A$.
\end{lemma}

\begin{proof}
By replacing $S$ with a finite cover, we may assume
that ${\mathcal H} \to S(\C)$ is unramified.
The universal cover of $A(\C)$ is not biholomorphic to ${\mathcal H}$,
so the composition ${\mathcal H} \to S(\C) \to A(\C)$ cannot
be unramified.  Hence $\Phi$ is ramified.
Pick $s \in S(\C)$ at which the ramification index $e$ is $>1$.
Let $a=\Phi(s)$.
Choose a Riemannian metric on $A(\C)$
inducing the Haar probability measure $\mu_A$.
Let $B_r$ be the disk of radius $r$ centered at $a$.
With respect to suitable uniformizing parameters, $\Phi$ near $s$
is equivalent to $z \mapsto z^e$,
so there exists $c>0$ such that
$\mu_{\mathcal H}(\Phi^{-1}(B_r)) > c \mu_A(B_r)^{1/e}$
for all sufficiently small $r$.
In particular, for sufficiently small $r$, we have
$\left(\Phi_* \mu_{\mathcal H} \right)(B_r) =
\mu_{\mathcal H}(\Phi^{-1}(B_r)) > \mu_A(B_r)$.
\end{proof}

\begin{proof}[Proof of Theorem~\ref{T:global4}]
As in the first three sentences of the proof of Theorem~\ref{T:global2},
we may use Corollary~9 of~\cite{Poonen1999}
to reduce to the case that $\Phi$ is surjective.
If $A=0$, there is nothing to show,
so we may assume that $A$ is an elliptic curve.

Choose a number field $k \subset \Qbar$
such that $A,S,X$ are all defined over $k$
and $\Gamma$ is contained in the division hull of $A(k)$.

If the conclusion fails, then there is an infinite sequence $(s_i)$
in $CM$ with $\Phi(s_i) \in \Gamma_{\epsilon_i}$ for some $\epsilon_i \to 0$.
Let $a_i=\Phi(s_i)$.
By Lemma~\ref{L:CM equidistribution},
the uniform probability measure on $\Gal(\kbar/k) s_i$
converges weakly to $\mu_{\mathcal H}$ on $S(\C)$.
It follows that the uniform probability measure on $\Gal(\kbar/k) a_i$
converges weakly to $\Phi_* \mu_{\mathcal H}$ on $A(\C)$.

On the other hand, $(a_i)$ is a sequence of almost division points.
By the previous paragraph, $[k(s_i):k] \to \infty$,
so $[k(a_i):k] \to \infty$.
Passing to a subsequence, we may assume that $(a_i)$ has a coherent limit,
which can only be $(A,\{0\})$, since $[k(a_i):k] \to \infty$.
By Theorem~1.1 of~\cite{Zhang2000},
the uniform probability measure on $\Gal(\kbar/k) a_i$
converges weakly to the Haar measure $\mu_A$ on $A(\C)$.

The previous two paragraphs imply that
$\Phi_* \mu_{\mathcal H} = \mu_A$,
contradicting
Lemma~\ref{L:hyperbolic measure and Haar measure}.
\end{proof}

\subsection{Proofs using Hecke divisors}
\label{S:Hecke divisors}

\begin{lemma}
\label{non} 
Let $S=X_1(N)$.
Let  $f=\sum a_n q^n$ be a weight-$2$ newform on $\Gamma_1(N)$.
let $C \subset S(\Qbar)$ be an isogeny class.
Let $\Sigmat$ be a rational function on $S$ 
none of whose poles are in $C$. 
Assume that for infinitely many primes $l$ and for any $P \in C$ 
we have
\begin{equation}
\label{gura1} 
	\sum_i \Sigmat(P_i^{(l)})-a_l\Sigmat(P)=0.
\end{equation} 
Then $\Sigmat=0$.
\end{lemma}

\begin{proof}
Assume that $\Sigmat \neq 0$.
The function
\[
	(T(l)\Sigmat)(x):=\sum \Sigmat(x_i^{(l)}), 
\]
defined for all but finitely many $x \in S(\bC)$, 
is a rational function on $S$ by~\cite[p.~55]{GACC}.
For the infinitely many given $l$, 
the rational functions $T(l)\Sigmat$ and $\Sigmat$ 
agree on the infinite set $C$ so they coincide.
Since $\Sigmat$ may be viewed as a ratio of modular forms over $\Qbar$,
each of which is a $\Qbar$-linear combination 
of newforms whose Fourier coefficients are algebraic integers,
the Fourier expansion $\varphi(q)$ of $\Sigmat$
is in $\OO_{K,\calS}((q))$ for some ring of $\calS$-integers in some number
field $K$, with $\calS$ finite.
We may restrict attention to primes $l \nmid N$
not lying under any prime in $\calS$.
We may assume also that the leading coefficient of $\varphi(q)$
is prime to $l$.
The $q$-values corresponding to the elliptic curves $l$-isogenous 
to the one corresponding to $q$ itself are $q^l$ and the $l$-th roots of $q$,
so taking Fourier expansions in $T(l)\Sigmat = \Sigmat$ yields
\begin{equation}
\label{E:Hecke}
	\varphi(q^l) + \sum_{b=0}^{l-1} \varphi(\zeta^b q^{1/l}) 
	= a_l \varphi(q),
\end{equation}
where $\zeta$ is a primitive $l$-th root of $1$.
Let $v_q$ be the valuation on $\Qbar((q))$.
Comparing leading terms in \eqref{E:Hecke} yields $v_q(\varphi) \ge 0$;
and if $v_q(\varphi)=0$, then $l+1=a_l$, 
which contradicts $|a_l| \le 2\sqrt{l} < l+1$.
Thus $v_q(\varphi)>0$.

The series $\sum_{b=0}^{l-1} \varphi(\zeta^b q^{1/l})$ is divisible by $l$,
so
\begin{equation}
\label{E:Hecke mod l}
	\varphi(q^l) \equiv a_l \varphi(q) \pmod{l \OO_{K,\calS}[[q]].}
\end{equation}
The leading coefficient of $\varphi(q^l)$ equals that of $\varphi(q)$,
so it is prime to $l$.
Then \eqref{E:Hecke mod l} shows that $a_l$ is prime to $l$.
Now \eqref{E:Hecke mod l} contradicts $v_q(\varphi)>0$.
\end{proof}

\begin{proof}[Proof of Theorem~\ref{nonu}]
Extend $\Phi$ linearly to
a homomorphism $\Phi_*\colon \Div^0(X_1(N)(\Qbar)) \to A(\Qbar)$.
Then $\Phi_* \circ T(l)_*=a_l \cdot \Phi_*$; see~\cite[p.~242]{DS}.
For any point $P \in C$  we have $T(l)_*(P-\infty)=\sum P_i^{(l)} -\sum
P_{i0}^{(l)}$ with $P_{i0}^{(l)}$ cusps.  We get
\begin{equation}
\begin{array}{rcl} \label{mar}
a_l \cdot \Phi(P) & = & a_l(\Phi_*(P-\infty))\\
\  & \  & \  \\
\  & = & \Phi_*(T(l)_*(P-\infty))\\
\  & \  & \  \\
\  & = & \Phi_*(\sum P_i^{(l)} -\sum P_{i0}^{(l)})\\
\  & \  & \  \\
\  & = & \sum \Phi(P_i^{(l)}) -\sum \Phi(P_{i0}^{(l)})
\end{array}
\end{equation}
By the Manin-Drinfeld theorem (see~\cite[p.~62]{lang}, for instance),
$\Phi(P_{i0}^{(l)}) \in A(\Qbar)_{\tors}$,
so \eqref{mar} yields
\begin{equation}
\label{heckerel} \sum \Phi(P_i^{(l)})-a_l \cdot \Phi(P)\in
A(\Qbar)_{\tors}.
\end{equation}
By \eqref{vivald}, 
we obtain $\sum_i \Sigmat(P_i^{(l)})-a_l \cdot \Sigmat(P)=0$.
Now Lemma~\ref{non} implies $\Sigmat = 0$.
\end{proof}

\begin{proof}[Proof of Theorem~\ref{triples}]
Without loss of generality, $p \nmid a_1$.
To prove that $\Phi^{\dug}$ is constant,
it will suffice to show that $\Phi^{\dug}$ is regular at every $P \in X(k)$.

Fix $P$.
Let $Y$ be the inverse image of $\{0\}$ under the morphism
\begin{align*}
	\beta \colon X \times (X^{\dug})^{n-1} &\to A \\
	(P_1,\ldots,P_n) & \mapsto \sum a_i \Phi(P_i).
\end{align*}
Let $\pi_i \colon Y \to X$ be the $i$-th projection.
The morphism $\pi_1 \colon Y \to X$ is surjective
since given $P_1$, if we choose $P_4,\ldots,P_n \in X^{\dug}(k)$ arbitrarily,
then there are only finitely many choices of $P_2 \in X^{\dug}$
such that the equation $\beta(P_1,\ldots,P_n)=0$ forces $P_3 \notin X^{\dug}$.
In particular, we can find a smooth irreducible curve $C$ 
and a morphism $\gamma \colon C \to Y$
such that $\pi_1(\gamma(C))$
is a dense subset of $X$ containing $P$.

By \eqref{ticc}, we have $\sum a_i \Phi^{\dug}(P_i) = 0$ for 
all $(P_1,\ldots,P_n) \in Y \intersect (X^{\dug})^n$.
In particular,
\[
	\sum_{i=1}^n a_i \Phi^{\dug}(\pi_i(\gamma(c))) = 0
\]
is an identity of rational functions of $c \in C$.
Since $\Phi^{\dug}$ is regular on $X^{\dug}$,
the last $n-1$ summands are regular on $C$.
Therefore the first summand is regular too.
So $a_1 \Phi^{\dug}$ is regular on $\pi_1(\gamma(C))$.
Since $a_1 \ne 0$ in $k$, and $P \in \pi_1(\gamma(C))$,
the function $\Phi^{\dug}$ is regular at $P$.
\end{proof}

\section{Proofs of local results}
\label{S:local}

Fix a prime $p \ge 5$.
Recall that $R=\hat{\bZ}_p^{\ur}$, $k=R/pR$, $K:=R[1/p]$,  
and $\phi \colon R \to R$ is the Frobenius automorphism.

\subsection{Review of CL and CM points}
\label{S:review of CL and CM}
This section reviews facts we need about CL abelian schemes 
and their relation with CM points; see~\cite{Katzcan,DO,messing}.
Expert readers should skip this discussion.

\begin{definition}
\label{deffCL}
An abelian scheme $E/R$ is $\CL$ (a {\it canonical lift})
if its reduction $\bar{E}:=E \otimes k$ is ordinary 
and there exists an $R$-homomorphism 
$E \to E^{\phi}:=E \otimes_{R,\phi} R$ 
whose reduction mod $p$ is the relative Frobenius $k$-homomorphism
$\bar{E} \to \bar{E}^{\Fr}:=\bar{E} \otimes_{k,\Fr}k$.
\end{definition}

\begin{theorem}
\label{T:revv1}
The following are equivalent for an elliptic curve $E$ over $R$:
\begin{enumerate}
\item $E$ is $\CL$.
\item $E$ has ordinary reduction and  Serre-Tate parameter $q(E)=1$
  (with respect to some, and hence any,  
   basis of the physical Tate module).
\item There exists a morphism of {\em $\bZ$-schemes} $E \to E$ whose
 reduction mod $p$ is the absolute Frobenius $\bF_p$-morphism
 $\bar{E} \to \bar{E}$. 
 (In~\cite{book} this situation was referred to by saying
  that $E$ {\it has a lift of Frobenius}.)
\end{enumerate}
\end{theorem}

\begin{proof}
The equivalence between 2 and 1 is essentially the definition of
the CL property in~\cite{Katzcan}.
The implication $1 \implies 3$ is trivial. For $3 \implies 1$, 
note first that $\bar{E}$ must be ordinary: 
this follows, for instance, 
from Proposition~7.15 and Corollaries 8.86 and~8.89
in~\cite{book}.
Finally, the $\bZ$-morphism $E \to E$ induces 
an $R$-morphism $E \to E^{\phi}$; 
the N\'{e}ron model property 
shows that the latter is a composition of a
homomorphism $u$ with translation by an $R$-point reducing to 
the identity mod $p$.
But then $u$ mod $p$ is the relative Frobenius.
\end{proof}

\begin{theorem}[Existence and uniqueness of CL abelian schemes]
\label{T:revv2}
\hfill
\begin{enumerate}
\item
Fix a prime $p$ and an ordinary abelian variety $\bar{E}$ over $k$. 
Then there exists a unique $\CL$ abelian scheme $E$ over $R$ 
with $E \otimes k \simeq \bar{E}$ (unique up to isomorphism).
\item
If $E$ and $E'$ are $\CL$ abelian schemes over $R$, 
then the natural map 
$\Hom_{R}(E,E')\to \Hom_{k}(\bar{E},\bar{E}')$ is an isomorphism. 
\item
If two elliptic curves over $R$ are related
by an isogeny of degree prime to $p$ and one of them is $\CL$, 
then so is the other.
\end{enumerate}
\end{theorem}

\begin{proof}
This is due to Serre and Tate: see~\cite{Katzcan, DO}.
\end{proof}

The {\em conductor} of an order in a quadratic number field
is the index of the order in the maximal order.

\begin{theorem}[Relation between CL and CM]
\label{T:revv3}
\hfill
\begin{enumerate}
\item
\begin{enumerate}
\item If $E$ is a $\CL$ elliptic curve over $R$, then $E$ has $\CM$
(part of this claim is that
$E$ is definable over $M = K \intersect \Qbar$).
Thus we have the relation $\CL \subseteq \CM$ between subsets
of $Y_1(N)(\Qbar)$.
\item Conversely, if $Q = (E,\alpha) \in Y_1(N)(\Qbar)$ is in $\CM$,
and $p$ is split in $\End E \tensor \Q$
and does not divide the conductor of $\End E$,
then $Q \in \CL$.
\end{enumerate}
\item
\begin{enumerate}
\item If $(E,i)$ is a $\CL$ false elliptic curve over $R$, 
then $(E,i)$ is CM.
Thus we have the relation $\CL \subseteq \CM$ between subsets
of $X^D(\calU)(M)$.
\item Conversely, for any $\CM$-point $Q \in X^D(\calU)(\Qbar)$,
we know that the associated abelian surface $E$ 
is the square of an elliptic curve with $\CM$ by an order in some $\calK$;
if $p$ splits in $\calK$ and $p$ does not divide the conductor of
the order, then $Q \in \CL$.
\end{enumerate}
\end{enumerate}
\end{theorem}

\begin{proof}\hfill
\begin{enumerate}
\item
\begin{enumerate}
\item 
If $E/R$ is a CL elliptic curve, then 
$\End_{R}(E) \isom \End_{k}(\bar{E}) \neq \bZ$.
\item
This follows from the theorem in the middle of p.~293 
in~\cite{SerreComplexMultiplication}.
\end{enumerate}
\item
\begin{enumerate}
\item
Let $\calE:= \End_{R}(E) \tensor \Q \isom \End_{k}(\bar{E}) \tensor \Q$.
Since $\bar{E}$ is ordinary, the center of $\calE$
contains an imaginary quadratic field ${\mathcal K}$:
see \cite[p.~247]{book}, say. 
In particular, $\calE \not\isom D$, so $(E,i)$ is CM.
\item 
Apply Theorem~\ref{T:revv3}(1)(b) to the elliptic curve.
\end{enumerate}
\end{enumerate}
\end{proof}

\subsection{$\d$-functions} See~\cite{char,book}.
Let $\d\colon R \to R$ be the {\it Fermat quotient map} 
$\d x:=(\phi(x)-x^p)/p$. 
Then
\begin{equation}
\label{axioms}
\begin{array}{rcl} \d(x+y) & = &  \d x + \d y
+C_p(x,y)\\
\d(xy) & = & x^p \cdot \d y +y^p \cdot \d x +p \cdot \d x \cdot \d
y,
\end{array}\end{equation}
where $C_p(X,Y):=\frac{X^p+Y^p-(X+Y)^p}{p} \in \bZ[X,Y]$.
 Following~\cite{char} we think of $\d$ 
as a ``derivation with respect to $p$''. 
If $P \in \bA^N(R) = R^N$, then $\delta P$ is defined by applying $\delta$
to each coordinate.

Let $X$ be  a smooth $R$-scheme and let $f\colon X(R) \to R$ be a map of sets. 
Following~\cite[p.~41]{book}, we say that $f$ is a 
{\em $\d$-function of order $r$}
if for any point in $X(R)$ there is
a Zariski open neighborhood $U \subset X$, 
a closed immersion $u\colon U \injects \bA^N_R$,
and a restricted power series $F$ with $R$-coefficients in $(r+1)N$ variables 
such that
\[
	f(P)= F(u(P),\d(u(P)),\ldots ,\d^r(u(P))) 
	\quad \text{for all $P \in U(R)$.}
\]
({\em Restricted} means that the coefficients converge $p$-adically to $0$.)
Let $\cO^r(X)$ be the ring of $\d$-functions of order $r$ on $X$. 

We have natural maps $\d\colon \cO^r(X) \to \cO^{r+1}(X)$, $f
\mapsto \d f:=\d \circ f$, and natural ring homomorphisms
$\phi\colon \cO^r(X) \to \cO^{r+1}(X)$, $f \mapsto
\phi(f)=f^{\phi}:=\phi \circ f$. The maps $\d$ above still satisfy
the identities in \eqref{axioms}. Let $X$ be affine, and
let $x$ be a system of \'{e}tale coordinates on $X$, that is to
say there exists an \'{e}tale map $X \to \bA^d$ such that $x$ is
the $d$-tuple of elements in $\cO(X)$ obtained by pulling back the
coordinates on $\bA^d$. 
Let $x',x'',\ldots ,x^{(r)}$ be $d$-tuples of variables 
and let $\h$ denotes $p$-adic completion, as usual.
Then the natural map
\begin{equation} \label{zim} \cO(X)\h[x',x'',\ldots ,x^{(r)}]\h \ra
\cO^r(X)
\end{equation}
 sending $x' \mapsto \d x$, $x'' \mapsto \d^2
x$,\ldots ,$x^{(r)} \mapsto \d^r x$ is an isomorphism: see
Propositions 3.13 and~3.19 in~\cite{book}. 

\subsection{$\d$-characters} 
\label{delta-characters}
We recall facts from~\cite{char,book}.
If $G$ is a smooth group scheme over $R$, 
then by a {\em $\d$-character of order $r$}
we understand a $\d$-function  $\psi\colon G(R) \to R$ of order $r$
which is also a group homomorphism into the additive group of $R$.
Following~\cite{char}, we view $\d$-characters of abelian schemes
as arithmetic analogues of the Manin maps~\cite{man,man1}. 
Let $\bX^r(G)$ be the $R$-module of $\d$-characters of order $r$ on
$G$.  
By~\cite[pp.~325-326]{char}, the following hold for an elliptic curve $E/R$:
\begin{enumerate}
\item If $E$ is CL, then $\bX^1(E)$ is free of rank $1$.
\item If $E$ is not CL, then $\bX^2(E)$ is free of rank $1$.
\end{enumerate}
We will need to review (and  complement)  some results in
\cite{char,frob} that can be viewed as an arithmetic analogue of
Manin's Theorem of the Kernel~\cite{man1,chai}. 
For any abelian group $G$,
we set $p^{\infty}G:=\cap_{n =1}^{\infty} p^nG$ and 
we let $p^{\infty}G:p^{\infty}$ be the group of all $x \in G$ 
for which there exists an integer $n\geq 1$ with $p^nx \in p^{\infty}G$. 

\begin{lemma}
\label{thofker} Let $E$ be an elliptic curve over $\Z_p$.
Let $r$ be $1$ or $2$ according as $E$ is $\CL$ or not.
Let $\psi\colon E(R) \to R$ be a generator of $\bX^r(G)$.
Then
\begin{enumerate}
\item $\psi$ is surjective and defined over $\bZ_p$.
\item $\ker \psi=p^{\infty}E(R):p^{\infty}$.
\item $\ker \psi+pE(R)=E(R)_{\tors}+pE(R)=:E(R)\pdiv$.
\item $\psi^{-1}(pR)=E(R)_{\tors}+pE(R)=:E(R)\pdiv$.
\item $(\ker \psi) \cap E(\bZ_p^{\ur})=E(\bZ_p^{\ur})_{\tors}$.
\end{enumerate}
\end{lemma}

\begin{proof}
\hfill
\begin{enumerate}
\item
Surjectivity follows from~\cite[Theorem~1.10]{frob}. 

That $\psi$ is defined over $\Z_p$ follows its construction in~\cite{char}.
\item
If $E$ has ordinary reduction, then \cite[Theorem~B', p.~312]{char}
shows that $(\ker \psi)/p^{\infty}E(R)$ is a
finite cyclic $p$-group; this implies the non-trivial inclusion ``$\subset$''.
If $E$ has supersingular reduction, 
we are done by~\cite[Corollary~1.12]{frob}.
\item
The non-trivial inclusion is ``$\subset$''. 
If $P \in \ker \psi$, by (2) there
exists $n$ such that $p^nP=p^{n+1}Q$ for some $Q \in E(R)$. So
$P-pQ \in E(R)_{\tors}$ and we are done.
\item 
This follows from (3) and (1).
\item
If $E$ has ordinary reduction,
then by Theorem~1.2 and Remark~1.3 on p.~209 of~\cite{frob},
we have 
$p^{\infty}E(R) \cap E(\bZ_p^{\ur}) \subset E(\bZ_p^{\ur})_{\tors}$;
combining this with (2) yields the nontrivial inclusion 
$(\ker \psi) \cap E(\bZ_p^{\ur}) \subset E(\bZ_p^{\ur})_{\tors}$ of (5).
Now assume that $E$ has supersingular reduction. 
If $a_p$ is the trace of Frobenius on $E_{\bF_p}$ 
then the map $\phi^2-a_p \phi+p\colon R \to R$ is injective.
By~\cite[Theorem~1.10, p.~212]{frob}, 
the restriction of $\psi$ to the kernel of the reduction map $E(R) \to E(k)$ 
is injective. Hence we have an injection $\ker \psi \to E(k)$. 
Since $E(k)$ is torsion, so is $\ker \psi$.
\end{enumerate}
\end{proof}

We now describe an explicit generator $\psi$ of $\bX^r(A_R)$,
where $A$ is an elliptic curve over $\Z_p$,
and $r$ is $1$ or $2$ according as $A_R$ is CL or not.
Fix a $1$-form $\omega$ generating the $\Z_p$-module $H^0(A,\Omega^1)$.
This uniquely specifies a Weierstrass model $y^2=x^3+ax+b$ for $A$ over $\Z_p$
such that $\omega=dx/y$.
Let $T:=-x/y$. 
So $T$ is an \'etale coordinate at the origin $0$ of $A$, vanishing at $0$. 
Let $L(T) \in \Q_p[[T]]$
be the logarithm of the formal group of $A$ associated to $T$,
so $dL(T) =\omega \in \Z_p[[T]]\,dT$ and $L(0)=0$.
If $A$ is $\CL$, let $up$ be the unique root in $p\Z_p$
of the polynomial $x^2-a_p x+p$. 
By \cite[Theorem~7.22]{book} and \cite[Theorem~1.10]{frob}, 
we may take 
\begin{equation}
\label{E:psi}
\psi := 
\begin{cases}
	\frac{1}{p} (\phi^2-a_p\phi+p) L(T) \in R[[T]][T',T'']\hat{\ }, 
		& \text{ if $A$ is not $\CL$;} \\
	\frac{1}{p} (\phi-up) L(T) \in R[[T]][T']\hat{\ },
		& \text{ if $A$ is $\CL$.}
\end{cases}
\end{equation}

\subsection{$\d$-Fourier expansions} 
\label{S:delta-Fourier}
See~\cite{difmod}.
We start by reviewing background on 
{\em classical} Fourier expansions as in~\cite[p.~112]{DI}. 
(The discussion there
involves the modular curve parameterizing elliptic curves with an
embedding of $\mu_N$ rather than $\bZ/N\bZ$ as here. 
But, the two modular curves are isomorphic over $\ZN$: see~\cite[p.~113]{DI}.) 
The cusp $\infty$ on $S:=X_1(N)$ arises from a $\ZN$-valued point;
so if $p \gg 0$ (specifically, $p \nmid N$),
then it gives rise to an $R$-point, 
which may be viewed as a closed immersion $s_\infty \colon \Spec R \to S_R$.
Let $[\infty] = s_\infty(\Spec R)$.
Let $\tilde{S}_{R}$ be the completion of $S_{R}$ along $[\infty]$. 
The Tate generalized elliptic curve $\operatorname{Tate}(q)/R[[q]]$ 
equipped with 
the standard immersion $\alpha_{can}$ of $\mu_{N,R} \simeq (\bZ/N\bZ)_{R}$ 
is a point in $S(R[[q]])$ that reduces mod $q$ to $s_\infty$.
For $p \gg 0$ there is an induced isomorphism 
$\Spf R[[q]] \simeq \tilde{S}_{R}$. 
Therefore, for any open subset $U \subset S_{R}$ containing $[\infty]$ 
we have an induced {\em Fourier $q$-expansion} homomorphism
\[ 
	\cO(U \setminus [\infty]) \to R((q)):=R[[q]][1/q].
\]

More generally, suppose that we are given a modular-elliptic
correspondence $\thecorr$.
Let $M$ be the ramification index of $\Pi$ at $x_{\infty}$. 
As before, we assume $p\gg 0$. 
Then
we have $\Spf R[[\qq]] \simeq \tilde{X}_{R}$,
where $\qq:=q^{1/M}$ and $\tilde{X}_{R}$ is the completion of
$X_{R}$ along the closure $[x_{\infty}]$ of $x_{\infty}$.
Moreover, for any open set $U \subset X_{R}$ containing
$[x_{\infty}]$ we have a {\em Fourier $q$-expansion} homomorphism
\[
	\cO(U \setminus [x_{\infty}]) \to R((\qq)).
\]

Next we move to the ``$\d$-theory''. 
Let $q',q'',\ldots ,q^{(r)},\ldots$ be new indeterminates.
Define
\[
	S_{\infty}^r:=R((q))\hat{\ }[q',q'',\ldots ,q^{(r)}]\hat{\ }.
\]
For each $r$, extend $\phi\colon R \to R$ to a ring homomorphism
$\phi\colon S_{\infty}^r \to S_{\infty}^{r+1}$
denoted $F \mapsto F^{\phi}$ 
by requiring 
\[
q^{\phi}:=q^p+pq', \quad (q')^{\phi}:=(q')^p+pq'', \quad \ldots,
\]
and define $\d\colon S_{\infty}^r \to S_{\infty}^{r+1}$ by
\begin{equation}
\label{E:d formula}
	\d F:= \frac{F^{\phi}-F^p}{p}. 
\end{equation}
By the universality property of
the sequence $\{\cO^r(U\setminus [\infty])\}_{r \geq 0}$ 
(see~\cite[Proposition~3.3]{book}), 
there exists a unique sequence of ring homomorphisms
\begin{equation}
\label{vaiy} \cO^r(U \setminus [\infty]) \to S_{\infty}^r,
\end{equation}
called {\em $\d$-Fourier expansion} maps 
and denoted $g \mapsto g_{\infty}$,
such that $(\d g)_{\infty}=\d(g_{\infty})$ for all $g$. 

More generally, given a modular-elliptic correspondence
$\thecorr$, define rings 
\[
	S^r_{x_{\infty}}:=R((\qq))\h[\qq',\ldots ,\qq^{(r)}]\h
\]
where $\qq',\ldots ,\qq^{(r)}$ are new variables. 
Again there are natural maps 
$\phi,\d\colon S^r_{x_{\infty}} \to S^{r+1}_{x_{\infty}}$ defined
exactly as above and there are {\em $\d$-Fourier expansion} maps
\[
	\cO^r(U \setminus [x_{\infty}]) \to S_{x_{\infty}}^r
\]
commuting with $\d$, and denoted $g \mapsto g_{x_\infty}$. 
There are natural maps $S^r_{\infty} \to S^r_{x_{\infty}}$. 
Since $\Spec R[\qq,\qq^{-1}] \to \Spec R[q,q^{-1}]$ 
is \'etale, \ref{zim} implies 
\[
	S^r_{x_{\infty}} \simeq R((\qq))\h[q',\ldots ,q^{(r)}]\h.
\]

\subsection{$\d$-Serre-Tate expansions} 
\label{S:delta Serre-Tate}
See~\cite{shimura,book}.
Assume that we are given a Shimura-elliptic correspondence $\thecorr$,
and that $p\gg 0$. 
By the proof of Lemma~2.6 in~\cite{shimura}, 
there exist infinitely many $k$-points $\bar{y}_0 \in S(k)$ 
whose associated triple
$(\bar{Y},\bar{i},\bar{\alpha})$ is such that 
\begin{enumerate}
\item $\bar{Y}$ is ordinary, and
\item if $\bar{\theta}$ is the unique principal polarization 
compatible with $\bar{i}$, then $(\bar{Y},\bar{\theta})$ 
is isomorphic to the polarized Jacobian of a genus-$2$ curve. 
\end{enumerate}
So we may choose a point
$\bar{y}_0 \in S(k)$ as above such that moreover,
there exists $\bar{x}_0 \in \bar{X}(k)$ with
$\Pi(\bar{x}_0)=\bar{y}_0$ such that both $\Pi$ and $\Phi$ are
\'etale at $\bar{x}_0$: here we use $p \gg 0$
to know that $\Pi \otimes k$ and $\Phi \otimes k$ are separable.

Let $Y$ be the canonical lift of $\bar{Y}$. 
Since $\End(Y) \isom \End(\bar{Y})$,
the embedding $\bar{i}\colon \cO_D \to \End(\bar{Y})$ 
induces an embedding $i\colon \cO_D \to \End(Y)$. 
Also the level $\cU$ structure
$\bar{\alpha}$ lifts to a level $\cU$ structure on $(Y,i)$. 
Let $y_0:=(Y,i,\alpha) \in S(R)$.
Since $\Pi$ is \'{e}tale at $\bar{x}_0$, there exists 
$x_0 \in X(R)$ such that $x_0 \bmod p = \bar{x}_0$ and $\Pi(x_0)=y_0$. 

Let $\bar{Y}^\vee$ be the dual of $\bar{Y}$.
By~\cite[Lemma~2.5]{shimura}, there exist
$\bZ_p$-bases of the Tate modules $T_p(\bar{Y})$ and $T_p(\bar{Y}^\vee)$,
corresponding to each other under $\bar{\theta}$,
such that any false elliptic curve over $R$ lifting $(\bar{Y},\bar{i})$ has a
diagonal Serre-Tate matrix $\operatorname{diag}(q,q^{disc(D)})$ 
with respect to these bases. 
Fix such bases. 
They define an isomorphism between
the completion of $S_{R}$ along the section $y_0$ and $\Spf R[[t]]$. 
The Serre-Tate parameter $q$ corresponds to the value
of $1+t$. Since $\Pi$ is \'{e}tale at $\bar{x}_0$ we have an
induced isomorphism between the completion of $X$ along the
section $x_0$ and $\Spf R[[t]]$. 
As in Section~\ref{S:delta-Fourier} define rings
\[S^r_{x_0} \simeq R[[t]][t',\ldots ,t^{(r)}]\h\]
and maps $\phi,\d\colon S^r_{x_0} \to S^{r+1}_{x_0}$; then for any
affine open set $U \subset X$  containing the image of the section
$x_0$ we have natural {\em $\d$-Serre-Tate expansion maps}
\begin{equation}
\label{vaiy2} \cO^r(U) \to S_{x_0}^r,
\end{equation}
denoted $g \mapsto g_{x_0}$, that commute with $\phi$ and $\d$.

\subsection{Pullbacks by $\Phi$ of $\d$-characters}
\label{S:pullbacks}
Assume that we are given 
a modular-elliptic or a Shimura-elliptic correspondence $\thecorr$.
Recall that $A$ is defined over a number field $F_0$.
We suppose that $p \gg 0$ and $p$ splits completely in $F_0$.
Then $A_R$ comes from an elliptic curve over $\bZ_p$.
Define $a_p$ and (if $A_R$ is CL) $u$ as in Section~\ref{delta-characters}.
Let $\psi$ be as in~\eqref{E:psi}.
The composition
\begin{equation}
\label{fsharp} f^{\sharp}\colon X(R) \stackrel{\Phi}{\ra} A(R)
\stackrel{\psi}{\ra} R.
\end{equation}
is in $\cO^r(X_{R})$. 
In what follows we compute the $\d$-Fourier expansion
$f^{\sharp}_{x_{\infty}} \in S^r_{x_{\infty}}$ (in the modular-elliptic case) 
or the $\d$-Serre-Tate expansion
$f^{\sharp}_{x_0} \in S^r_{x_0}$ (in the Shimura-elliptic case).

\subsubsection{Modular-elliptic case}
\label{S:modular-elliptic pullback}
Suppose that $S=X_1(N)$.
We have
$\Phi^*\colon R[[T]] \to R[[\qq]]$.
Define $b_n \in F_0 \cap R$ by
\[
	\left(\sum_{n \geq 1} b_n \qq^{n-1} \right)\, d\qq := d(\Phi^*(L(T))) = \Phi^*(dL(T))= \Phi^* \omega.
\]
so 
\begin{equation}
\label{E:b_n series}
\sum_{n \geq 1} \frac{b_n}{n} \qq^n = \Phi^*(L(T)).
\end{equation}
Applying $\Phi^*$ to \eqref{E:psi} and substituting \eqref{E:b_n series}
yields
\begin{equation}
\label{titi}
f^{\sharp}_{x_{\infty}} = \Phi^{*} \psi = 
\begin{cases}
\frac{1}{p} \sum_{n\geq 1} \left(
\frac{b_n^{\phi^2}}{n} \qq^{n\phi^2} -a_p \frac{b_n^{\phi}}{n}
\qq^{n\phi} +p \frac{b_n}{n} \qq^n \right),
		& \text{ if $A$ is not $\CL$;} \\
\frac{1}{p} \sum_{n\geq 1} \left(
 \frac{b_n^{\phi}}{n}
\qq^{n\phi} -up \frac{b_n}{n} \qq^n \right),
		& \text{ if $A$ is $\CL$.}
\end{cases}
\end{equation}
In both cases, $f^{\sharp}_{x_{\infty}} \in R[[\qq]][\qq',\qq'']\h$. 
Applying the substitution homomorphism
\begin{align*}
	R[[\qq]][\qq',\qq'']\h &\to R[[\qq]] \\
	G &\mapsto G_{\natural}:=G(\qq,0,0)=G|_{\qq'=\qq''=0},
\end{align*}
we obtain
\begin{equation}
\label{keti}
(f^{\sharp}_{x_{\infty}})_{\natural} =
\begin{cases}
\frac{1}{p} \sum_{n
\geq 1} \left( \frac{b^{\phi^2}_{n/p^2}}{n/p^2} -a_p
\frac{b^{\phi}_{n/p}}{n/p} +p \frac{b_n}{n}
\right) q^n_{M},
		& \text{ if $A$ is not $\CL$;} \\
\frac{1}{p} \sum_{n
\geq 1} \left( \frac{b^{\phi}_{n/p}}{n/p} -up \frac{b_n}{n}
\right) \qq^n,
		& \text{ if $A$ is $\CL$,}
\end{cases}
\end{equation}
where $b_{\gamma}:=0$ if $\gamma \in \bQ \setminus \bZ$. 
(In particular, the right hand side of \eqref{keti}
has coefficients in $R$, which is not a priori obvious.)

Let us consider the special case when
$\thecorr$ arises from a modular parametrization
associated to the newform $f=\sum a_n q^n$,
so $S=X=X_1(N)$, $\Pi=\Id$,
 $x_{\infty}=\infty$, $M=1$, and $\qq=q$. 
We may take $\omega$ so that $\Phi^* \omega=\sum a_n q^{n-1}dq$;
then $b_n=a_n$ for all $n$. 
Since $f$ is a newform,
the $a_n$ satisfy the usual relations~\cite[Theorem~3.43]{Shimura1971}
(we use $p \gg 0$ to know that $p \nmid N$):
\begin{align}
\label{tzu1}
a_{p^i m} &= a_{p^i}a_{m} \quad \text{ for $(p,m)=1$}, \\
\label{tzu2}
a_{p^{i-1}}a_p &= a_{p^i}+p a_{p^{i-2}} \quad \text{ for $i \geq 2$}.
\end{align}

\begin{lemma}
\label{compeverything}
Assume that $\thecorr$ arises from a modular parametrization attached to $f$.
\begin{enumerate}
\item 
With notation as in \eqref{fruct1} and \eqref{fruct2}, the
following holds in $\bZ_p[[q]]$:
\begin{equation}
\label{floare}
(f^{\sharp}_{\infty})_{\natural} = 
\begin{cases}
f^{(-1)}(q),
		& \text{ if $A$ is not $\CL$;} \\
-uf^{(-1)}_{[u]}(q),
		& \text{ if $A$ is $\CL$.}
\end{cases}
\end{equation}
\item
With notation as in \eqref{fructnou1} and \eqref{fructnou2},
the following holds in $k[[q]][q',q'']$:
\begin{equation}
\label{floarenoua}
\overline{f^{\sharp}_{\infty}} =
\begin{cases}
\overline{f^{(-1)}(q)}
 + \left( \frac{q'}{q^p} \right)^p
 \left(\overline{f^{(0)}_{[a_p]}(q)}\right)^{p^2}
 - \bar{a}_p \left( \frac{q'}{q^p} \right)
 \left(\overline{f^{(0)}_{[a_p]}(q)}\right)^{p},
		& \text{ if $A$ is not $\CL$;} \\ \\
-\bar{u} \overline{f^{(-1)}_{[u]}(q)}+
 \left( \frac{q'}{q^p} \right)
 \left(\overline{f^{(0)}_{[a_p]}(q)}\right)^{p},
		& \text{ if $A$ is $\CL$.}
\end{cases}
\end{equation}
\end{enumerate}
\end{lemma}

\begin{proof}
We shall prove~\eqref{floarenoua} in the case where $A_{R}$ is not CL. 
The other three statements are proved similarly
(and are actually easier).

To simplify notation, let $\square$ stand for any element
of $\bZ_p[[q]][q^{-1},q',q'']\h$. 
For any $\gamma,\beta \in \bZ_p[[q]][q^{-1},q',q'']\h$, 
any $\ell \in \Z_{\ge 2}$, and any $m \in \Z_{\ge 1}$ we have
\begin{equation}
\label{oindd}
(1+p\gamma+p^2 \beta)^{mp^{\ell-2}}=1+mp^{\ell-1}\gamma+p^\ell\square.\end{equation}
(Writing $(1+p\gamma+p^2\beta)^m$ as $1+p\gamma'$ lets us reduce to the
case $\beta=0$ and $m=1$, which is proved by induction on $\ell$.)

By \eqref{titi} we get
\begin{align*}
f^{\sharp}_{\infty} &= \frac{1}{p} \left[ 
		\sum \frac{a_n}{n} \left(q^{p^2}+p(q')^p+p^2 \square \right)^n
		-a_p \sum \frac{a_n}{n} \left(q^p+pq' \right)^n
		+p\sum \frac{a_n}{n}q^n \right]\\
	&= \frac{1}{p} \left[ 
		\sum \frac{a_n}{n} \left(1+p(\frac{q'}{q^p})^p
		+p^2\square \right)^n q^{p^2n}
		-a_p\sum \frac{a_n}{n} \left(1+p\frac{q'}{q^p} \right)^n q^{pn}
		+p\sum \frac{a_n}{n}q^n \right]\\
	&= \sum\left[ \frac{a_{n/p^2}}{n/p}
		\left(1+p\left(\frac{q'}{q^p}\right)^p
			+p^2\square \right)^{n/p^2}
		-a_p \frac{a_{n/p}}{n}\left(1+p\frac{q'}{q^p}\right)^{n/p}
			+\frac{a_n}{n} \right]q^n\\
	&=: \sum \gamma_nq^n,
\end{align*}
where $a_r=0$ for $r \in \bQ \setminus \bZ$.

If $(n,p)=1$, then $\gamma_n=a_n/n$.

If $n=pm$ with $(m,p)=1$, then \eqref{tzu1} and \eqref{oindd} yield
\[
	\gamma_n
	=\frac{a_pa_m}{pm}
	 -a_p\frac{a_m}{pm} \left(1+pm\frac{q'}{q^p}+p^2\square\right) 
	\equiv -a_pa_m\frac{q'}{q^p} \pmod{p}.
\]

If $n=p^\ell m$ with $\ell \geq 2$ and $(m,p)=1$, 
then \eqref{tzu1}, \eqref{tzu2}, and \eqref{oindd} yield
\begin{align*}
\gamma_n &= \frac{a_{p^{\ell-2}}a_m}{p^{\ell-1}m} \left( 1+mp^{\ell-1}
\left( \frac{q'}{q^p} \right)^p+p^\ell \square \right)
-\frac{a_p a_{p^{\ell-1}}a_m}{p^\ell m}\left(1+mp^\ell \frac{q'}{q^p}+p^{\ell+1}\square \right)+
\frac{a_{p^\ell}a_m}{p^\ell m}\\
&\equiv a^{\ell-2}_p a_m \left( \frac{q'}{q^p} \right)^p
  - a_p^\ell a_m \frac{q'}{q^p} \pmod{p}.
\end{align*}
Therefore
\[
f^{\sharp}_{\infty} \equiv \sum_{(m,p)=1} \frac{a_m}{m} q^m
-a_p \frac{q'}{q^p} \sum_{(m,p)=1} a_m q^{mp}
+ \sum_{\ell \geq 2} \sum_{(m,p)=1} a_m\left( a_p^{\ell-2}
\left(\frac{q'}{q^p} \right)^p-a_p^\ell \frac{q'}{q^p} \right) q^{mp^\ell} 
\pmod{p},
\]
and the first case of \eqref{floarenoua} 
follows via a trivial algebraic manipulation.
\end{proof}

\begin{remark}
\label{R:f-sharp mod p is order 1}
The right hand side of~\eqref{floarenoua} belongs to 
the subring $k[[q]][q']$ of $k[[q]][q',q'']$.
In the case where
$\thecorr$ does not necessarily arise from a modular parametrization,
an argument similar to the one in the proof of Lemma~\ref{compeverything}
still yields
\begin{equation}
\label{erasasterb}
	\overline{f^{\sharp}_{x_{\infty}}} \in k[[\qq]][\qq'].
\end{equation}
\end{remark}

\subsubsection{Shimura-elliptic case}
\label{S:Shimura pullback}
Suppose that $S=X^D(\calU)$.
Recall that we fixed $x_0 \in X(R)$ 
and a corresponding $\d$-Serre-Tate expansion map
$\cO^2(X_{R}) \to S^2_{x_0}=R[[t]][t',t'']\h$,
denoted $G \mapsto G_{x_0}$.
Let $z_0=\Phi(x_0) \in A(R)$.
Let $\lambda \colon A_{R}\to A_{R}$ be the translation by $-z_0$. 
Recall the \'etale coordinate $T$ on $A_{R}$ at $0$; 
use $T_{z_0}:=\lambda^* T$ as \'etale coordinate at $z_0$.
Now we have 
$R[[T]] \stackrel{\lambda^*}\to R[[T_{z_0}]] \stackrel{\Phi^*}\to R[[t]]$.
Define $b_n \in F_0 \cap R$ by 
\[
\left(\sum_{n \ge 1} b_n t^{n-1} \right)\, dt := d(\Phi^* \lambda^*(L(T))) = \Phi^* \lambda^* d(L(T)) = \Phi^* \lambda^* \omega,
\]
so
\begin{equation}
\label{E:b_n series 2}
\sum_{n \geq 1} \frac{b_n}{n} t^n = \Phi^* \lambda^*(L(T)).
\end{equation}
Since $\Phi$ is \'etale at $x_0$, we have $b_1 \ne 0$;
scaling $\omega$, we may assume that $b_1=1$.
Since $\psi$ is a group homomorphism, 
we have $\psi-\psi(z_0) = \lambda^* \psi$.
Add the constant $\psi(z_0)$ to both sides, and apply $\Phi^*$
to obtain
\[
	f^\sharp_{x_0} = \Phi^* \psi = \psi(z_0) + \Phi^* \lambda^* \psi.
\]
Evaluate $\Phi^* \lambda^* \psi$ 
by applying $\Phi^* \lambda^*$ to \eqref{E:psi}
and substituting \eqref{E:b_n series 2} into the right hand side:
the final result is
\begin{equation}
\label{titish}
f^{\sharp}_{x_0} =
\begin{cases}
\psi(z_0) + 
\frac{1}{p} \sum_{n\geq 1} \left(
\frac{b_n^{\phi^2}}{n} t^{n\phi^2} -a_p \frac{b_n^{\phi}}{n}
t^{n\phi} +p \frac{b_n}{n} t^n \right),
		& \text{ if $A$ is not $\CL$;} \\
\psi(z_0) + 
\frac{1}{p} \sum_{n\geq 1} \left(
 \frac{b_n^{\phi}}{n}
t^{n\phi} -up \frac{b_n}{n} t^n \right),
		& \text{ if $A$ is $\CL$.}
\end{cases}
\end{equation}
An argument similar to the one in the proof of Lemma~\ref{compeverything}
shows that
\begin{equation}
\label{erasadau}
\overline{f^{\sharp}_{x_0}} \in k[[t]][t'].
\end{equation}

 \subsection{$\d$-modular forms: modular-elliptic
  case} 
\label{S:d-modular-elliptic}
We recall some concepts from~\cite{difmod,book,Barcau}.
The ring of {\em $\d$-modular functions}~\cite{difmod}
is 
\[
	M^r:=R[a_4^{(\leq r)},a_6^{(\leq r)}, \Delta^{-1}]\h,
\]
where $a_4^{(\leq r)}$ is a tuple of variables
$(a_4,a'_4,a''_4,\ldots ,a_4^{(r)})$ and $a_6^{(\leq r)}$ is similar,
and $\Delta:=-2^6a_4^3-2^43^3a_6^2$. 
If $g \in M^0 \setminus pM^0$, define
\[
	M^r_{\{g\}} := M^r[g^{-1}]\h
	= R[a_4^{(\leq r)},a_6^{(\leq r)}, \Delta^{-1},g^{-1}]\h.
\] 
An element of $M^r$ or $M^r_{\{g\}}$ is {\em defined over $\bZ_p$}
if it belongs to the analogously defined ring with $\Z_p$ in place of $R$.
Define $\d\colon M^r \to M^{r+1}$
and $\d\colon M^r_{\{g\}} \to M^{r+1}_{\{g\}}$
as $\d \colon S_\infty^r \to S_\infty^{r+1}$ was defined 
in Section~\ref{S:delta-Fourier}.
Let $j:-2^{12}3^3 a_4^3/\Delta$, let $i:=2^63^3-j$, and let $t:=a_6/a_4$.
(This $t$ is unrelated to the $t$ used in $\d$-Serre-Tate expansions.)  
By~\cite[Proposition~3.10]{difmod}, we have
\[
	M^r_{\{a_4a_6\}} 
	= R[j^{(\leq n)},j^{-1}, i^{-1}, t^{(\leq r)}, t^{-1}]\h.
\] 

If $w=\sum n_i \phi^i \in \Z[\phi]$, define $\deg w = \sum n_i$.
If moreover $\lambda \in R$,
define $\lambda^w:=\prod (\lambda^{\phi^i})^{n_i}$.
For $w \in \Z[\phi]$,
say that $f$ in $M^r$ or $M^r_{\{g\}}$ is of {\em weight} $w$
if
\begin{equation}
\label{nu shtiu} f(\lambda^4a_4,\lambda^6a_6,\d(\lambda^4a_4),\d(
\lambda^6a_6),\ldots )=\lambda^w f(a_4,a_6,a'_4,a'_6,\ldots ),
\end{equation}
for all $\lambda \in R$.
Let $M^r(w)$ be the set of $f \in M^r$ of weight $w$,
and define $M^r_{\{g\}}(w)$ similarly.
In~\cite{difmod}, elements of $M^r_{\{g\}}(w)$ were called
{\em $\d$-modular forms of weight $w$} (holomorphic outside $g=0$).

If $f \in M^r_{\{g\}}(w)$ and $E$ is an elliptic curve given by
$y^2=x^3+Ax+B$ with $A,B \in R$ and $g(A,B) \in R^{\times}$, 
then define $f(A,B) \in R$ by making the substitutions
$a_4 \mapsto A$, $a_6 \mapsto B$, $a'_4 \mapsto \d A$, $a_6' \mapsto \d B$, 
$a''_4 \mapsto \d^2 A$, and so on.
Recall from~\cite{difmod} that $f$ is called {\em isogeny covariant} if for
any isogeny $u$ of degree prime to $p$ from an elliptic
curve $y^2=x^3+A_1 x+B_1$ with $g(A,B) \in R^{\times}$ to an
elliptic curve $y^2=x^3+A_2 x+B_2$ with $g(A_2,B_2) \in R^{\times}$  
that pulls back $dx/y$ to $dx/y$ we have
\[
	f(A_1,B_1) = \deg(u)^{-\deg(w)/2} f(A_2,B_2).
\]

By~\cite[Corollary~3.11]{difmod},
$M^r_{\{a_4a_6\}}(0)=R[j^{(\leq r)}, j^{-1}, i^{-1}]\h$. 
More generally, if $m \in 2\Z$ and $g \in M^0(m)$, 
define $\tilde{g}:=gt^{-m/2}$; then
\begin{equation}
\label{urssu} 
	M^r_{\{a_4a_6g\}}(0)
	= R[j^{(\leq r)}, j^{-1}, i^{-1},\tilde{g}^{-1}]\h.
\end{equation} 
Also define the open subscheme 
$Y(1)^{\dug}:=\Spec R[j,j^{-1},i^{-1}, \tilde{g}]$
of the modular curve $Y(1)_R := \Spec R[j]$.
If we define
\[
	b:=a_6^2/a_4^3=-2^23^{-3}+2^8j^{-1}.
\]
then $R[j,j^{-1},i^{-1}]=R[b,b^{-1},(4+27b)^{-1}]$, 
so $b$ is an \'etale coordinate on $Y(1)^{\dug}$,
and $Y_1(N)_{R} \to Y(1)_{R}$ is \'etale over $Y(1)^{\dug}$. 
Suppose that in addition we are
given a modular-elliptic correspondence $\thecorr$.
Then we may (and will) choose $g$ so that the composition
$v\colon X_{R} \stackrel{\Pi}\to X_1(N)_{R} \to X(1)_{R}$
is \'etale above $Y(1)^{\dug}$.
Set 
\begin{equation}
\label{cocer} 
	X^{\dug}:=v^{-1}(Y(1)^{\dug}).
\end{equation}
The pull-back of $b$ to $X^{\dug}$,
which we will still call $b$, is an \'{e}tale coordinate on $X^{\dug}$. 
By~\eqref{zim}, we have natural isomorphisms
\begin{equation}
\label{mazel} \cO(X^{\dug})\h[b',\ldots ,b^{(r)}]\h \simeq
\cO^r(X^{\dug}),
\end{equation}
where $b',\ldots ,b^{(r)}$ are new indeterminates. 
We view \eqref{mazel} as an identification.
Similarly, since $j$ is an \'{e}tale coordinate on $Y(1)$, 
\eqref{zim} and~\eqref{urssu} yield
 \begin{equation}
 \label{foxx}
M^r_{\{a_4a_6g\}}(0) \simeq \cO^r(Y(1)^{\dug}) \subset
\cO^r(X^{\dug}).
 \end{equation}
Since $X^{\dug}$ is standard in the sense of Definition~\ref{D:standard},
we have the $\d$-Fourier expansion map
\begin{equation}
\label{catt} \cO^r(X^{\dug}) \to S^r_{x_{\infty}}.
\end{equation}
Composing \eqref{foxx} and~\eqref{catt} yields $\d$-Fourier expansion maps
\begin{equation}
\label{liion} M^r_{\{a_4a_6g\}}(0) \to S^r_{x_{\infty}}.
\end{equation}

Let $E_4(q)$ and $E_6(q)$ be the normalized Eisenstein series of
weights $4$ and $6$: ``normalized'' means with constant
coefficient equal to $1$.
We have natural ring homomorphisms, also
referred to as {\em $\d$-Fourier expansion maps}~\cite{difmod},
\begin{align}
\label{sqrl} 
	M^r &\to S^r_{\infty} \\
\notag
	g &\mapsto g_{\infty}=g(q,q',\ldots ,q^{(r)}),
\end{align}
characterized by the properties that they 
send $a_4$ and $a_6$ to
$-2^{-4}3^{-1}E_4(q)$ and $2^{-5}3^{-3}E_6(q)$, respectively, 
and commute with $\d$.
There exists a unique $E_{p-1} \in M^0(p-1)$ 
such that $E_{p-1}(q)$ is the normalized Eisenstein series of weight $p-1$.

By (4.1) and~(7.26) in~\cite{difmod}, 
there exists a unique $f^1 \in M^1(-1-\phi)$, defined over $\bZ_p$, 
such that
\begin{equation}
\label{maus} 
	f^1(q,q') 
	=\frac{1}{p} \log \frac{q^{\phi}}{q^p}
	:= \sum_{n \geq 1} 
		(-1)^{n-1}n^{-1} p^{n-1} \left( \frac{q'}{q^p} \right)^n
	\in R((q))\h[q']\h.
\end{equation}
As explained in~\cite[pp.~126--129]{difmod}, 
$f^1$ is isogeny covariant and may
be interpreted as a (characteristic zero) 
{\em arithmetic Kodaira-Spencer class}. 
If $E$ is an elliptic curve given by $y^2=x^3+Ax+B$ with $A,B \in R$ then, 
by~\cite[Proposition~7.15]{book},
\begin{equation}
\label{gugu}
	f^1(A,B)=0 \iff \text{$E$ is $\CL$.}
\end{equation}
Define
\begin{equation}
\label{fffut} 
	t^{\frac{\phi+1}{2}}
	:= t^{\frac{p+1}{2}} \left( \frac{t^{\phi}}{t^p} \right)^{1/2}
	= t^{\frac{p+1}{2}} \left( 1+p \frac{\d t}{t^p} \right)^{1/2} 
	= t^{\frac{p+1}{2}} \sum_{j \geq 0} 
			\binom{1/2}{j} p^j \left( \frac{\d t}{t^p} \right)^j;
\end{equation}
this function is an element of $M^1_{\{a_4a_6\}}(1+\phi)$.
Next define
\begin{equation}
\label{varfi} 
	\varphic :=f^1 \cdot t^{\frac{\phi+1}{2}} 
		\in M^1_{\{a_4a_6\}}(0) \subset M^1_{\{a_4a_6 g\}}(0) \subset \OO^1(X^{\dug}).
\end{equation}
The maps in \eqref{liion} and~\eqref{sqrl} are compatible, so
\begin{equation}
\label{elephant} 
	\varphic_{\infty} \in q'R((q))\h[q']\h \subset \qq' R((\qq))\h[\qq']\h.
\end{equation}
Finally, by the main theorem of~\cite{H},
\begin{equation}
\label{hurlburt} 
	f^1=c E_{p-1}\Delta^{-p}(2a_4^pa_6'-3a_6^pa'_4)+f_0+pf_1,
\end{equation}
for some $c \in R^{\times}$, $f_0 \in M^0(-1-p)$, and $f_1 \in M^1$.
A calculation using the analogue of~\eqref{E:d formula} yields
\begin{equation}
\label{bambilici} 
	\d b=a_4^{-4p}a_6^p(2a_4^pa_6'-3a_6^pa'_4)+ph
\end{equation}
for some $h \in M^1_{\{a_4a_6\}}$.
Set $a_0:=cE_{p-1}\Delta^{-p} a_4^{4p}a_6^{-p}$. 
Then combining \eqref{hurlburt} and~\eqref{bambilici} yields
\begin{equation}
\label{hurlprim} 
	\varphic = f^1 \cdot t^{\frac{\phi+1}{2}}
	= a_0 t^{\frac{p+1}{2}}\d b+f_0 t^{\frac{p+1}{2}}+ph_1,
\end{equation}
for some $h_1 \in M^1_{\{a_4a_6\}}$.
Let $\alpha=a_0 t^{\frac{p+1}{2}} \in M^0_{\{a_4a_6\}}(0)$.
Then by \eqref{hurlprim} and~\eqref{axioms}, respectively, we obtain, 
for $n=0$ and $n=1$,
\begin{equation}
\label{lolla} \d^n \varphic=\alpha^{p^n} \d^{n+1} b +\beta_n+p \gamma_n,
\end{equation} 
for some $\beta_n \in M^n_{\{a_4a_6\}}(0)$
and $\gamma_n \in M^{n+1}_{\{a_4a_6\}}(0)$.

\begin{lemma}
\label{lem1} 
Assume that the element $g \in M^0(m)$ is in $E_{p-1} M^0$.
Then $\overline{\varphic}$ and $\overline{\d \varphic}$ 
are algebraically independent over $\OO(\bar{X}^{\dug})$,
and the natural maps
\begin{align}
\label{E:O1}
	\OO(\bar{X}^{\dug})[\overline{\varphic}] 
		&\to \OO^1(X^{\dug}) \tensor_R k \\
\label{E:O2}
	\OO(\bar{X}^{\dug})[\overline{\varphic}, \overline{\d \varphic}]
		&\to \OO^2(X^{\dug}) \tensor_R k \\
\label{E:O3}
	\cO(X^{\dug})\h
		&\to \cO^2(X^{\dug})/(\varphic,\d \varphic)
\end{align}
are isomorphisms.
\end{lemma}

\begin{proof}
By \eqref{maus}, \eqref{fffut}, and~\eqref{varfi}, we have
\begin{equation}
\label{E:varphic mod p}
	\overline{\varphic_\infty} = t_\infty^{\frac{p+1}{2}} q' / q^p,
\end{equation}
which involves $q'$, so the algebraic independence follows.
Reducing \eqref{mazel} mod $p$ gives isomorphisms like \eqref{E:O1}
and~\eqref{E:O2}
but with $\overline{b'}$ and $\overline{b''}$ on the left
in place of $\overline{\varphic}$ and $\overline{\d \varphic}$.
To change variables, observe that since $g \in E_{p-1} M^0$,
the element $\alpha$ is invertible in $\cO(X^{\dug})$;
thus \eqref{lolla} implies
$\OO(\bar{X}^{\dug})[\overline{\varphic}]
\isom \OO(\bar{X}^{\dug})[\overline{b'}]$
and
$\OO(\bar{X}^{\dug})[\overline{\varphic}, \overline{\d \varphic}] 
\isom \OO(\bar{X}^{\dug})[\overline{b'}, \overline{b''}]$.
This proves \eqref{E:O1} and~\eqref{E:O2}.

Now \eqref{E:O2} implies that \eqref{E:O3} induces an isomorphism mod $p$,
Since both sides of $\eqref{E:O3}$ are $p$-adically complete and
separated rings, \eqref{E:O3} is surjective.
The $\d$-Fourier expansion map $\OO^2(X^{\dug}) \to R((\qq))\h[\qq',\qq'']\h$
followed by the evaluation map mapping $\qq'$ and $\qq''$ to $0$
induces a map
$\cO^2(X^{\dug})/(\varphic,\d \varphic) \to R((\qq))$,
by~\eqref{elephant}.
The composition of~\eqref{E:O3} with this is simply the
Fourier expansion map, since elements of $\OO(X^{\dug})\h$
have Fourier expansions in $R((\qq))$.
So the Fourier expansion principle implies that~\eqref{E:O3} 
is injective.
\end{proof}

\begin{corollary} 
\label{C:series are modular}
The series
$\overline{f^{(0)}(q)}$ 
and 
$\overline{f^{(0)}_{[a_p]}(q)}$ 
are Fourier expansions 
of weight-$2$ quotients of modular forms.
\end{corollary}

\begin{proof}
We have
$\overline{f^{(0)}(q)} 
= \left( \theta^{p-1}\bar{f} \right) / \bar{E}_{p-1}^{p+1}$,
which is the Fourier expansion of a weight-$2$ quotient.
We handle the second series in an indirect way, using $f^\sharp$.
Although $f^\sharp \in \OO^2(X^{\dug})$,
we have $\overline{f^\sharp} \in \OO^1(X^{\dug}) \tensor_R k$
by \eqref{erasasterb}.
So~\eqref{E:O1}
identifies $\overline{f^\sharp}$ with a polynomial in
$\OO(\bar{X}^{\dug})[\overline{\varphic}] \subset L[\overline{\varphic}]$,
where $L:=k(\overline{X_1(N)})$.
We can find this polynomial explicitly from the $\d$-Fourier expansion,
since elements of $L$ have expansions in $k((q))$
while $\overline{\varphic_\infty}$ involves $q'$: see~\eqref{E:varphic mod p}.
By Lemma~\ref{floarenoua} and \eqref{E:varphic mod p},
\[
\overline{f^{\sharp}_{\infty}} =
\begin{cases}
\overline{f^{(-1)}(q)}
 + t_{\infty}^{-\frac{p^2+p}{2}}
 \left(\overline{f^{(0)}_{[a_p]}(q)}\right)^{p^2} \overline{\varphic_\infty}^p
 - \bar{a}_p t_{\infty}^{-\frac{p+1}{2}}
 \left(\overline{f^{(0)}_{[a_p]}(q)}\right)^{p} \overline{\varphic_\infty},
		& \text{ if $A$ is not $\CL$;} \\ \\
-\bar{u} \overline{f^{(-1)}_{[u]}(q)}+
   t_{\infty}^{-\frac{p+1}{2}}
 \left(\overline{f^{(0)}_{[a_p]}(q)}\right)^{p} \overline{\varphic_\infty},
		& \text{ if $A$ is $\CL$.}
\end{cases}
\]
In either case, taking the coefficient of $\overline{\varphic_\infty}$ 
shows that 
	$\bar{a}_p t_{\infty}^{-\frac{p+1}{2}} 
	\left(\overline{f^{(0)}_{[a_p]}(q)} \right)^{p}$
is the Fourier expansion of an element of $L$.
Since $t$ is a weight $2$ quotient,
$\bar{a}_p \left(\overline{f^{(0)}_{[a_p]}(q)} \right)^{p}$
is the Fourier expansion of a weight $p+1$ quotient,
and hence (by dividing by $\bar{E}_{p-1}$)
also of a weight-$2$ quotient.
By~\eqref{fructnou2},
\[
	- \bar{a}_p \left(
\overline{f^{(0)}_{[a_p]}(q)}\right)^p+\overline{f^{(0)}_{[a_p]}(q)}=
\overline{f^{(0)}(q)};
\]
now $\overline{f^{(0)}_{[a_p]}(q)}$ is the Fourier expansion of
a weight-$2$ quotient since the other terms are.
\end{proof}

\begin{remark}
The proof that $\overline{f^{(0)}_{[a_p]}(q)}$ is a Fourier expansion
of a quotient of modular forms used the theory of $\d$-modular forms;
we know no direct proof.
\end{remark}

Recall the Igusa curve $I_1(N)$ and its quotient $J$ defined in
Section~\ref{S:refinement for isogeny classes}.

\begin{lemma}
\label{L:series is modular function}
The Fourier series of any modular form $f$ on $X_1(N)$ over $k$
is also the Fourier series of a rational function $g \in k(I_1(N))$.
If the weight of $f$ is even, then we may take $g \in k(J)$.
\end{lemma}

\begin{proof}
By~\cite[Proposition~2.2]{gross},
there is a line bundle $\omega$ on $\overline{X_1(N)}$
such that for each $i \in \Z$, 
the global sections of $\omega^i$ are the modular forms of weight $i$.
We denote by $\omega$ also the pullback of $\omega$ to $I_1(N)$ or $J$.
By~\cite[p.~461]{gross},
the sections of $\omega^i$ on $I_1(N)$ or $J$ have naturally defined
Fourier expansions, compatible with the Fourier expansions 
of modular forms on $X_1(N)$.
There is a section $a$ of $\omega$ on $I_1(N)$ whose Fourier expansion
is $1$: see~\cite[Proposition~5.2]{gross}.
Given a modular form $f$ of weight $i$ on $X_1(N)$,
let $g:=f/a^i \in k(I_1(N))$.

The action of $\F_p^\times$ on $I_1(N)$ lifts to an action 
of $\F_p^\times$ on $\omega$,
and $-1 \in \F_p^\times$ sends $a$ to $-a$
(see~\cite[Proposition~5.2(5)]{gross}),
so if $i$ is even, $f/a^i \in k(J)$.
\end{proof}

By Construction~3.2 and Theorem~5.1 of~\cite{Barcau},
there exist unique $\d$-modular forms 
$f^{\partial} \in M^1_{\{E_{p-1}\}}(\phi-1)$ and 
$f_{\partial} \in M^1_{\{E_{p-1}\}}(1-\phi)$, defined over $\bZ_p$, with
$\d$-Fourier expansions identically equal to $1$. 
Moreover, these forms are isogeny covariant 
and $f^{\partial} \cdot f_{\partial}=1$.
Furthermore, the reduction $\overline{f^{\partial}} \in M^1 \tensor k$
equals the image of $\bar{E}_{p-1}\in M_{p-1}$ in $M^1 \tensor k$.
For $\lambda \in R^{\times}$, define
\medskip
\begin{equation}
\label{flam} 
	f_{\lambda}  :=  (f^1)^{\phi}-\lambda f^1
(f^{\partial})^{-\phi-1} \in M^2_{\{E_{p-1}\}}(-\phi-\phi^2).
\end{equation}
\medskip
Since $f_1$ and $f^\partial$ are isogeny covariant,
so is $f_{\lambda}$.
Furthermore consider the series
\[t^{\frac{\phi^2+\phi}{2}}:=t^{\frac{p^2+p}{2}} \left(
\frac{t^{\phi}}{t^p} \right)^{1/2} \left( \frac{t^{\phi^2}}{t^{p^2}}
\right)^{1/2} \in M^2_{\{a_4a_6\}}(\phi+\phi^2),\] and
 define
\begin{equation}
\label{flats} f^{\flat}_{\lambda}  :=  f_{\lambda} \cdot
t^{\frac{\phi^2+\phi}{2}} \in M^2_{\{a_4a_6E_{p-1}\}}(0).
\end{equation}
\medskip
The main reason for considering these forms comes from the
following
\begin{lemma}
\label{9438hfb} 
Let $E_1$ be an elliptic curve $y^2=x^3+A_1x+B_1$ over $R$ 
with ordinary reduction.  Then
\begin{enumerate}
\item 
There exists $\lambda \in R^{\times}$ such that $f_{\lambda}(A_1,B_1)=0$.
\item
If $\lambda$ is as in (1) and there is an isogeny of
degree prime to $p$ between $E_1$ and an elliptic curve $E_2$
over $R$ given by $y^2=x^3+A_2x+B_2$, then $f_{\lambda}(A_2,B_2)=0$.
\item
If in addition, $A_2B_2 \not\equiv 0 \pmod{p}$, then 
$f_{\lambda}^{\flat}(A_2,B_2) = (\d f_{\lambda}^{\flat})(A_2,B_2) 
= \cdots = 0$.
\end{enumerate}
\end{lemma}

\begin{proof}
\hfill
\begin{enumerate}
\item
If $f^1(A_1,B_1) = 0$, any $\lambda \in R^\times$ will do.
If $f^1(A_1,B_1) \ne 0$, set
\[
	\lambda:=\frac{f^1(A_1,B_1)^{\phi}}{f^1(A_1,B_1)}
		f^{\partial}(A_1,B_1)^{\phi+1};
\]
the numerator and denominator of the first factor have the same
$p$-adic valuation and 
$\overline{f^{\partial}(A_1,B_1)} \equiv \bar{E}_{p-1}(\bar{A},\bar{B}) \ne 0$,
so $\lambda \in R^\times$.
\item
Scaling $A_2$ and $B_2$ by suitable elements of $R^\times$,
we may assume that the isogeny pulls back $dx/y$ to $dx/y$.
Now use the isogeny covariance of $f_{\lambda}$.
\item
By \eqref{flats}, $f_{\lambda}^{\flat}(A_2,B_2) = 0$.
Now use $\delta 0 = 0$.
\end{enumerate}
\end{proof}

Set $\sigma:=q'/q^p$.
Then \eqref{maus}, \eqref{flam}, and~\eqref{flats} yield
\begin{equation}
\label{somefourexp}
	\overline{f^1_{\infty}} = \sigma,
	\qquad
	\overline{f_{\lambda,\infty}} = \sigma^p-\overline{\lambda} \sigma,
	\qquad \text{ and} \qquad
	\overline{f^{\flat}_{\lambda,x_{\infty}}} = 
	   t_{\infty}^{\frac{p^2+p}{2}}(\sigma^p- \overline{\lambda} \sigma).
\end{equation}
In what follows we assume that $X^{\dug}=U \setminus [x_{\infty}]$ 
where $U$ has an \'etale coordinate $\tau \in \cO(U)$ such that
$[x_{\infty}]$ is scheme-theoretically given by $\tau$: 
we can arrange this by shrinking $X^{\dug}$. 
Then $R[[\qq]]=R[[\tau]]$, so
\[
	R((\tau))\h[\tau',\ldots ,\tau^{(r)}]\h
	=R((\qq))\h[\qq',\ldots ,\qq^{(r)}]\h
	=R((\qq))\h[q',\ldots ,q^{(r)}]\h.
\] 
Also
$\cO^r(X^{\dug})=\cO(X^{\dug})\h[\tau',\ldots ,\tau^{(r)}]\h$.  
Since
\begin{equation}
\label{E:f-lambda-flat-bar}
\overline{f_{\lambda}^{\flat}} \in
\cO(\bar{X}^{\dug})[\tau',\tau''] \cap k((\tau))[\tau']  = 
\cO(\bar{X}^{\dug})[\tau']  =
\cO^1(X^{\dug})\tensor_R k,
\end{equation}
we may define a quotient ring
\begin{equation}
\label{thering}
	\cB:=(\cO^1(X^{\dug})\tensor_R k)/(\overline{f_{\lambda}^{\flat}})
\end{equation}
and a scheme $\bar{X}^{\ddug}:=\Spec \cB$.
View $\cB$ as an algebra over 
$\cA:=\cO(X^{\dug})\otimes k=\cO(\bar{X}^{\dug})$.

\begin{lemma}
\label{coppies}
The $k((\qq))$-algebra $\cB \otimes_{\cA} k((\qq))$ is a product
of $p$ copies of $k((\qq))$.
\end{lemma}

\begin{proof}
We have
\begin{align*}
\cB \otimes_{\cA} k((\qq)) & = 
\left(\cO(\bar{X}^{\dug})[\tau']/ (\overline{f_{\lambda}^{\flat}})\right) 
	\tensor_{\cA} k((\tau))\\
&= k((\tau))[\tau']/ (\overline{f_{\lambda}^{\flat}})\\
&= k((\qq))[q']/(\overline{f^{\flat}_{\lambda,x_{\infty}}})\\
&= k((\qq))[\sigma]/ (\sigma^{p}-\bar{\lambda} \sigma)\\
&\simeq \prod_{i=1}^{p} k((\qq)),
\end{align*}
since $\sigma^{p}-\bar{\lambda}\sigma = \prod_{i=1}^p (\sigma-\lambda_i)$
for some $\lambda_i \in k$.
Explicitly, the last isomorphism is given by
\begin{equation}
\label{iuegfyc}
	q' \mapsto (\lambda_1 q^p,\ldots ,\lambda_p q^p).
\end{equation}
\end{proof}

\begin{lemma}
\label{uewrybcxmzni} 
One can choose $X^{\dug}$ so that
$\bar{X}^{\ddug} \to \bar{X}^{\dug}$ is a finite \'etale cover
of degree $p$.
\end{lemma}

\begin{proof}
By definition, $\bar{X}^{\ddug} \to \bar{X}^{\dug}$ is of finite type.
Lemma \ref{coppies} shows that it is \'etale of degree $p$
above the generic point of $\bar{X}^{\dug}$.
Therefore $\bar{X}^{\ddug} \to \bar{X}^{\dug}$ is finite \'etale
of degree $p$ over some open neighborhood of the generic point.
\end{proof}

In case our correspondence arises from a modular parametrization
one has the following variant of Lemma \ref{coppies}.

\begin{lemma}
\label{shadap}
Assume $\thecorr$ arises from a modular parametrization and let
$L=k(\overline{X_1(N)})$.
Then
\[
	\cB \tensor_{\cA} L \isom L \times \calA^+ \times \calA^-
\]
where
\[
	\calA^{\pm} := L[y]/\left(y^{(p-1)/2} - \bar{E}_{p-1}/t^{(p-1)/2} \right).
\]
\end{lemma}

\begin{proof}
By \eqref{E:O1}, we have $\cB \otimes_{\cA} L \simeq L
[\overline{f^{\flat}}]/(\overline{f_{\lambda}^{\flat}})$. 
On the other hand
\begin{align*}
\overline{f^{\flat}_{\lambda}} 
&= t^{\frac{p^2+p}{2}} \left[ (\overline{f^1})^p - \bar{\lambda} \overline{f^1}(\overline{f^{\partial}})^{-p-1} \right]\\
&= (\overline{f^{\flat}})^p-\lambda
t^{\frac{p^2-1}{2}}\bar{E}_{p-1}^{-p-1}\overline{f^{\flat}}\\
&= \overline{f^{\flat}} \left[(\overline{f^{\flat}})^{(p-1)/2}
+\sqrt{\lambda} t^{(p-1)/2}\bar{E}_{p-1}^{-1} (t^{(p-1)/2}/\bar{E}_{p-1})^{(p-1)/2}
\right] \\
& \phantom{= \overline{f^{\flat}}} \cdot \left[(\overline{f^{\flat}})^{(p-1)/2}
-\sqrt{\lambda} t^{(p-1)/2}\bar{E}_{p-1}^{-1} (t^{(p-1)/2}/\bar{E}_{p-1})^{(p-1)/2}
\right],
\end{align*}
so the result follows.
\end{proof}

\subsection{$\d$-modular forms: Shimura-elliptic case} 
We continue using the notation and assumptions of
Sections \ref{S:delta Serre-Tate} and~\ref{S:pullbacks}.
Assume that the $U$ in \eqref{vaiy2} is small enough 
that the {\em line bundle of false $1$-forms} on $U$ 
(see~\cite[p.~230]{book}) 
is trivial.
Let $q:=1+t \in R[[t]]$ and write $q'=\d(1+t)$, $q''=\d^2(1+t)$, and so on.
Define
\[
	\Psi=\Psi(t,t'):=\frac{1}{p} \log \frac{q^{\phi}}{q^p}
	=\frac{q'}{q^p}-\frac{p}{2} \left(\frac{q'}{q^p} \right)^2+\cdots  
		\in R[[t]][t']\h.
\]
By (8.116), (8.82), and Proposition~8.61 in~\cite{book},
one can find a series $u(t) \in R[[t]]^{\times}$ and a  function $f^{\flat}
  \in \cO^1(U)$ such that
\begin{equation}
\label{kellly}
	f^{\flat}_{x_0} = u(t)^{\phi+1} \cdot \Psi(t,t') \in q' R[[t]][t']\h,
\end{equation}
and 
\begin{equation}
\label{okellly} 
	f^{\flat}(P)=0 \quad \text{ for $P\in \Pi^{-1}(\CL) \intersect U(R)$}.
\end{equation}
(In the notation of~\cite{book}, one takes $f^{\flat}$ to be  the
value of the ``$\d$-modular form'' $f^1_{\crys}$ at the pull back
to $U$ of the universal false elliptic curve equipped with some
invertible false $1$-form; again $f^1_{\crys}$ should be viewed as
an arithmetic Kodaira-Spencer class.)

\begin{lemma}
\label{lemuli} 
There exists a neighborhood $X^{\dug} \subset U$ of
the section $x_0$ such that 
$\overline{\varphic}$ and $\overline{\d \varphic}$ 
are algebraically independent over $\OO(\bar{X}^{\dug})$
and the natural maps
\begin{align}
\label{E:SO1}
	\OO(\bar{X}^{\dug})[\overline{\varphic}] 
		&\to \OO^1(X^{\dug}) \tensor_R k \\
\label{E:SO2}
	\OO(\bar{X}^{\dug})[\overline{\varphic}, \overline{\d \varphic}]
		&\to \OO^2(X^{\dug}) \tensor_R k \\
\label{E:SO3}
	\cO(X^{\dug})\h
		&\to \cO^2(X^{\dug})/(\varphic,\d \varphic)
\end{align}
are isomorphisms.
\end{lemma}

\begin{proof}
By \eqref{kellly},
\begin{equation}
\label{cocota1}
	\overline{f^{\flat}_{x_0}} 
	= \frac{\bar{u}(t)^{p+1}}{(1+t)^p} t'+ S_0 \in k[[t]][t'],
\end{equation}
for some $S_0 \in k[[t]]$.
Using \eqref{axioms} one obtains
\begin{equation}
\label{cocota2}
	\overline{\d f^{\flat}_{x_0}} =
	\frac{\bar{u}(t)^{p^2+p}}{(1+t)^{p^2}} t''+ S_1
	\in k[[t]][t',t'']
\end{equation}
for some $S_1 \in k[[t]][t']$.
We may assume that there is an \'etale coordinate $\tau$ on $U$
such that $x_0$ is given scheme-theoretically by $\tau=0$. 
Then $R[[t]]=R[[\tau]]$ (and $R[[t]][t',t'']\h=R[[\tau]][\tau',\tau'']\h$) 
so $t=S(\tau):=\sum_{n \geq 1} c_n \tau^n$ for some $c_n \in R$
with $c_1 \in R^{\times}$. 
One can easily see that
\[
	t'=\frac{1}{p}\left[ \sum c_n^{\phi}(\tau^p+p\tau')^n-
			\left( \sum c_n \tau^n \right)^p \right]
	=(\partial S/\partial \tau)^p \tau'+B_0+pB_1
\]
for some $B_0 \in R[[\tau]]$ and $B_1 \in R[[\tau]][\tau']\h$.
Using \eqref{axioms} we obtain
\[
	t''=(\partial S/\partial \tau)^{p^2} \tau'' + B_1^* + pB_2
\] 
for some $B_1^* \in R[[\tau]][\tau']\h$ 
and $B_2 \in R[[\tau]][\tau',\tau'']\h$. 
Combining with \eqref{cocota1} and~\eqref{cocota2}
and setting
\[
	\bar{v}(\tau):=\frac{\bar{u}(\bar{S}(\tau))^{p+1} 
		(\partial \bar{S}/\partial \tau)^p}
	{(\bar{S}(\tau)+1)^p} \in k[[\tau]],
\]
we obtain
\begin{align}
\label{yiyiyi}
	\overline{f^{\flat}_{x_0}} 
	&= \bar{v}(\tau) \tau'+C_0(\tau) \in k[[\tau]][\tau'],\\
\notag	\overline{\d f^{\flat}_{x_0}} 
	&= \bar{v}(\tau)^p \tau''+C_1(\tau,\tau') \in k[[\tau]][\tau',\tau''].
\end{align}
where $C_0(\tau) \in k[[\tau]]$ and $C_1(\tau,\tau') \in k[[\tau]][\tau']$. 
On the other hand, by \eqref{zim},
we have $\overline{f^{\flat}}\in \cO(\bar{U})[\tau']$ and
$\overline{\d f^{\flat}}\in \cO(\bar{U})[\tau',\tau'']$. 
Thus $\bar{v}(\tau)$, $C_0(\tau)$, and $C_1(\tau,\tau')$ 
are images of elements $\bar{v} \in \cO(\bar{U})$,
$C_0 \in \cO(\bar{U})$, and $C_1 \in \cO(\bar{U})[\tau']$, respectively,
such that
\begin{equation}
\label{yiyiyi2}
	\overline{f^{\flat}} = \bar{v} \tau'+C_0, \quad\text{ and }\quad
	\overline{\d f^{\flat}} = \bar{v}^p \tau''+C_1.
\end{equation}
Lift $\bar{v}$ to $v \in \cO(U)$.
Let $X^{\dug}$ be the complement in $U$ 
of the closed subscheme defined by $v$. 
Since $\bar{v}(\tau)$ has a nonzero constant term, 
$\bar{v}$ does not vanish at $\bar{x}_0$,
so $X^{\dug}$ contains the section $x_0$.
The proof now follows the proof of Lemma~\ref{lem1},
using \eqref{yiyiyi2} in place of \eqref{lolla}.
\end{proof}

\begin{remark}
\label{cecece} 
Using~\cite[pp.~268--269]{book}, 
for $\bar{U}$ contained in the ordinary locus
one can construct forms
$f^{\flat}_{\lambda} \in \cO^2(U)$ 
analogous to the ones in \eqref{flats}. 
(In the notation of~\cite{book}, one takes
$f^{\flat}_{\lambda}$ to be the Shimura analogues of the forms in
\eqref{flam} evaluated at a basis of the module of false
$1$-forms on $U$.) The analogues of Lemmas \ref{9438hfb},
\ref{coppies}, and~\ref{uewrybcxmzni} still hold with Fourier
expansions replaced by Serre-Tate expansions. The corresponding
statements and their proofs are analogous to the ones in
the modular-elliptic case.
\end{remark}

\subsection{Proofs of the local results}
\label{S:proofs of local results}
\begin{proof}[Proof of Theorem~\ref{mainth}]
Assume that we are given either a modular-elliptic  or a
Shimura-elliptic correspondence $\thecorr$.
Assume that $p$ is sufficiently large and $p$ splits completely in $F_0$. 
In the Shimura-elliptic case we also assume that the places $v|p$ 
are not anomalous for $A$. 
In the modular-elliptic case, 
choose $g$ as in Lemma~\ref{lem1}
and define $X^{\dug}$ as in~\eqref{cocer}. 
In the Shimura-elliptic case, choose $X^{\dug}$ as in Lemma~\ref{lemuli}.
By Lemma \ref{lem1} or~\ref{lemuli}, 
there exists $\Sigmat \in \cO(X^{\dug})\h$ such that
\begin{equation}
\label{aproape}  
	f^{\sharp}-\Sigmat = h_0 \varphic+h_1 \d \varphic,
\end{equation} 
for some $h_j \in \cO^2(X^{\dug})$. 
Suppose that $P_1,\ldots ,P_n \in \Pi^{-1}(\CL) \cap X^{\dug}(R)$ and 
$m_1,\ldots,m_n \in \Z$. 
By \eqref{gugu} or~\eqref{okellly}, we have $\varphic(P_i)=0$,
so $\d \varphic(P_i)=0$. 
Thus
\begin{equation}
\label{cliopr}
	f^{\sharp}(P_i)=\Sigmat(P_i).
 \end{equation}
Now \eqref{cliopr} implies
\begin{equation}
\label{easyrider} 
	\sum m_i\Sigmat(P_i) 
	= \sum m_i f^{\sharp}(P_i) 
	= \sum m_i\psi(\Phi(P_i))
	= \psi \left(\sum m_i\Phi(P_i) \right).
\end{equation}
Equation~\eqref{easyrider} and Lemma~\ref{thofker}(4) imply 
the second of the two equivalences in Theorem~\ref{mainth}. 

We now prove the first equivalence in Theorem~\ref{mainth}.
Let $Q:=\sum m_i\Phi(P_i)$.
If $Q \in A(R)_{\tors}$, then $\psi(Q)=0$
and \eqref{easyrider} implies $\sum m_i\Sigmat(P_i)=0$.
Conversely, suppose that $\sum m_i\Sigmat(P_i)=0$;
then $Q \in \ker \psi$.
Since $\CL \subseteq S(\Qbar)$, 
we have $P_i \in X(\Qbar) \intersect X(R)$,
so $Q \in A(\Qbar) \intersect A(R) \subset A(\Z_p^{\ur})$.
So Lemma~\ref{thofker}(5) implies $Q \in A(R)_{\tors}$.

To complete our proof, we need to check that $\overline{\Sigmat} \notin k$.

Assume first that we are in the modular-elliptic case. 
By \eqref{elephant},
\begin{equation}
\label{elephantt} \d \varphic_{\infty}  \in (\qq',\qq'')
R((\qq))\h[\qq',\qq'']\h.
\end{equation}
Taking $\d$-Fourier expansions in \eqref{aproape},
taking $\natural$ (i.e., setting $\qq'=\qq''=0$), 
and using \eqref{elephant} and~\eqref{elephantt}, 
we obtain
\begin{equation}
\label{tototo} 
	\Sigmat_{x_{\infty}} = (f^{\sharp}_{x_{\infty}})_{\natural}.
\end{equation} 
Let $e$ be the ramification index of $\Phi\colon X \to A$ at $x_{\infty}$.
Then the $b_e \in F_0$ of Section~\ref{S:modular-elliptic pullback}
is nonzero.
We may assume that $p$ is large enough that $e,b_e \not\equiv 0 \pmod{p}$.
By \eqref{tototo} and~\eqref{keti}, the coefficient of $\qq^e$
in $\Sigmat_{x_{\infty}}$ is $\frac{b_e}{e}$ or $-u \frac{b_e}{e}$,
where $u \not\equiv 0 \pmod{p}$;
in either case this coefficient is nonzero mod $p$.
Thus $\overline{\Sigmat_{x_{\infty}}} \notin k$.
Hence $\overline{\Sigmat} \notin k$.

Finally, assume that we are in the Shimura-elliptic case. 
By~\eqref{kellly},
 \begin{equation}
\label{popopo} 
	\d f^{\flat}_{x_0} \in (q',q'')R[[t]][t',t'']\h.
\end{equation}
By~\eqref{axioms},
\begin{equation}
\label{pesste}
\begin{array}{rcl}
q'& = & t'-G_1(t)\\
q'' & = & t''-G_2(t,t'),
\end{array}
\end{equation}
for some $G_1(t) \in \Z[t]$ and $G_2(t,t') \in \Z[t,t']$. 
Denote by $G \mapsto G_\natural$ the substitution homomorphism
\[
R[[t]][t',t'']\h \to R[[t]]
\]
sending $t'$ to $G_1(t)$ and $t''$ to $G_2(t,t')$.
Then $(q')_{\natural}=(q'')_{\natural}=0$, so 
\eqref{kellly} and~\eqref{popopo} imply
$(f_{x_0}^{\flat})_{\natural}=(\d f^{\flat}_{x_0})_{\natural}=0$.
Taking $\d$-Serre-Tate expansions in \eqref{aproape},
taking $\natural$, and substituting \eqref{titish},
we obtain
\begin{equation}
\label{titishhh}
\Sigmat_{x_0} = 
\begin{cases}
\psi(z_0)+ \frac{1}{p} \sum_{n\geq 1} \left(
\frac{b_n^{\phi^2}}{n} ((t^{\phi^2})_{\natural})^n -a_p
\frac{b_n^{\phi}}{n}
((t^{\phi})_{\natural})^n +p \frac{b_n}{n} t^n \right),
&\text{ if $A$ is not $\CL$;}\\
\psi(z_0)+\frac{1}{p} \sum_{n\geq 1} \left(
 \frac{b_n^{\phi}}{n}
((t^{\phi})_{\natural})^n -up \frac{b_n}{n} t^n \right),
&\text{ if $A$ is $\CL$.}
\end{cases}
\end{equation}
Substituting the two formulas
\begin{align*}
(t^{\phi})_{\natural} 
	&= (q^{\phi}-1)_{\natural} 
	= (q^p+pq'-1)_{\natural}
	= q^p-1 
	= (1+t)^p-1
	= pt + \cdots + t^p, \text{ and}\\
(t^{\phi^2})_{\natural}
	&= (q^{\phi^2}-1)_{\natural} 
	= ((q^p+pq')^p+p((q')^p+pq'')-1)_{\natural}
	= q^{p^2}-1 
	= (1+t)^{p^2}-1
	= p^2 t + \cdots + t^{p^2},
\end{align*}
and recalling from Section~\ref{S:Shimura pullback} that $b_1=1$,
we deduce that the coefficient of $t$ in $\Sigmat_{x_0}$
is $1-a_p+p$ if $A_R$ is not CL,
and $1-u$ if $A_R$ is CL.
This coefficient is nonzero mod $p$,
since our non-anomalous assumption implies $a_p \not\equiv 1 \pmod{p}$
and we have $\bar{u} \bar{a_p} = 1$ in the CL case.
So $\Sigmat_{x_0} \notin R+ pR[[t]]$.  
Hence $\overline{\Sigmat} \notin k$.
\end{proof}

\begin{remark}
Using the naturality of $\psi$ and the isogeny covariance of
$f^{\flat}$ and $f^{\partial}$, we easily verify
that $\Sigmat$ satisfies the functoriality properties in Remark~\ref{kokk}.
\end{remark}

\begin{proof}[Proof of Theorem~\ref{refined}] Assume, in the proof of Theorem
\ref{mainth}, that our modular-elliptic correspondence $\thecorr$
satisfies $S=X=X_1(N)$, $\Pi=\Id$, and $\Phi$ is a modular
parametrization attached to a newform $f$. 
We may choose $g:=E_{p-1}$ in Section~\ref{S:d-modular-elliptic}; then 
$\bar{X}^{\dug}=\overline{Y_1(N)}^{\ord}\setminus \{x \mid j(x)=0,1728\}$.
Now \eqref{floare} and~\eqref{tototo} give the formula for $\Sigmat_\infty$.
\end{proof}

\begin{lemma}
\label{nuzzi} 
Suppose that $\overline{\Theta} \in \cF$. 
Assume that for every $P \in {\mathcal P}$ and every prime $l\neq p$ we have
\begin{equation}
\label{gura2} \sum_i \overline{\Theta}(\bar{P}_i^{(l)})-
a_l\overline{\Theta}(\bar{P})=0.\end{equation} Then
$\overline{\Theta}=\bar{\lambda} \overline{f^{(-1)}}$ for
some $\bar{\lambda} \in k$.
\end{lemma}

\begin{proof}
Since $\overline{\Theta}$ is regular on $\overline{X_1(N)}^\ord$,
there exists $m \in \Z_{\ge 1}$ 
such that $\bar{G}:=\bar{E}_{p-1}^m \overline{\Theta}$ is a modular form
over $k$ on $\Gamma_1(N)$.
View modular forms as functions on the set of triples $(E,\alpha,\omega)$
where $E$ is an elliptic curve over $k$, 
where $\alpha \colon \Z/n\Z \injects E(k)$ is an injective homomorphism,
and $\omega$ is a nonzero $1$-form on $E$.
Given $P \in \calP$ and a prime $l \ne p$, 
choose $(E,\alpha,\omega)$ such that $(E,\alpha)$ represents $P$,
and choose $(E_i,\alpha_i,\omega_i)$ such that $(E_i,\alpha_i)$
represents $P_i^{(l)}$
and such that $\omega_i$ pulls back to $\omega$ under the $l$-isogeny
$E \to E_i$.
Then $\bar{E}_{p-1}(E_i,\alpha_i,\omega_i) = \bar{E}_{p-1}(E,\alpha,\omega)$
by~\cite[p.~269]{book}, for instance.
Multiplying \eqref{gura2} by this yields
\[
	\sum_i \bar{G}(E_i,\alpha_i,\omega_i) = a_l \bar{G}(E,\alpha,\omega).
\]
By \cite[p.~452]{gross} or~\cite[p.~90]{Katz},
the left hand side equals $(l T(l) \bar{G})(E,\alpha,\omega)$.
Since $\calP$ is infinite, it follows that $l T(l) \bar{G} = a_l \bar{G}$
for all $l \ne p$.
On the other hand, $\bar{G}(q)=\overline{\Theta}(q)$,
and $U\overline{\Theta}=0$, so $U\bar{G}=0$.
Furthermore $\bar{G}$ is invariant under the diamond operators. 
Thus $\bar{G}$ is a Hecke eigenform with the same eigenvalues 
as $\theta^{p-2}\bar{f}$, so by~\cite[p.~453]{gross},
we have $\bar{G}(q) = \bar{\lambda} \cdot (\theta^{p-2}\bar{f})(q)$
for some $\bar{\lambda} \in k$.
Thus $\overline{\Theta}(q) = \bar{\lambda} \overline{f^{(-1)}}(q)$,
so $\overline{\Theta} = \bar{\lambda} \overline{f^{(-1)}}$.
\end{proof}

\begin{proof}[Proof of Theorem~\ref{isoggg}]
Assume that we have a modular-elliptic correspondence. 
Pick $Q \in C$  represented by
$(E_Q,\alpha_Q)$ where $E_Q$ is given by $y^2=x^3+Ax+B$. 
By Lemma~\ref{9438hfb}(1), there exists $\lambda \in R^{\times}$ such that
$f_{\lambda}(A,B)=0$. 
Let $X^{\dug} \subset X$ satisfy the conclusions of Lemmas~\ref{lem1} 
and~\ref{uewrybcxmzni}. 
View $f_{\lambda}^{\flat}$ and $f^{\sharp}$ as elements of $\cO^2(X^{\dug})$; 
then $\overline{f_{\lambda}^{\flat}}, \overline{f^\sharp} 
\in \cO^1(X^{\dug}) \tensor_R k$
by \eqref{E:f-lambda-flat-bar} and \eqref{erasasterb},
respectively.
Let $\overline{\Sigmatt}$ be the image
of $\overline{f^{\sharp}}$ in the ring $\cB=\cO(\bar{X}^{\ddug})$
of~\eqref{thering}.

{\em Claim.} $\overline{\Sigmatt}$ is non-constant on each 
irreducible component of $\bar{X}^{\ddug}$.

If not, there is a minimal prime $\cP$ of $\cB$ such that
the image of $\overline{\Sigmatt}$ in $\cB/\cP$, and hence in
$(\cB/\cP) \tensor_{\cA} k((\qq))$, is in $k$. 
By Lemma~\ref{coppies},
$(\cB/\cP) \tensor_{\cA} k((\qq))$ 
is a nonzero product of copies of $k((\qq))$. 
By~\eqref{iuegfyc}, the element
\[\overline{f^{\sharp}_{x_{\infty}}} \in
k[[\qq]][\qq']\subset k((\qq))[q']\] is sent into  an
element of $k$ by at least one of the $k((\qq))$-algebra
homomorphisms
\begin{equation}
\label{acompo} 
	k((\qq))[q'] \to k((\qq)),
\end{equation}
denoted $s \mapsto s_*$ and defined by $(q')_*:=\lambda_i q^p$. 
Since $q = \qq^M$, we have
\[
	q' = \d(\qq^M) = \frac{(\qq^p + p \qq')^M - \qq^{pM}}{p} \equiv M \qq^{p(M-1)} \qq' \pmod{p},
\]
so $\qq' \equiv M^{-1} \qq^{-p(M-1)} q' \pmod{p}$.
Thus 
\[
	(\qq')_* = M^{-1} \qq^{-p(M-1)} \lambda_i q^p = M^{-1} \lambda_i \qq^p \in \qq^p k[[\qq]].
\]
Hence
\[
	(\overline{f^{\sharp}_{x_{\infty}}})_*
	\in
	(\overline{f^{\sharp}_{x_{\infty}}})_{\natural}+\qq^pk[[\qq]],
\]
where we recall that $\natural$ means setting $\qq'=0$.
Let $e$ be the ramification index of $\Phi$ at $x_{\infty}$.
Exactly as in the proof of Theorem~\ref{mainth}, 
since $p \gg 0$, the coefficient of $\qq^e$ in
$(\overline{f^{\sharp}_{x_{\infty}}})_\natural$ is nonzero. So
$(\overline{f^{\sharp}_{x_{\infty}}})_*$ is not in $k$, 
a contradiction. This ends the proof of our Claim.

Now consider the set $\cC:=\Pi^{-1}(C) \cap X^{\dug}(R)$ and let
$P_1 \in \cC$,  $Q_1:=\Pi(P_1)$. Let $E_{Q_1}$ be given by
$y^2=x^3+A_1x+B_1$. 
By choice of $X^{\dug}$, we have $A_1B_1 \not\equiv 0 \pmod{p}$. 
By Lemma~\ref{9438hfb}, $f_{\lambda}^{\flat}(P_1)=0$.
Therefore the homomorphism $\cO^1(X^{\dug}) \to R$
sending a function to its value at $P_1$
induces a homomorphism $\cB \to k$,
which may be viewed as a point $\sigma(P_1) \in \bar{X}^{\ddug}(k)$
mapping to $P_1 \in \bar{X}^{\dug}(k)$.
This defines $\sigma\colon \calC \to \to \bar{X}^{\ddug}(k)$.
By definition of $\sigma(P_1)$ and $\Sigmatt$,
$\overline{f^{\sharp}(P_1)}=\overline{\Sigmatt}(\sigma(P_1))$.
Now, for $P_1,\ldots ,P_n \in \cC$,
\begin{align*}
\sum_{i=1}^n m_i \overline{\Sigmatt}(\sigma(P_i))
&= \sum_{i=1}^n m_i\overline{f^{\sharp}(P_i)}\\
&= \sum_{i=1}^n m_i\overline{\psi(\Phi(P_i))}\\
&= \overline{\psi \left(\sum_{i=1}^n m_i\Phi(P_i) \right)},
\end{align*}
so the desired equivalence follows from Lemma~\ref{thofker}(4).

The case of Shimura-elliptic correspondences is entirely similar,
given Remark~\ref{cecece}. 
We skip the details but point out one slight difference 
in the computations. 
The proof of the analogue of the Claim above, 
uses a $k[[t]]$-algebra homomorphism
\[
	k[[t]][t'] \to k[[t]],
\]
denoted $s \mapsto s_*$,
defined by requiring $(q')_*=\lambda_iq^q$,
where $q=1+t$ and $q'=\d(1+t)$. 
Then one must check that for $f_{x_0}^{\sharp}$ as in~\eqref{titish}, 
the coefficient of $t$ in $(f^{\sharp}_{x_0})_*$ 
is nonzero mod $p$. 
This coefficient can be computed explicitly, and, unlike in the
modular-elliptic case, its expression  has contributions from all
the terms with $n \geq 1$.
Nevertheless all the contributions from terms with $n \geq 2$ turn
out to be $0$ mod $p$, and the coefficient in question turns
out to be congruent mod $p$ to either $1-a_p$ or $1-u$,
and hence is nonzero mod $p$.
\end{proof}

The following will be used to prove Corollary~\ref{bizzet}:

\begin{lemma}
\label{berliozz} 
Under the assumptions of Corollary \ref{bizzet}
there is a constant $\gamma$ depending only on $N$ such that all
the fibers of the reduction mod $p$ map $C \to \overline{C}$ are
finite of cardinality at most $\gamma$.
\end{lemma}

\begin{proof}
Assume that we are in the modular-elliptic case; the Shimura-elliptic
case follows by the same argument. 
Suppose that $Q_1,Q_2 \in C$ are such that $\bar{Q}_1=\bar{Q}_2 \in S(k)$.
Let $Q_i$ be represented by $(E_i,\alpha_i)$,
so there is an isogeny $u \colon E_1 \to E_2$
of degree $\prod l_j^{e_j}$ where the $l_j$ are inert in $\calK_Q$.

We claim that $E_1 \isom E_2$.
Since $\bar{E}_1 \isom \bar{E}_2$, we may view $\bar{u}$ as an element
of $\End \bar{E}_1$, which may be identified with a subring of the
ring of integers $\OO$ of $\calK_Q$.
The norm of this element equals $\deg \bar{u} = \deg u$,
but the only elements of $\OO$ whose norm is a product of inert primes
are those in $\Z \cdot \OO^\times$.
Hence $u$ factors as $E_1 \stackrel{n}\to E_1 \stackrel{\epsilon}\to E_2$
for some $n \in \Z$ and $\epsilon$ of degree~$1$.
In particular, $E_1 \isom E_2$.

By the claim, Lemma~\ref{berliozz} holds with $\gamma$ equal to the number 
of possible $\Gamma_1(N)$-structures on an elliptic curve.
\end{proof}

\begin{proof}[Proof of Corollary \ref{bizzet}]
By Lemma~\ref{berliozz}, the map
\[
	\Phi(\Pi^{-1}(C)) \cap \Gamma 
	\to \overline{\Phi(\Pi^{-1}(C)) \cap \Gamma}
\]
has finite fibers 
of cardinality bounded by a constant independent of $\Gamma$. 
On the other hand, by Corollary~\ref{tridd}, 
the target of this map
has cardinality at most $c p^{r}$ for some $c$ independent of~$\Gamma$.
\end{proof}

\begin{proof}[Proof of Theorem~\ref{refined2}]
Assume, in the proof of Theorem~\ref{isoggg}, that we have
a modular-elliptic correspondence arising from a modular parametrization
attached to $f$. 
Part~(1) follows from Lemma~\ref{shadap}.
Part~(2) follows comparing Fourier expansions of the two sides:
apply the substitution maps as in~\eqref{acompo}
to $\overline{f^{\sharp}_{\infty}}$ given in~\eqref{titi} 
to obtain the $p$ different series 
\[
	\overline{\Sigmatt}_{\infty i} = (\overline{f^{\sharp}_{\infty}})_* 
	= \overline{\Phi^{\dug}}(q)+\overline{\lambda}_i \overline{\Phi^{\dug \dug}}(q) \in k((q))
\]
where $\lambda_1,\ldots,\lambda_p \in k$ 
are the zeros of $x^p-\bar{\lambda} x$ 
as in the proof of Lemma~\ref{coppies}.
\end{proof}

\section{Proofs of global results, II}
\label{S:global 2}

In this section we prove the rest of our global results; we will
derive them from the corresponding local results.

We begin a lemma to be used in the proof of Theorem~\ref{glver}.
If $E$ is an abelian variety over a field $L$
and $n \in \Z_{>0}$, let $\overline{L}$ be an algebraic closure of $L$
and let 
$E[n] := \ker \left( E(\overline{L}) \stackrel{n}\to E(\overline{L}) \right)$.

\begin{lemma}
\label{anom} 
Let $L$ be a  number field.
Let $A$ be an elliptic curve over $L$.
Let $E$ be an elliptic curve or abelian surface over $L$. 
Then there exist infinitely many rational primes $p$ for which
\begin{enumerate}
\item $p$ splits completely in $L$,
\item $A$ has good non-anomalous reduction at any prime above $p$, and
\item $E$ has good ordinary reduction at any prime above $p$.
\end{enumerate}
\end{lemma}

\begin{proof}
Recall first the following general criterion for a $g$-dimensional abelian
variety $E$ over $\F_p$ to be ordinary. 
Let $P(x)$ be the characteristic polynomial
of $\Frob_p$ acting on a Tate module. So $P(x)$ is a monic degree
$2g$ polynomial in $\Z[x]$. Then $E$ is ordinary if and only if
exactly $g$ of the $2g$ zeros of $P(x)$ in $\overline{\bQ}_p$  are
$p$-adic units, or equivalently (via the theory of Newton
polygons) if and only if the middle coefficient (coefficient of
$x^g$) in $P(x)$ is not divisible by $p$.

We may enlarge $L$ to assume that $A[13] \subset A(L)$,
that $E[13] \subset E(L)$, and that $L$ is Galois over $\Q$.
The Chebotarev density theorem gives infinitely many $p$
splitting completely in $L$
such that $A$ and $E$ have good reduction at the primes above $p$.

We claim that any such $p$ satisfies all three properties.
First, by the Weil pairing, 
$L$ contains a primitive $13^{\operatorname{th}}$ root of~$1$, 
so $p \equiv 1 \pmod{13}$.

Suppose that $\dim E = 2$.
Fix a place $v$ above $p$.
Let $P(x)$ the characteristic polynomial of $\Frob_v$ acting on
the $13$-adic Tate module $T_{13}(E)$.
By the Weil conjectures,
the $4$ zeros of $P(x)$ have complex absolute value $\sqrt{p}$, 
so the absolute value of the middle coefficient $b$ of $P(x)$ 
is at most $6p$. Equality would mean
that all the zeros were $\sqrt{p}$ or all the zeros were $-\sqrt{p}$, 
but this is impossible, since $P(x) \in \Z[x]$. 
If $E$ were not ordinary at $v$, then we would have 
$b = cp$ for some $c \in \Z$, and $|b| < 6p$ would imply $|c|<6$;
then $b$ is congruent to one of $-5,-4,\ldots ,4,5$ mod $13$. 
On the other hand, since $p$ splits completely in $L = L(E[13])$, 
we have $P(x) \equiv (x-1)^4 \pmod{13}$, 
so $b \equiv 6 \pmod{13}$. 
This contradiction shows that $v$ is an ordinary prime for $E$. 

The case where $E$ is an elliptic curve can be treated similarly,
and is even easier.
Moreover, by a similar argument, the trace of $\Frob_v$ on $T_{13}(A)$
is $2 \bmod 13$, and in particular is not $1$, so $v$ is not anomalous
for $A$.
\end{proof}

\begin{remark}
\hfill
\begin{enumerate}
\item
The first to prove that an abelian surface over a number field
has infinitely many ordinary primes was A.~Ogus~\cite[Corollary~2.9]{ogus}.
\item
The method used to prove~\cite[Theorem~3.6.4]{SerreTopics} 
for non-CM elliptic curves over $\Q$ can be generalized 
to show that the set of anomalous primes of an elliptic curve
over a number field $L$ has density~$0$.
\end{enumerate}
\end{remark}

\begin{proof}[Proof of Theorem~\ref{glver}]
The point $Q$ corresponds to a (possibly false) elliptic curve
$E$ over a number field $L$.
We may assume that $L \supset F_0$.
If $E$ is $\CM$, we assume also that $L \supset \calK_Q$,
where $\calK_Q$ is the imaginary quadratic field $(\End E) \tensor \Q$
(or $\End(E,i) \tensor \Q$ in the Shimura case).
Let $\Delta$ be the infinite set of rational primes 
given by Lemma~\ref{anom}.
Removing finitely many elements from $\Delta$, we may assume 
for any $p \in \Delta$ that $Q \in S(\Z_p^{\ur})$,
that $\Pi$ reduces modulo primes above $v$ to a separable morphism,
and that Corollary~\ref{bizzet} holds for $p$;
if in addition $E$ is CM, we may assume that $Q \in \CL$,
and Corollary~\ref{coru} holds for $p$.) 

Fix $p \in \Delta$.
Define $R$ and so on as in Section~\ref{S:local}.
Then $Q$ is an ordinary point of $S(R)$. 
If $E$ is not CM, let $\Sigma$ be the set of all rational primes 
that are inert in $\cK_Q$; 
if $E$ is CM, take $\Sigma:=\{l \mid l \neq p\}$.
Let $C$ be the $\Sigma$-isogeny class of $Q$ in $S(\Qbar)$;
if $Q \in \CM$, then $C \subset \CL$ by Theorem~\ref{T:revv2}(3).
Let $\bar{S}_\ram \subset \bar{S} := S \tensor k$ 
be the branch locus of $\bar{\Pi} \colon \bar{X} \to \bar{S}$.
Write $C$ as the disjoint union of $C_\et$ and $C_\ram$,
where $C_\ram$ is the set of points of $C$ whose reduction
lies in $\bar{S}_\ram$.

Since $\bar{S}_\ram$ is finite, Lemma~\ref{berliozz}
implies that $C_\ram$ is finite of cardinality
bounded independently of $\Gamma$,
so the same is true of $\Phi(\Pi^{-1}(C_\ram))$.

On the other hand, $\Pi^{-1}(C_\et) = \Pi_R^{-1}(C_\et) \subset X(R)$
where $\Pi_R$ is the set map $X(R) \to S(R)$,
so
\[
	\Phi(\Pi^{-1}(C_\et)) \intersect \Gamma
	\subset
	\Phi(\Pi_R^{-1}(C)) \intersect \Gamma.
\]
By Corollary~\ref{bizzet} (or Corollary~\ref{coru} if $E$ is CM)
the set
$\Phi(\Pi_R^{-1}(C)) \intersect \Gamma$ is finite of
cardinality at most $c p^r$,
where $r=\rank_p^{A(R)}(\Gamma) =\rank_p^{A(M)}(\Gamma)$. 
Thus $\# \Phi(\Pi^{-1}(C_\et)) \intersect \Gamma \le c p^r$ too.

Combining the previous two paragraphs gives the desired bound
on $\Phi(\Pi^{1}(C)) \intersect \Gamma$.
\end{proof}

We use the following to prove Theorem~\ref{incauna}:

\begin{lemma}
\label{focati} 
Let $E$ be a non-$\CM$ elliptic curve over a finite Galois extension
$L$ of $\Q$.
Then there exists a constant $l_0$ such that for any finite set
$\Sigma$ of rational primes greater than $l_0$ 
there is an infinite set $\Delta_L$ of primes $v$ of $L$
such that
\begin{enumerate}
\item $E$ has good ordinary reduction $\overline{E}_v$ at $v$;
\item each $l \in \Sigma$ is inert in the imaginary quadratic field
$\End(\overline{E}_v) \otimes \bQ$; and
\item $v$ has degree $1$ (so the rational prime beneath it 
splits completely in $L$).
\end{enumerate}
\end{lemma}

\begin{proof}
Let $G_L$ be the absolute Galois group of $L$.
For a set of primes $\Sigma$, 
let
$\rho_\Sigma \colon G_L \to \prod_{l \in \Sigma} \GL_2(\F_l)$
be assembled from the homomorphisms 
$\rho_l \colon G_L \to \GL_2(\F_l)$ giving the Galois action on $E[l]$.
By~\cite{dordre}, one can find $l_0$ such that $\rho_{\Sigma}$ is surjective
for any finite set $\Sigma$ of rational primes greater than $l_0$.

Fix a finite set $\Sigma$ of primes greater than $l_0$.
For each $l$, choose $B_l \in \GL_2(\F_l)$ with irreducible characteristic
polynomial.
Add the identity matrix to $B_l$ if necessary to make $\tr(B_l)\ne 0$.
By the Chebotarev density theorem, 
the set $\Delta_1$ of primes $v$ of $L$
such that $\rho_l(\Frob_v)$ is conjugate to $B_l$ for all $l$
has positive Dirichlet density.
The set of primes of $L$ of degree greater than $1$ has density $0$,
so the subset $\Delta_L \subset \Delta_1$ obtained by excluding 
these and the primes of bad reduction for $E$ is still infinite.

Fix $v \in \Delta_L$ and $l \in \Sigma$.
We now check that $v$ satisfies the three desired conditions; 
(3)~holds by definition.
Since $\tr(\rho_l(\Frob_v)) = \tr(B_l) \ne 0$,
we have~(1).
Let $P(x) \in \Z[x]$ be the characteristic polynomial of Frobenius
acting on a Tate module of $\bar{E}_v$.
By choice of $B_l$, 
the quadratic polynomial $(P(x) \bmod l) \in \F_l[x]$ 
is irreducible.
Thus $(\End \bar{E}_v) \tensor \Q$ is the quadratic
field defined by $P(x)$, and $l$ is inert in this field;
this proves~(2).
\end{proof}

\begin{proof}[Proof of Theorem~\ref{incauna}]
Choose a number field $L$
such that $Q \in S(L)$.
We may assume that $L \supset F_0$ and that $L$ is Galois over $\Q$.
Let $E$ be the corresponding elliptic curve over $L$.

Suppose first that $E$ is not CM. 
Let $l_0$ be as in Lemma~\ref{focati},
and let $\Sigma$ be a finite set of rational primes greater than $l_0$. 
Let $\Delta_L$ be as in Lemma~\ref{focati}.
Let $\Delta$ be the set of rational primes divisible by primes 
in $\Delta_L$.
Each $p \in \Delta$ splits completely in $L$ and hence in $F_0$.
For each $p \in \Delta$ we choose $M \subset \overline{\bQ}$ 
maximally unramified at some prime $\wp$ of $M$
such that $\wp$ lies above  some $v \in \Delta_L$ over $p$.
We have that $Q \in S(\cO_{M,\wp})$ is ordinary, 
and each $l\in \Sigma$ is inert in $\cK_Q$.
Remove finitely many primes from $\Delta$ 
so that Corollary~\ref{bizzet} holds for $p$. From
now on, proceed as in the proof of Theorem~\ref{glver}.

Finally suppose instead that 
$E$ has CM by an imaginary quadratic field $\calK$.
Let $\Sigma$ be any finite set of rational primes. 
There exist infinitely many rational primes 
that split completely in $\cK F_0$. 
So there is an infinite set $\Delta$ of rational primes, disjoint from
$\Sigma$, such that for any $p \in \Delta$, 
we have that $p$ splits in $\cK F_0$,
that $Q \in S(\cO_{M,\wp})$ 
(with $M$ maximally unramified at a $\wp$ over $p$),
that $Q \in \CL$ (by Theorem~\ref{T:revv3}(1)(b)), 
and that Corollary~\ref{coru} holds for $p$.
Again we may proceed as in the proof of Theorem~\ref{glver}.
\end{proof}

\begin{proof}[Proof of Theorem~\ref{simairefined}]
 Let $P_i$ be represented by data
$(E_i,\alpha_i)$, and let $\omega_i$ be any $1$-form on $E_i$
defined over $\Qbar$. 
Since $p$ is sufficiently large, 
the triple $(E_i,\alpha_i,\omega_i)$ is definable over $\cO_{M,\wp}$ 
and hence may be considered over $R$.
Since $p$ splits completely in $\cK_i$ and is sufficiently large, 
$E_i \in \CL$ by Theorem~\ref{T:revv3}(1)(b).
Also since $p$ is sufficiently large, we may assume
$\bZ[\zeta_N,1/N] \subset \cO_{M,\wp}$:
we want this because of Remark~\ref{R:zeta_N}.
In particular,
$f_{p^2-p}(E_i,\alpha_i,\omega_i) \in \cO_{M,\wp}$
and
$E_{p-1}(E_i,\alpha_i,\omega_i) \in (\cO_{M,\wp})^{\times}$. 
With $\bar{f}^{(-1)}$ as in Remark~\ref{irigutza}, we have
\[
\bar{f}^{(-1)}(\bar{P}_i)   =  \frac{(\theta^{p-2}
\bar{f})(\bar{E}_i,\bar{\alpha}_i,\bar{\omega}_i)}{\bar{E}_{p-1}^p
(\bar{E}_i,\bar{\alpha}_i,\bar{\omega}_i)}=\overline{ \left(
\frac{f_{p^2-p}(E_i,\alpha_i,\omega_i)}{E_{p-1}^p(E_i,
\alpha_i,\omega_i)}\right)}= \overline{\Sigmat_p(P_i)}:
\]
the hypothesis $\calK_i \ne \Q(\sqrt{-1}),\Q(\sqrt{-3})$
guarantees that $\bar{P}_i \in \bar{X}^{\dug}$.
Now apply Theorem~\ref{mainth} with Theorem~\ref{refined}(2).
\end{proof}

\bibliographystyle{amsplain}

\begin{thebibliography}{10}

\bibitem{Barcau} Barcau, M.:Isogeny covariant differential
modular forms and the space of elliptic curves up to isogeny,
Compositio Math.\ {\bf 137} (2003), 237--273.

\bibitem{four} Breuil C., Conrad B., Diamond F., Taylor R.:
On the modularity of elliptic curves over $\bQ$: wild $3$-adic
exercises, J.\ Amer.\ Math.\ Soc.\ {\bf 14} (2001), no.~4, 843--939.

\bibitem{char} Buium, A.:
Differential characters of abelian varieties over $p$-adic fields, 
Invent.\ Math.\ {\bf 122} (1995), 309--340.

\bibitem{frob} Buium, A.:
Differential characters and characteristic polynomial of Frobenius,
J.\ reine angew.\ Math.\ {\bf 485} (1997), 209--219.

\bibitem{difmod} Buium, A.: Differential modular forms,
J.\ reine angew.\ Math.,\ {\bf 520} (2000), 95--167.

\bibitem{shimura}  Buium, A.: Differential modular forms on
Shimura curves, I, Compositio Math.\ {\bf 139} (2003), 197--237.

\bibitem{book} Buium, A.:
Arithmetic Differential Equations. Math. Surveys and Monographs
{\bf 118}, AMS (2005).

\bibitem{buzzard} Buzzard, K.: Integral models of certain
    Shimura curves, Duke Math J.\ {\bf 87} (1997), no.~3, 591--612.

\bibitem{chai} Chai C-L., A note on Manin's Theorem of the Kernel,
Amer.\ J.\ Math.\ {\bf 113} (1991), no.~3, 387--389.

\bibitem{conrad} Conrad, B.: The Shimura Construction in
weight~$2$. Appendix to: Ribet, K. A., Stein, W.: Lectures on
Serre's conjecture.  In: Arithmetic Algebraic Geometry, Conrad,
B.,  Rubin K., Eds., IAS/Park City Math Series, Vol.~9, AMS (2001).

\bibitem{cornut} Cornut, C.: Mazur's conjecture on higher Heegner
points, Invent.\ Math.\ {\bf 148} (2002), 495--523.

\bibitem{darmon} Darmon, H., Rational
 points on modular elliptic curves.  CBMS No.~101, AMS (2004).

\bibitem{DI}  Diamond, F., and Im, J.:
Modular forms and modular curves. In: Seminar on Fermat's Last
Theorem, Conference Proceedings, Volume 17, Canadian Mathematical
Society, pp.~39--134 (1995).

\bibitem{DS} Diamond, F., Shurman, J.: A first course in
modular forms.  GTM 228, Springer (2005)

\bibitem{Duke1988} Duke, W.:
Hyperbolic distribution problems and half-integral weight Maass forms,
Invent.\ Math.\ {\bf 92} (1988), 73--90.

\bibitem{DO}  Dwork, B.,  Ogus, A.: Canonical liftings of Jacobians,
Compositio Math.\ {\bf 58} (1986), 111--131.

\bibitem{faltings} Faltings, G.: Endlichkeitss\"{a}tze f\"{u}r abelsche
Variet\"{a}ten \"{u}ber Zahl\-k\"{o}r\-pern,
Invent.\ Math.\ {\bf 73} (1983), 349--366.

\bibitem{gross} Gross B., A tameness criterion for Galois
representations associated to modular forms mod $p$, 
Duke Math.\ J.\ {\bf 61} (1990), no.~2, 445--517.

\bibitem{GZ} Gross B., Zagier D.: Heegner points and derivatives
of L-series, Invent.\ Math.\ {\bf 84} (1986), no.~2, 225--320.

\bibitem{GKZ} Gross B., Kohnen W., Zagier D.: Heegner points and
derivatives of L-series II, Math.\ Ann.\ {\bf 278} (1987), nos.~1--4, 497--562.

\bibitem{H}  Hurlburt, C.: Isogeny covariant
differential modular forms modulo p, Compositio Math., {\bf 128} (2001), 
no.~1, 17--34.

\bibitem{Katz} Katz, N.: $p$-adic properties of
modular schemes and modular forms, LNM 350, Springer 1973, 69--190.

\bibitem{Katzcan} Katz, N.:  Serre-Tate local moduli,
Springer LNM 868 (1981), 138--202.

\bibitem{Kollar} Koll\'ar, J: Lectures on resolution of singularities,
Annals of Math.\ Studies {\bf 166}, Princeton University Press, 2007.

\bibitem{lang} Lang, S.:
Introduction to Modular forms.  Springer, Heidelberg (1976)

\bibitem{kolivagin} Kolyvagin, V. A.: Finiteness of $E(\bQ)$ and
$SH(E,\bQ)$ for a subclass of Weil elliptic curves, 
Izv.\ Akad.\ Nauk SSSR Ser.\ Mat.\ {\bf 52} (1988), no.~3, 522--540.

\bibitem{man} Manin, Yu.\ I.: Algebraic curves over fields
with differentiation, Izv.\ Akad.\ Nauk SSSR, Ser.\ Mat.\ \textbf{22},
737--756 (1958)

\bibitem{man1}   Manin Yu.\ I., Rational points of
algebraic curves over function fields, Izv.\ Akad.\ Nauk SSSR Ser.\ Mat.\ 
{\bf 27} (1963), 1395--1440.

\bibitem{mazur} Mazur, B.: Rational points of abelian
varieties with values in towers of number fields, Invent.\ Math.\ {\bf 18}
(1972), 183--266.

\bibitem{mazuricm} Mazur, B.: Modular curves and arithmetic,
Proc.\ ICM, Warsaw, 1983, PWN (1984), 185--211.

\bibitem{messing} Messing, W.: The Crystals Associated to
Barsotti-Tate Groups, LNM 264, Springer 1972.

\bibitem{Michel-Venkatesh2006} Michel, P. and Venkatesh, A.:
Equidistribution, L-functions and ergodic theory:
on some problems of Yu.\ V.\ Linnik,
preprint (2006), available at
{\tt http://cims.nyu.edu/\~{}venkatesh/research/linnik.pdf}

\bibitem{MumfordAV} Mumford, D.: Abelian varieties,
Oxford University Press, 1970.

\bibitem{MunkresTopology} Munkres, J.: Topology: a first course,
Prentice-Hall, 1975.

\bibitem{turk} Nekov\'{a}r, J. and Schappacher, N.: On the
asymptotic behaviour of Heegner points, Turkish J. Math.\ {\bf 23} (1999), 
549--556.

\bibitem{ogus} Ogus, A.: Hodge cycles and crystalline cohomology, 
pp.~357--414. In: Hodge cycles, motives and Shimura varieties,
by P. Deligne, J. Milne, A. Ogus, and K.-y. Shih,
Lecture Notes in Math.\ {\bf 900}, Springer-Verlag, 1982.

\bibitem{Poonen1999} Poonen, B.: Mordell-Lang plus Bogomolov,
Invent.\ Math.\ {\bf 137} (1999), 413--425.

\bibitem{RS} Rosen, M., and Silverman, J. H.: On the independence of
Heegner points associated to distinct imaginary fields,
arXiv.math.NT/0508259v2, 15 August 2005.

\bibitem{SerreComplexMultiplication} Serre, J.-P.: Complex multiplication,
pp.~292--296 in 
Algebraic Number Theory (Proc.\ Instructional Conf., Brighton, 1965), 
Thompson, 1967.

\bibitem{dordre} Serre, J.-P.: Propri\'{e}t\'{e}s galoisennes des
point d'ordre fini des courbes elliptiques, Invent.\ Math.\  {\bf 15}
(1972), 259--331.

\bibitem{serre} Serre, J.-P.: Formes modulaires et fonctions
z\'{e}ta $p$-adiques. In: Springer Lecture Notes in Math.\ \textbf{350} (1973).

\bibitem{GACC} Serre, J.-P.: Algebraic groups and class fields,
GMT 117,  Springer, Heidelberg, New York, 1988.

\bibitem{SerreTopics} Serre, J.-P.: Topics in Galois theory,
Jones and Bartlett, Boston, 1992.

\bibitem{Shimura1971} Shimura, G.: Introduction to the arithmetic
theory of automorphic functions, Princeton Univ.\ Press, 1971.

\bibitem{silwief} Silverman, J. H., Wieferich criterion and the abc
conjecture, J. Number Theory {\bf 30} (1988), 226--237.

\bibitem{TW} Taylor R., Wiles A.: Ring-theoretic properties of
certain Hecke algebras, Annals of Math.~(2) {\bf 141} (1995), no.~3,
553--572.

\bibitem{vatsal} Vatsal V.: Uniform distribution of Heegner
points, Invent.\ Math.\ {\bf 148} (2002), 1--48.

\bibitem{volwief} J.F.Voloch,
Elliptic Wieferich primes, J. Number Theory {\bf 81} (2000), no.~2, 205-209.

\bibitem{wiles} Wiles, A.: Modular elliptic curves and Fermat's
Last Theorem, Annals of Math.~(2) {\bf 141} (1995), no.~3, 443--551.

\bibitem{Zhang2000} Zhang, S.: Distribution of almost division points,
Duke Math.\ J. {\bf 103} (2000), 39--46.

\bibitem{ZZ} Zhang, S.: Heights of Heegner points on Shimura
curves, Annals of Math.~(2) {\bf 153} (2001), 27--147.

\bibitem{Zhang2005} Zhang, S.: Equidistribution of CM-points on
quaternion Shimura varieties,
Int.\ Math.\ Res.\ Not.\ {\bf 2005}, no.~59, 3657--3689.

\end{thebibliography}

\end{document}